\documentclass[reqno,10pt,a4paper]{amsart}
\usepackage[utf8]{inputenc}
\usepackage{amssymb,eucal,mathrsfs,array,setspace,geometry,enumitem,cite,tensor,amsmath,wrapfig,amscd,mathptmx,graphicx,bm,bbm,multirow,multicol}
\usepackage[dvipsnames]{xcolor}
\usepackage{tikz}
\usetikzlibrary{calc}
\usepackage[centertableaux]{ytableau}

\usepackage{txfonts}            
\usepackage{newtxtext}            
\usepackage{mathtools} 
\usepackage{stmaryrd} 

\usetikzlibrary{positioning}

\usepackage[colorlinks=true,citecolor=red,linkcolor=blue]{hyperref} 
\usepackage[capitalise,noabbrev]{cleveref}

\DeclareSymbolFont{largesymbols}{OMX}{zplm}{m}{n} 

\setitemize{leftmargin=*}   
\setenumerate{leftmargin=*, 
              label=\textup{(\arabic*)}} 

\let\originalleft\left     
\let\originalright\right
\renewcommand{\left}{\mathopen{}\mathclose\bgroup\originalleft}
\renewcommand{\right}{\aftergroup\egroup\originalright}

\newcolumntype{C}{>{$}c<{$}} 

\numberwithin{equation}{section}
\allowdisplaybreaks

\renewcommand{\Re}{\operatorname{Re}}

\newcommand{\img}{\operatorname{im}}

\renewcommand{\ge}{\geq}
\renewcommand{\le}{\leq}

\DeclarePairedDelimiter{\brac}{\lparen}{\rparen} 
\DeclarePairedDelimiter{\sqbrac}{\lbrack}{\rbrack} 
\DeclarePairedDelimiter{\set}{\lbrace}{\rbrace}
\newcommand{\st}{\mspace{5mu} : \mspace{5mu}} 
\DeclarePairedDelimiter{\abs}{\lvert}{\rvert}

\DeclarePairedDelimiter{\ang}{\langle}{\rangle}
\DeclarePairedDelimiter{\normord}{{} :}{: {}} 
\DeclarePairedDelimiter{\powser}{\llbracket}{\rrbracket} 

\DeclarePairedDelimiterX{\comm}[2]{\lbrack}{\rbrack}{#1 , #2}  
\DeclarePairedDelimiterX{\acomm}[2]{\lbrace}{\rbrace}{#1 , #2} 
\DeclarePairedDelimiterX{\super}[2]{\lparen}{\rparen}{#1 \delimsize\vert \mathopen{} #2} 

\newcommand{\dpair}[2]{\ang{#1,#2}} 
\newcommand{\pair}[2]{\brac*{#1,#2}} 

\newcommand{\dd}{\mathrm{d}}   
\newcommand{\ee}{\mathsf{e}}   

\newcommand{\wun}{\mathbf{1}}  

\newcommand{\vac}{\Omega}

\DeclareMathOperator{\Res}{Res}
\DeclareMathOperator{\cspn}{span}
\newcommand{\spn}[1]{\cspn_{\CC}\set*{#1}}                    
\newcommand{\zspn}[1]{\cspn_{\ZZ}\set*{#1}}
\newcommand{\rspn}[1]{\cspn_{\RR}\set*{#1}}
\newcommand{\rgen}[1]{\cspn_{R}\set*{#1}}

\newcommand{\ra}{\rightarrow}
\newcommand{\lra}{\longrightarrow}

\newcommand{\dses}[5]{0 \lra #1 \overset{#2}{\lra} #3 \overset{#4}{\lra} #5 \lra 0} 

\DeclareMathOperator{\ind}{Ind}
\newcommand{\Ind}[3]{\ind^{#1}_{#2} #3}
\DeclareMathOperator{\ext}{Ext}
\newcommand{\Extgrp}[3]{\ext^{#1}\brac*{#2,#3}}
\DeclareMathOperator{\Hom}{Hom}
\newcommand{\Homgrp}[2]{\Hom\brac*{#1,#2}}
\newcommand{\dExt}[2]{\dim\Extgrp{}{#1}{#2}}
\newcommand{\dHom}[2]{\dim\Homgrp{#1}{#2}}

\newcommand{\fld}[1]{\mathbb{#1}}    
\newcommand{\alg}[1]{\mathfrak{#1}}  
\newcommand{\Mod}[1]{\mathcal{#1}}   
\newcommand{\VOA}[1]{\mathsf{#1}}    
\newcommand{\categ}[1]{\mathscr{#1}} 

\newcommand{\ZZ}{\fld{Z}}
\newcommand{\NN}{\fld{N}}

\newcommand{\RR}{\fld{R}}
\newcommand{\CC}{\fld{C}}

\newcommand{\UEA}[1]{\mathsf{U}\brac*{#1}}

\newcommand{\bglie}{\alg{G}} 
\newcommand{\weyl}{\alg{A}}

\newcommand{\latva}{\VOA{V}_{K}}                                  
\newcommand{\bgva}{\VOA{G}}

\newcommand{\hva}{\VOA{H}}

\DeclareMathOperator{\opp}{opp}

\DeclareMathOperator{\wt}{wt}                                               

\newcommand{\finite}[1]{\overline{#1}}
\newcommand{\func}[2]{#1 #2}                         
\newcommand{\conjaut}{\mathsf{c}}                    
\newcommand{\sfaut}{\sigma}                          
\newcommand{\conjmod}[1]{\func{\conjaut}{#1}}        
\newcommand{\sfmod}[2]{\func{\sfaut^{#1}}{#2}}       

\newcommand{\catR}{\categ{R}}           
\newcommand{\catFLE}{\categ{F}}         

\newcommand{\Fock}[1]{\Mod{F}_{#1}}          
\newcommand{\LFock}[1]{\mathbb{F}_{#1}}      

\DeclarePairedDelimiter{\bra}{\langle}{\rvert}
\DeclarePairedDelimiter{\ket}{\lvert}{\rangle}
\DeclarePairedDelimiterX{\braket}[2]{\langle}{\rangle}{#1 \delimsize\vert \mathopen{} #2}
\DeclarePairedDelimiterX{\bracket}[3]{\langle}{\rangle}{#1 \delimsize\vert \mathopen{} #2 \delimsize\vert \mathopen{} #3}

\newcommand{\fuse}{\mathbin{\boxtimes}}                                            
\newcommand{\ffuse}{\fuse^{\mathsf{ff}}}                                        
\newcommand{\itype}[3]{\binom{#3}{#1,\ #2}}
\DeclareRobustCommand{\ddfuse}{\text{\reflectbox{$\boxslash$}}}
\DeclareMathOperator{\compspace}{COMP}
\newcommand{\comp}[2]{\compspace\brac*{#1,#2}}
\DeclareMathOperator{\lgrspace}{LGR}
\newcommand{\lgr}[2]{\lgrspace\brac*{#1,#2}}

\newcommand{\vop}[2]{\mathrm{V}_{#1}\brac*{#2}}
\newcommand{\SCR}{\mathcal{S}}
\newcommand{\scr}[1]{\SCR_{#1}}

\newcommand{\hw}{highest-weight}

\newcommand{\hwvs}{\hw{} vectors}

\newcommand{\va}{vertex algebra}

\newcommand{\voa}{vertex operator algebra}
\newcommand{\ope}{operator product expansion}

\newcommand{\PBW}{Poincar\'{e}-Birkhoff-Witt}

\newcommand{\lhs}{left-hand side}
\newcommand{\rhs}{right-hand side}

\theoremstyle{plain}
\newtheorem{thm}{Theorem}[section]
\newtheorem{prop}[thm]{Proposition}
\newtheorem{lem}[thm]{Lemma}
\newtheorem{cor}[thm]{Corollary}
\newtheorem*{thm*}{Theorem}

\theoremstyle{definition} 

\newtheorem{defn}[thm]{Definition}
\newtheorem*{rmk}{Remark}

\Crefname{thm}{Theorem}{Theorems}
\Crefname{prop}{Proposition}{Propositions}
\Crefname{lem}{Lemma}{Lemmas}
\Crefname{cor}{Corollary}{Corollaries}
\Crefname{defn}{Definition}{Definitions}
\Crefname{tab}{Table}{Tables}

\newcommand*\tFo[4]{\tensor[_2]{F}{_1}\brac*{#1, #2; #3; #4}}
\newcommand*\Gam[1]{\Gamma\brac*{#1}}
\newcommand{\bfn}[2]{B\brac*{#1,#2}}

\Crefname{ass}{Assumption}{Assumptions}

\newcommand{\module}[1]{\mathcal{#1}}
\newcommand{\Bmod}[2]{\module{B}^{#1}_{#2}}
\newcommand{\Tmod}[2]{\module{T}^{#1}_{#2}}
\newcommand{\Vmod}[1]{\VacMod_{#1}}
\newcommand{\Stmod}[1]{\Stag_{#1}}
\newcommand{\Stsum}[2]{\module{Q}^{#1}_{#2}} 
\newcommand{\Amod}{\module{A}}
\newcommand{\Mmod}{\module{M}}
\newcommand{\Nmod}{\module{N}}
\newcommand{\Wmod}{\module{W}}

\newcommand{\VacMod}{\module{V}}          
\newcommand{\Typ}[1]{\module{W}_{#1}}     
\newcommand{\Stag}{\module{P}}            

\newcommand{\iop}[2]{\mathcal{Y}\brac{#1,#2}}

\newlength\squareheight
\DeclareMathOperator{\soc}{soc}
\DeclareMathOperator{\hd}{hd}

\newcommand{\Proj}[1]{\mathsf{P}\sqbrac{#1}}
\newcommand{\Inje}[1]{\mathsf{J}\sqbrac{#1}} 
\newcommand{\JmodTL}[1]{\mathsf{J}_{#1}} 
\newcommand{\PmodTL}[1]{\mathsf{P}_{#1}} 
\newcommand{\SmodTL}[1]{\mathsf{S}_{#1}}
\newcommand{\CmodTL}[1]{\mathsf{C}_{#1}}
\newcommand{\ImodTL}[1]{\mathsf{I}_{#1}}
\newcommand{\BmodTL}[2]{\mathsf{B}^{#1}_{#2}}
\newcommand{\TmodTL}[2]{\mathsf{T}^{#1}_{#2}}

\DeclareMathOperator{\coker}{coker}

\begin{document}

\title[]{Bosonic ghostbusting --- The bosonic ghost vertex algebra
  admits a logarithmic module category with rigid fusion}

\author[R Allen]{Robert Allen}

\address[Robert Allen]{
School of Mathematics \\
Cardiff University \\
Cardiff, United Kingdom, CF24 4AG.
}

\email{allenr7@cardiff.ac.uk}

\author[S Wood]{Simon Wood}

\address[Simon Wood]{
School of Mathematics \\
Cardiff University \\
Cardiff, United Kingdom, CF24 4AG.
}

\email{woodsi@cardiff.ac.uk}


\begin{abstract}
The rank 1 bosonic ghost \va{}, also known as the $\beta \gamma$
ghosts, symplectic bosons or Weyl \va{}, is a simple example of a
conformal field theory which is neither rational, nor $C_2$-cofinite. We
identify a module category, denoted category \(\catFLE\), which satisfies three necessary conditions coming
from conformal field theory considerations: closure under restricted duals,
closure under fusion and closure under the action of the modular group on
characters. We prove the second of these conditions, with the other two
already being known. Further, we show that category \(\catFLE\) has
sufficiently many projective and injective modules, give a classification of
all indecomposable modules, show that fusion is rigid and
compute all fusion products. The fusion product formulae turn out to
perfectly match a previously proposed
Verlinde formula, which was computed using a conjectured
generalisation of the usual rational Verlinde formula, called the standard
module formalism. The bosonic ghosts therefore exhibit essentially all of the
rich structure of rational theories despite satisfying none of the standard
rationality assumptions such as \(C_2\)-cofiniteness, the \va{} being
isomorphic to its restricted dual or having a one-dimensional conformal weight
0 space. In particular, to the best of the authors' knowledge this
is the first example of a proof of rigidity for a logarithmic
non-\(C_2\)-cofinite \va{}.
\end{abstract}

\maketitle

\onehalfspacing

\section{Introduction}
\label{sec:intro}

A \va{} is called logarithmic if it admits 
reducible yet indecomposable modules on which the Virasoro \(L_0\) operator
acts non-semisimply, giving rise to logarithmic singularities in the
correlation functions of the associated conformal
field theory.
There is a general consensus within the research community that many
of the structures familiar from rational \va{s} such as
modular tensor categories \cite{HuaVer08} and, in particular, the
Verlinde formula should generalise in some form to the logarithmic case, at least for
sufficiently nice logarithmic \va{s}. To this end,
considerable work has been done on developing non-semisimple or non-finite
generalisations of modular tensor categories
\cite{FuclHo13,FuclCFT17,CrelMod17}. However, progress has been
hindered by a severe lack of examples, making it hard to come up with
the right set of assumptions.

Ghost systems have been used extensively in theoretical physics and quantum
algebra. Their applications include gauge fixing in string theory
\cite{FriCon86}, Wakimoto free field realisations \cite{WakFoc86}, quantum
Hamiltonian reduction \cite{FeiDSred90} and constructing the chiral de Rham 
complex on smooth manifolds \cite{MalChi99}.
Fermionic ghosts at central charge \(c=-2\) in the form of symplectic fermions
have received a lot of attention in the past
\cite{KauLog95,GabInd96,GabBdyBlk07}, due to their even subalgebra being one of the
first known examples of a logarithmic \va{}. 
In particular, they are one of the few known examples of
\(C_2\)-cofinite yet logarithmic \va{s}
\cite{RunSBos14,AdaTrip08,TsuTrip15}. This family also provides the
only known examples of logarithmic \(C_2\)-cofinite \va{s}
with a rigid fusion product \cite{RunSBos14,TsuRig13}.

Here we study the rank 1 bosonic ghosts at central charge \(c=2\). One
of the motivations for studying this algebra is that it is simple
enough to allow many quantities to be computed explicitly, while
simultaneously being distinguished from better understood algebras in a number
of interesting ways. For example, the bosonic ghosts are
not \(C_2\)-cofinite and they were shown to be logarithmic by D. Ridout and the second
author in \cite{RidBos14}, in which the module category to be studied here,
denoted category \(\catFLE\), was introduced. The main goals of \cite{RidBos14}
were determining the modular properties of characters in category \(\catFLE\)
and computing the Verlinde formula, using the standard module formalism pioneered
by D. Ridout and T. Creutzig \cite{RidSL208,CreMod12,CreWZW13}, to
predict fusion product formulae. Later, D. Adamovi{\'c} and V. Pedi{\'c} computed the
dimensions of spaces of intertwining operators among the simple modules of
category \(\catFLE\) in \cite{AdaBG19}, which turned out to match the predictions made by the
Verlinde formula in \cite{RidBos14}. Here we show that fusion (in the sense of
the \(P(w)\)-tensor products of \cite{HuaLog}) equips 
category \(\catFLE\) with the structure of a braided tensor
category. This, in particular, implies that category \(\catFLE\) is closed under fusion, that is, the fusion
product of any two objects in \(\catFLE\) has no contributions from
outside \(\catFLE\) and is hence again an object in \(\catFLE\). We
derive explicit formulae for the decomposition of any fusion product
into indecomposable direct summands, and we show that fusion is rigid and
matches the Verlinde formula of \cite{RidBos14}.

A further source of interest for the bosonic ghosts is an exciting recent
correspondence between four-dimensional super conformal field theory
and two-dimensional conformal field theory \cite{Beem4D2D15}, where the bosonic
ghosts appear as one of the smaller examples on the two-dimensional
side. Within this context the bosonic ghosts are the first member of a
family of \va{s} called the \(\mathcal{B}_p\) algebras
\cite{CreuBp14,AugBp19}. The \(\mathcal{B}_p\)-module categories are conjectured to
satisfy interesting tensor categorical equivalences to the module
category of the unrolled restricted quantum groups of
\(\mathfrak{sl}_2\). It will be an interesting future problem to
explore these categorical relations in more detail using the results
of this paper.

The paper is organised as follows. In \cref{sec:ghosts}, we fix notation by giving an introduction to the
bosonic ghost algebra and certain important automorphisms called conjugation
and spectral flow; construct category \(\catFLE\), the module category to be studied; and give two
free field realisations of the bosonic ghost algebra.
In \cref{sec:proj} we begin the analysis of category \(\catFLE\) as an abelian
category by using the free field realisations of the bosonic ghost algebra to
construct a logarithmic module, denoted \(\Stag\), on which the operator \(L_0\) has rank 2 Jordan blocks. We
further show that \(\Stag\) is both an injective hull and a projective cover
of the vacuum module (the bosonic ghost \va{} as a module over
  itself), and classify all projective modules in category \(\catFLE\),
thereby showing that category \(\catFLE\) has sufficiently many projectives
and injectives.
In \cref{sec:indecomp} we complete the analysis of category \(\catFLE\) as an
abelian category by classifying all indecomposable modules. 
In \cref{sec:rigid} we show that fusion equips category \(\catFLE\) with the
structure of a vertex tensor category, the main obstruction being showing that
certain conditions, sufficient for the existence of associativity isomorphisms, hold.
We further show that the simple projective modules of \(\catFLE\) are rigid. 
In \cref{sec:fusion} we show that category \(\catFLE\) is rigid and
determine direct sum decompositions for all fusion products of modules in
category \(\catFLE\).
In \cref{sec:suffconvext} we review an argument by Yang
\cite{Yang18}, which provides sufficient conditions for a technical
property, called convergence and extension, required for the existence of associativity
isomorphisms. We adjust the argument of Yang slightly to remove certain assumptions
on module categories. This adjusted argument proves \cref{thm:gradconext},
which should also prove useful for the generalisations of category \(\catFLE\) to
other \va{s} such as those constructed from affine Lie algebras at
admissible levels.

\subsection*{Acknowledgements}
\label{sec:acknowledgements}

SW would like to thank Yi-Zhi Huang, Shashank Kanade, Robert McRae and Jinwei
Yang for helpful and stimulating discussions regarding the subtleties of
vertex tensor categories, and Ehud Meir for the same regarding finitely
generated modules. Both authors are very grateful to Thomas Creutzig for
drawing their attention to a mistake relating to convergence and extension
properties in a previous version of this manuscript. RA's research is
supported by the EPSRC Doctoral Training Partnership grant EP/R513003/1. 
SW's research is supported by the Australian Research Council Discovery
Project DP160101520. 

\section{Bosonic ghost \va{}}
\label{sec:ghosts}

In this section we introduce the bosonic ghost \va{}, along with its
gradings and automorphisms. We define the module category which will be the
focus of this paper. We also introduce useful tools for the classification of
modules and calculation of fusion products, including two free field
realisations. Note that we will make specific choices of conformal
  structure for all \va{s} considered in this paper and so will not
  distinguish between \va{s}, \voa{s} and conformal \va{s} or other similar
  naming conventions.

\subsection{The algebra and its automorphisms}\label{sec:prelim}

The bosonic ghost \va{} (also called \(\beta\gamma\) ghosts) is closely related to the Weyl
algebra. Their defining relations resemble each other and the Zhu algebra of the bosonic ghosts is isomorphic to the
Weyl algebra. The bosonic ghosts are therefore also often referred to
as the Weyl \va{}. Due to these connections,
we first introduce the Weyl algebra and its modules before
going on to consider the bosonic ghosts.

\begin{defn}
  The (rank 1) \emph{Weyl} algebra \(\weyl\) is the unique unital associative
  algebra with two generators \(p,q,\) subject to the relations
  \begin{equation}
    [p,q]=1,
  \end{equation}
  and no additional relations beyond those required by the axioms of
  an associative algebra. The grading operator is the element \(N=qp\).
\end{defn}

\begin{defn}
  \label{def:weylmods}
  We define the following indecomposable \(\weyl\)-modules:
  \begin{enumerate}
  \item \(\CC[x]\), where \(p\) acts as \(\partial/\partial x\) and
    \(q\) acts as \(x\). Denote this module by \(\finite{\VacMod}\).
    \label{itm:finhw}
  \item \(\CC[x]\), where \(p\) acts as \(x\) and \(q\) acts as
    \(-\partial/\partial x\). Denote this module by
    \(\conjmod{\finite{\VacMod}}\).
    \label{itm:finlw}
  \item \(\CC[x,x^{-1}]x^{\lambda}\), \(\lambda\in \CC\setminus\ZZ\), where \(p\) acts as \(\partial/\partial x\) and
    \(q\) acts as \(x\). Note that shifting \(\lambda\) by an integer
    yields an isomorphic module. Denote the mutually inequivalent
    isomorphism classes of these modules by \(\finite{\Typ{\mu}}\), where
    \(\mu\in \CC/\ZZ\), \(\mu\neq \ZZ\) and \(\lambda\in \mu\).
    \label{itm:dense}
  \item \(\CC[x,x^{-1}]\), where \(p\) acts as \(\partial/\partial x\) and
    \(q\) acts as \(x\). Denote this module by \(\finite{\Typ{0}^+}\). This
    module is uniquely characterised by the non-split exact sequence
    \begin{equation}
      \dses{\finite{\VacMod}}{}{\finite{\Typ{0}^{+}}}{}{\conjmod{\finite{\VacMod}}}.
    \end{equation}
    \label{itm:indecp}
  \item \(\CC[x,x^{-1}]\), where \(p\) acts as \(x\) and \(q\) acts as
    \(-\partial/\partial x\). Denote this module by
    $\finite{\Typ{0}^-}$. This
    module is uniquely characterised by the non-split exact sequence
    \begin{equation}
      \dses{\conjmod{\finite{\VacMod}}}{}{\finite{\Typ{0}^{-}}}{}{\finite{\VacMod}}.
    \end{equation}
    \label{itm:indecm}
  \end{enumerate}
  A module on which \(N=qp\) acts semisimply is called a weight module.
  Note that \(N\) acts semisimply on all modules above.
\end{defn}

\begin{prop}[Block \cite{BloIrr79}]
  \label{thm:weylmods}
  Any simple \(\weyl\)-module on which \(N\) acts semisimply is
  isomorphic to one of those listed in \cref{def:weylmods}, Parts \ref{itm:finhw} -- \ref{itm:dense}.
\end{prop}

\begin{defn}
  The \emph{bosonic ghost \va{}} $\bgva$ is the unique
  \va{} strongly generated by two fields \(\beta,\gamma\),
    subject to the defining \ope{s}
	\begin{equation}
	\label{eq:bgoperel}
	\gamma(z)\beta(w) 
        \sim \frac{1}{z-w},\qquad \beta(z)\beta(w)\sim\gamma(z)\gamma(w)\sim 0,
      \end{equation}
      and no additional relations beyond those required by \va{}
      axioms.
  \end{defn}
  The bosonic ghost \va{} admits a one-parameter family
      of conformal structures. Here we choose the Virasoro field (or energy
      momentum tensor) to be
      \begin{equation}
        T(z)=-\normord{\beta(z)\partial\gamma(z)},
      \end{equation}
      thus determining the central charge to be \(c=2\) and the conformal
      weights of \(\beta\) and \(\gamma\) to be \(1\) and \(0\), respectively.
The bosonic ghost fields can thus be expanded as formal power series with
  the mode indexing chosen to reflect the conformal weights.
\begin{equation}
\beta(z) = \sum_{n \in \ZZ} \beta_n z^{-n-1}, \qquad 
\gamma(z) = \sum_{n \in \ZZ} \gamma_n z^{-n}.
\end{equation}
The \ope{s} of \(\beta\) and \(\gamma\) fields imply that their modes generate the
\emph{bosonic ghost Lie algebra} $\bglie$ 
satisfying the Lie brackets
\begin{equation}
  \label{eq:bgcommrel}
  \comm{\gamma_m}{\beta_n}=\delta_{m+n,0}\wun,\qquad
  \comm{\beta_m}{\beta_n}=\comm{\gamma_m}{\gamma_n}
    =0,\qquad m,n\in\ZZ,
\end{equation}
where \(\wun\) is central and acts as the
identity on any \(\bgva\)-module, since it corresponds to the identity (or vacuum) field.

Within $\bgva$ there is a rank 1 Heisenberg \va{} generated by the field
\begin{equation}
  J(z) = \normord{\beta(z) \gamma(z)}.
\end{equation}
A quick calculation reveals that \(J\) is a free boson of Lorentzian
signature, not a conformal primary, and that \(J\) defines a grading on
\(\beta\) and \(\gamma\) called ghost weight (or ghost number), that is,
\begin{align}
  J(z)J(w)&\sim \frac{-1}{(z-w)^2},  & T(z)
                                      J(w)&\sim\frac{-1}{(z-w)^3}+\frac{J(w)}{(z-w)^2}+\frac{\partial
                                            J(w)}{z-w}\nonumber,\\
  J(z)\beta(w)&\sim \frac{\beta(w)}{z-w},&
  J(z)\gamma(w)&\sim \frac{-\gamma(w)}{z-w}.
\end{align}
Note that for the distinguished elements \(\beta,\ \gamma,\ J,\) and \(T\)
we suppress the field map symbol \(Y:\bgva \to \bgva\powser{z,z^{-1}}\). For
generic elements \(A\in\bgva\) we will use both \(Y(A,z)\) and \(A(z)\) to
denote the field corresponding to \(A\), depending on what is easier to read
in the given context.

We make frequent use of two automorphisms of $\bglie$. The first is spectral flow, which acts on the $\bglie$ modes as
\begin{equation}
\sfmod{\ell}{\beta_n} = \beta_{n-\ell}, \qquad \sfmod{\ell}{\gamma_n} = \gamma_{n+\ell}, \qquad \sfmod{\ell}{\wun} = \wun .
\end{equation}
The second is conjugation which is given by
\begin{equation}
\conjmod{\beta_n} = \gamma_n, \qquad \conjmod{\gamma_n} = -\beta_n, \qquad \conjmod{\wun} = \wun .
\end{equation}
These automorphisms satisfy the relation 
\begin{equation}\label{eq:dihrel}
\conjaut \sfaut^\ell = \sfaut^{-\ell} \conjaut.
\end{equation}
At the level of fields, these automorphisms act as 
  \begin{align}
       \sfmod{\ell}{\beta(z)} &=\beta(z)z^{-\ell},&
       \sfmod{\ell}{\gamma(z)}&=\gamma(z)z^{\ell},&\sfmod{\ell}{J(z)}&=J(z)+\ell\wun z^{-1},&
        \sfmod{\ell}{T(z)}&=T(z)- \ell J(z)z^{-1} - 
        \tfrac{1}{2} \ell (\ell -1)\wun z^{-2}\nonumber,\\
       \conjmod{\beta(z)} &=\gamma(z),&
       \conjmod{\gamma(z)}&=-\beta(z),&\conjmod{J(z)}&=-J(z)+\wun z^{-1},&
       \conjmod{T(z)}&=T(z)+\partial J(z)+ J(z) z^{-1}.
  \end{align}
The primary utility of the conjugation and spectral flow automorphisms
lies in constructing new modules from known ones through twisting.

\begin{defn}
  Let \(\Mmod\) be a \(\bgva\)-module and \(\alpha\) an
  automorphism. The \(\alpha\)-twisted module \(\alpha \Mmod\) is defined
  to be \(\Mmod\) as a vector space, but with the action of \(\bgva\)
  redefined to be
  \begin{equation}
    A(z)\cdot_\alpha m= \alpha^{-1}(A(z)) m,\qquad A\in \bgva,\ m\in \Mmod,
  \end{equation}
  where the action of \(\bgva\) on the \rhs{} is the original untwisted action of \(\bgva\) on \(\Mmod\).
\end{defn}

\begin{rmk}
  Due to being algebra automorphisms, spectral flow and
  conjugation twists both define exact covariant functors. 
  Further, the respective
  ghost and conformal weights \(\sqbrac{j,h}\) of a vector \(m\) in a
  \(\bgva\)-Module \(\Mmod\) transform as follows under conjugation and
  spectral flow.
  \begin{align}\label{eq:twistwts}
    \sfaut^\ell: \sqbrac*{j, \ h} &\mapsto \sqbrac*{ j - \ell, \  h +
                                    \ell j -
                                    \tfrac{1}{2} \ell 
                                    \brac{\ell +
                                    1} }, \nonumber\\
    \conjaut: \sqbrac*{j, \ h} &\mapsto \sqbrac*{ 1 - j , \  h}.
  \end{align}
 Since \(\conjaut^2 \beta_n=-\beta_n\) and
   \(\conjaut^2 \gamma_n=-\gamma_n\), we have \(\conjaut^2 \Mmod\cong \Mmod\),
   for any \(\bgva\)-module \(\Mmod\). We shall later see that spectral
   flow has infinite order and thus the relations \eqref{eq:dihrel}
   imply that at the level of the module
   category spectral flow and conjugation generate the infinite
   dihedral group.
\end{rmk}

\begin{thm}\label{thm:specflow}
  For any \(\bgva\)-modules $\Mmod$ and $\Nmod$, conjugation and
  spectral flow are compatible with fusion products in the following sense.
  \begin{align}
    \nonumber
     \sfmod{\ell}{\Mmod} \fuse \sfmod{m}{\Nmod} &\cong \sfmod{\ell + m}{\brac*{\Mmod \fuse \Nmod}},   \\
    \conjmod{\Mmod}\fuse \conjmod{\Nmod} &\cong
   \conjmod{{\sfmod{}{\brac*{\Mmod \fuse \Nmod}}}}.
  \end{align}
\end{thm}
The behaviour of spectral flow under fusion was proven for \va{s} with finite dimensional conformal weight spaces in \cite[Proposition 2.4]{LiSup97}. However, the proof does not rely on this fact, and so we can apply the result to $\bgva$-modules, as in \cite[Proposition 3.1]{AdaBG19}.
The behaviour of conjugation under fusion was proven in \cite[Proposition
2.1]{AdaBG19}, where conjugation was denoted by $\sigma$ and spectral flow by
$\rho_{\ell}$. There the automorphism $g$ corresponds to $\sfaut^{-1} \conjaut
= \conjaut \sfaut$ here. 
These formulae mean that the fusion of modules twisted by
spectral flow is determined by the fusion of untwisted modules, a
simplification we shall make frequent use of.

\subsection{Module category}\label{sec:modcat}

Every \(\bgva\)-module is a \(\bglie\)-module, however, the
converse is not true (consider for example the adjoint \(\bglie\)-module). The
category of smooth \(\bglie\)-modules consists of precisely those modules
which are also \(\bgva\)-modules. Such modules are also commonly called
weak
\(\bgva\)-modules and we shall use these terms interchangeably. Unfortunately the category of all smooth
modules is at present too unwieldy to analyse and so we must invariably
consider some subcategory.

In this section we define the module category, which we believe to
be the natural one from the perspective of conformal field theory, because it
is compatible with the following two necessary conformal field theoretic conditions.
\begin{enumerate}
\item Non-degeneracy of \(n\)-point conformal blocks (chiral
  correlation functions) on the sphere.
  \label{itm:nondeg}
\item Well-definedness of conformal blocks at higher genera.
  \label{itm:genus}
\end{enumerate}
Condition \ref{itm:nondeg} can be reduced to the
non-degeneracy of two and three-point conformal
blocks. The non-degeneracy of two-point conformal blocks requires
the module category to be closed under taking restricted duals,
while non-degeneracy of three-point conformal blocks requires the
module category to be closed under fusion (as, for example, constructed by the \(P(w)\)-tensor
product of Huang-Lepowsky-Zhang). Conformal blocks at
higher genera can be constructed from those on the sphere provided
there is a well-defined action of the modular group on
characters. Thus Condition \ref{itm:genus} requires characters to be well-defined, that is,
for all modules to decompose into direct sums of finite
dimensional simultaneous generalised \(J_0\) and 
\(L_0\) eigenspaces. On any simple such module 
both \(L_0\) and \(J_0\) will act semisimply. Further, the action of \(J_0\) is
semisimple on a fusion product if \(J_0\) acts
semisimply on both factors of the product. We can therefore restrict ourselves
to a category of \(J_0\)-semisimple modules without endangering closure under
fusion. We cannot, however, assume that \(L_0\) will act semisimply in general.

The main tool for identifying and classifying \voa{} modules is Zhu's algebra. Sadly Zhu's
algebra is only sensitive to modules containing vectors annihilated
by all positive modes. Any simple such module is a simple module in the
category called \(\catR\) below. We will see that \(\catR\) is closed
under taking restricted duals, however, as can be seen later in
\cref{sec:fusion}, category \(\catR\) is not closed under fusion. Further, it was
shown in \cite{RidBos14} that the action of the modular group does
not close on its characters. Thus a larger category is needed, which will be
denoted \(\catFLE\) below.
It was shown in \cite{RidBos14} that the action of the modular group closes on the
characters of \(\catFLE\) and strong evidence was presented that fusion does as well.
We will see in \cref{sec:fusion} that category \(\catFLE\) is indeed closed
under fusion and that it satisfies numerous other nice properties.

The definition of the module categories mentioned above
requires the following choice of parabolic triangular decomposition of \(\bglie\).
\begin{equation}
  \bglie^{\pm}=\cspn\{\beta_{\pm n}, \gamma_{\pm n}\st n\geq 1\},\qquad
  \bglie^0=\cspn\{\wun,\beta_0,\gamma_0\}.
\end{equation}
This decomposition is parabolic, because \(\bglie^0\) is not abelian
and thus not a choice of Cartan subalgebra. The role of the Cartan
subalgebra will instead be played by \(\cspn\{\wun, J_0\}\), which
is technically a subalgebra of the completion of \(\UEA{\bglie}\)
rather than \(\bglie\).

\begin{defn}\leavevmode{}
  \begin{enumerate}
  \item Let \(\bgva\)-WMod be the category of smooth weight
    \(\bglie\)-modules, that is the category whose objects are all
    smooth (or weak) 
    \(\bgva\)-modules \(\Mmod\) (we follow the conventions of \cite{FreVer01}
    regarding smooth modules) which in addition satisfy
    that \(J_0\) acts semisimply
    and whose arrows are all \(\bglie\)-module homomorphisms.
  \item Let \(\catR\) be the full subcategory of \(\bgva\)-WMod
    consisting of those modules \(\Mmod\in \bgva\)-WMod satisfying
    \begin{itemize}
    \item \(\Mmod\) is finitely generated,
    \item \(\bglie^+\) acts locally nilpotently, that is, for all
      \(m\in \Mmod\), \(\UEA{\bglie^+}m\) is finite dimensional.
    \end{itemize}
  \item Let \(\catFLE\) be the full subcategory of \(\bgva\)-WMod
    consisting all finite length extensions of arbitrary spectral
    flows of modules in \(\catR\) with real \(J_0\) weights.
  \end{enumerate}
\end{defn}
 
The \(\weyl\)-modules of \cref{def:weylmods} induce to 
modules in category \(\catR\).
\begin{defn}
  \label{def:inds}
  Let \(\Mmod\) be a $\weyl$-module, then we induce \(\Mmod\) to a
  \(\bgva\)-module \(\Ind{}{}{\Mmod}\) in \(\catR\) by having \(\bglie^+\) act
  trivially on \(\Mmod\), \(\beta_0\) and \(\gamma_0\) act as \(-p\) and
  \(q\), respectively, and \(\bglie^-\) act freely. We denote
  \begin{enumerate}
  \item \(\VacMod\cong \Ind{}{}{\finite{\VacMod}}\), the vacuum module or bosonic ghost \va{} as a module over
    itself.
    \label{itm:vac}
  \item \(\conjmod{\VacMod}\cong \sfmod{-1}{\VacMod}\cong \Ind{}{}{\conjmod{\finite{\VacMod}}}\), the
    conjugation twist of the vacuum module.
    \label{itm:cvac}
  \item $\Typ{\lambda}\cong \Ind{}{}{\finite{\Typ{\lambda}}}$ with
    $\lambda \in \CC/\ZZ$, $\lambda \neq \ZZ$.
    \label{itm:rel}
  \item $\Typ{0}^\pm\cong\Ind{}{}{\finite{\Typ{0}^\pm}}$.
    \label{itm:indrel}
  \end{enumerate}
\end{defn}
Note that due to the simple
nature of the \(\bglie\) commutation relations \eqref{eq:bgcommrel}
\(\Ind{}{}{\Mmod}\) is simple whenever \(\Mmod\) is, that is, the modules
listed in parts \ref{itm:vac} -- \ref{itm:rel} are simple.

\begin{prop}
    \label{thm:simclass}
    \leavevmode
    \begin{enumerate}
    \item Any simple module in \(\catR\) is isomorphic to one of those
      listed in Parts \ref{itm:vac} -- \ref{itm:rel} of
      \cref{def:inds}.
      \label{itm:simps}
    \item Any simple module in \(\catFLE\) is isomorphic to one of the
      following mutually inequivalent modules.
      \begin{equation}
        \sfmod{\ell}{\VacMod},\qquad \sfmod{\ell}{\Typ{\lambda}},\qquad
        \ell\in \ZZ,\ \lambda\in \RR/\ZZ,\ \lambda\neq \ZZ.
      \end{equation}
    \item The conjugation twists of simple modules in \(\catFLE\) satisfy
      \begin{equation}
        \conjmod{\sfmod{\ell}{\VacMod}}\cong
        \sfmod{-1-\ell}{\VacMod},\qquad
        \conjmod{\sfmod{\ell}{\Typ{\lambda}}}\cong \sfmod{-\ell}{\Typ{-\lambda}},\qquad
        \ell\in \ZZ,\ \lambda\in \RR/\ZZ,\ \lambda\neq \ZZ.
      \end{equation}
    \item The indecomposable modules \(\Typ{0}^{\pm}\) satisfy the non-split exact sequences
      \begin{subequations}\label{eq:Weqseq}
        \begin{align}
          &\dses{\VacMod}{}{\Typ{0}^{+}}{}{\sfmod{-1}{\VacMod}}, \label{eq:Weqseq1} \\
          &\dses{\sfmod{-1}{\VacMod}}{}{\Typ{0}^{-}}{}{{\VacMod}}.
            \label{eq:Weqseq2}
        \end{align}
      \end{subequations}
    \end{enumerate}
  \end{prop}
This proposition was originally given in \cite[Proposition 1]{RidBos14}, however, Part \ref{itm:simps} is an immediate
consequence of Block's classification of simple Weyl modules \cite{BloIrr79}.
We shall show in \cref{thm:VacExt} that, up to spectral flow twists, the indecomposable modules
\(\Typ{0}^\pm\) are the only indecomposable length 2 extensions of
spectral flows of the vacuum module.
In \cref{sec:indecomp} we extend the indecomposable modules \(\Typ{0}^\pm\) to
infinite families of indecomposable modules.

\subsection{Restricted duals}\label{sec:rdual}

As mentioned above, conformal field theories require their representation categories to be
  closed under taking restricted duals. They are also an essential tool for the
  computation of fusion products using the Huang-Lepowsky-Zhang (HLZ) double dual
  construction \cite[Part IV]{HuaLog}, also called the $P(w)$-tensor product, and so we
  record the necessary definitions here.
\begin{defn}
  \label{def:dual}
  Let $\Mmod$ be a weight \(\bgva\)-module. The restricted dual (or contragredient)
  module
  is defined to be
  \begin{equation}
    \Mmod' = \bigoplus_{h,j\in\CC} \Homgrp{\Mmod^{(j)}_{[h]}}{\CC}, \qquad
    \Homgrp{\Mmod^{(j)}_{[h]}}{\CC}=\{m\in \Mmod \st (J_0-j)m=0, \ (L_0-h)^nm=0, \ n\gg 0\},
  \end{equation}
 where the action of \(\bgva\) is characterised by
  \begin{equation}
    \ang{Y(A,z) \psi, m} = \ang{\psi, Y(A,z)^{\opp} m}, \qquad A\in\bgva,\ \psi \in \Mmod',\ m\in \Mmod,
  \end{equation}
  and where $Y(A,z)^{\opp}$ is given by the formula 
  \begin{equation}
    Y(A,z)^{\opp} =Y\brac*{\ee^{z L_1}\brac*{-z^{-2}}^{L_0}A,z^{-1}}.
  \end{equation}
\end{defn}

\begin{prop} \label{thm:duallist}
 The \va{} $\bgva$ and its modules have the following properties.
  \begin{enumerate}
  \item The modes of the generating fields and the Heisenberg field satisfy
    \begin{equation}
      \beta_{n}^{\opp} = -\beta_{-n}, \qquad \gamma_{n}^{\opp} = \gamma_{-n}, \qquad J_{n}^{\opp} = \delta_{n,0} - J_{-n}.
    \end{equation}
    \label{itm:bgopps}
  \item The restricted duals of spectral flows of the indecomposable
    modules in \cref{def:inds} can
    be identified as
    \begin{equation}
      \brac*{\sfmod{\ell}{\VacMod}}' \cong
      \sfmod{-1-\ell}{\VacMod},\qquad
      \brac*{\sfmod{\ell}{\Typ{\lambda}}}' \cong
      \sfmod{-\ell}{\Typ{-\lambda}},\qquad
      \brac*{\sfmod{\ell}{\Typ{0}}^\pm}' \cong \sfmod{-\ell}{\Typ{0}}^\pm.
    \end{equation}
    \label{itm:resdual}
  \item Denote by \({}^\ast\) the composition of twisting by
    \(\conjaut\) and taking the restricted dual, then
    \begin{equation}
      \brac*{\sfmod{\ell}{\VacMod}}^\ast \cong
      \sfmod{\ell}{\VacMod},\qquad
      \brac*{\sfmod{\ell}{\Typ{\lambda}}}^\ast \cong
      \sfmod{\ell}{\Typ{\lambda}},\qquad
      \brac*{\sfmod{\ell}{\Typ{0}}^\pm}^\ast \cong \sfmod{\ell}{\Typ{0}}^\mp.
    \end{equation}
    \label{itm:strdual}
  \item Let \(\Amod,\module{B}\in \catFLE\) and \(\ell\in\ZZ\), then the homomorphism and first
    extension groups satisfy
    \begin{align}
      \Homgrp{\Amod}{\module{B}}&=\Homgrp{\conjmod{\Amod}}{\conjmod{\module{B}}}=
      \Homgrp{\sfmod{\ell}{\Amod}}{\sfmod{\ell}{\module{B}}}=\Homgrp{\module{B}'}{\Amod'}=\Homgrp{\module{B}^\ast}{\Amod^\ast},\nonumber\\
      \Extgrp{}{\Amod}{\module{B}}&=\Extgrp{}{\conjmod{\Amod}}{\conjmod{\module{B}}}=
      \Extgrp{}{\sfmod{\ell}{\Amod}}{\sfmod{\ell}{\module{B}}}=\Extgrp{}{\module{B}'}{\Amod'}=\Extgrp{}{\module{B}^\ast}{\Amod^\ast}.
    \end{align}\label{itm:extgrps}
  \end{enumerate}
\end{prop}

\begin{proof}
Part \ref{itm:bgopps} follows immediately from \cref{def:dual}.

  Part \ref{itm:resdual}: Since \(\sfmod{\ell}{\VacMod}\) is simple,
  \(\brac*{\sfmod{\ell}{\VacMod}}'\) is too, due to taking duals being an invertible
  exact contravariant functor. Further, by the action given in \cref{def:dual}
  it is easy to see that \(\beta_n,\ n\ge\ell+1\) and \(\gamma_m,\ m\ge-\ell\)
  act locally nilpotently and therefore \(\brac*{\sfmod{\ell}{\VacMod}}'\) is an
  object of both \(\sfmod{-\ell}{\catR}\) and
  \(\sfmod{-1-\ell}{\catR}\). Thus, \(\brac*{\sfmod{\ell}{\VacMod}}' \cong
  \sfmod{-1-\ell}{\VacMod}\). 
  
  Similarly, since \(\sfmod{\ell}{\Typ{\lambda}}\) is simple,
  \(\brac*{\sfmod{\ell}{\Typ{\lambda}}}'\) is too. The modes \(\beta_n,\ n\ge\ell+1\) and \(\gamma_m,\ m\ge1-\ell\) act locally nilpotently and therefore \(\brac*{\sfmod{\ell}{\Typ{\lambda}}}'\) is an
  object of \(\sfmod{-\ell}{\catR}\). Further, for \(J_0\) homogeneous \(m\in
  \sfmod{\ell}{\Typ{\lambda}}\) and \(\psi\in
 \brac*{\sfmod{\ell}{\Typ{\lambda}}}'\), consider
  \begin{equation}
    \ang{J_0\psi, m} = \ang{\psi, (1-J_0)m}.
  \end{equation}
  Thus, \(\brac*{\sfmod{\ell}{\Typ{\lambda}}}' \cong \sfmod{-\ell}{\Typ{-\lambda}}.\)

  Finally, the duals of \(\sfmod{\ell}{\Typ{0}}^\pm\)
  follow from that fact that the duality functor is exact and
  contravariant, and by applying it to
  the exact sequences \eqref{eq:Weqseq}.

  Part \ref{itm:strdual} follows from composing the formulae of Part \ref{itm:resdual}
  with the conjugation twist formulae of \cref{thm:simclass}.

  Part \ref{itm:extgrps} follows from \(\conjaut\), \(\sfaut\) and \({}'\)
  being exact invertible functors, the first two covariant and the last contravariant.
\end{proof}

\subsection{Free field realisation}\label{sec:ffr}

We present two embeddings of $\bgva$ into a rank 1 lattice
algebra constructed from a rank 2 Heisenberg \va{}. We refer to
\cite{DoLGVA93} for a comprehensive discussion of Heisenberg and lattice \va{s}.

Let \(\hva\) be the rank 2 Heisenberg \va{} with choice of generating fields
\(\psi,\theta\) normalised such that they satisfy the defining \ope{s}
\begin{equation}
  \psi(z)\psi(w)\sim\frac{1}{(z-w)^2},\qquad
  \theta(z)\theta(w)\sim\frac{-1}{(z-w)^2},\qquad \psi(z)\theta(w)\sim 0.
\end{equation}
By a slight abuse of notation we also use \(\psi\) and \(\theta\) to denote a basis of
a rank 2 lattice \(L_\ZZ=\zspn{\psi,\theta}\) with symmetric bilinear lattice form corresponding to the
above \ope{s}, that is \((\psi,\psi)=-(\theta,\theta)=1\) and
\((\psi,\theta)=0\). Let \(L= \rspn{\psi,\theta}\) be the extension of scalars
of \(L_{\ZZ}\) by \(\RR\). We denote the Fock spaces of \(\hva\) by
\(\Fock{\lambda},\ \lambda\in L\), where the zero mode of a Heisenberg
\va{} field \(a(z),\ a\in L\) acts as scalar multiplication by
\(\pair{a}{\lambda}\). We assign to the highest weight vector \(\ket{\lambda}\) of
\(\Fock{\lambda}\) the vertex operators
\(\vop{\lambda}{z}:\Fock{\mu}\to\Fock{\mu}\powser{z,z^{-1}}z^{\pair{\lambda}{\mu}}\)
given by the expansion
\begin{equation}\label{eq:vopformula}
  \vop{\lambda}{z}=\ee^\lambda z^{\lambda_0}\prod_{m\ge1}\exp\brac*{\frac{\lambda_{-m}}{m}z^m}\exp\brac*{-\frac{\lambda_m}{m}z^{-m}},
\end{equation}
where \(\ee^{\lambda}\in \CC[L]\) is the basis element in the group algebra
of \(L\) corresponding to \(\lambda\in L\) and satisfies the relations
\begin{equation}
\comm{b_n}{\ee^{\lambda}}=\delta_{n,0}\pair{b}{\lambda}\ee^{\lambda},
\qquad
\ee^{\lambda}\ket{\mu}=\ket{\lambda+\mu}.
\end{equation}
Finally, let \(\latva\) be the lattice \va{} extension
of \(\hva\) along the indefinite rank 1 lattice \(K=\zspn{\psi+\theta}\).
The simple modules of \(\latva\) are given by
\begin{equation}
  \LFock{\Lambda}=\bigoplus_{\lambda\in\Lambda}\Fock{\lambda},\qquad
  \Lambda\in L/K\ \text{and}\ \pair{\Lambda}{\psi+\theta}\in\ZZ.
\end{equation}
Note that the pairing \(\pair{\Lambda}{\psi+\theta}\) is well-defined,
since it does not depend in the choice of representative \(\lambda\in \Lambda\).
It will occasionally be convenient to label the lattice modules by a
representative $\lambda \in \Lambda$ rather than the coset itself, that is
$\LFock{\lambda} = \LFock{\Lambda}$. Note also that our notation differs from
conventions common in theoretical physics literature. There, for \(a\in L\),
\(\vop{a}{z}\) would be denoted by \(\normord{\ee^{a(z)}}\)
and \(a(z)\) by \(\partial a(z)\).

\begin{prop}\label{thm:ffr}
  \leavevmode
  \begin{enumerate}
  \item The assignment
    \begin{equation}\label{eq:bosonisation}
      \beta(z)\mapsto \vop{\theta+\psi}{z},\qquad \gamma(z)\mapsto \normord{\psi(z)\vop{-\theta-\psi}{z}}
    \end{equation}
    induces an embedding \(\phi_1:\bgva\to \latva\). Restricting to
    the image of this embedding, \(\latva\)-modules can be identified
    with \(\bgva\)-modules as
    \begin{equation}\label{eq:FFR}
      \LFock{\ell\psi} \cong \sfmod{\ell+1}{\Typ{0}^-}, \qquad
      \LFock{\Lambda} \cong \sfmod{\pair{\Lambda}{\psi+\theta}+1}
      {\Typ{\pair{\Lambda}{\psi}}},\quad \Lambda\in L/K, \
      \pair{\Lambda}{\psi+\theta}\in\ZZ
      \ \text{and}\ \pair{\Lambda}{\psi}\neq\ZZ,
    \end{equation}
    where \(\pair{\Lambda}{\psi}\) is the coset in \(\RR/\ZZ\) formed
    by pairing all representatives of \(\Lambda\) with \(\psi\).
    \label{itm:ffr1}
  \item The assignment
    \begin{equation}\label{eq:bosonisation2}
      \beta(z)\mapsto \normord{\psi(z)\vop{\theta +\psi}{z}} ,\qquad \gamma(z)\mapsto \vop{-\theta-\psi}{z}
    \end{equation}
    induces an embedding \(\phi_2:\bgva\to\latva\). Restricting to
    the image of this embedding, \(\latva\)-modules can be identified
    with \(\bgva\)-modules as
    \begin{equation}
      \LFock{\ell\psi} \cong \sfmod{\ell}{\Typ{0}^+}, \qquad
      \LFock{\Lambda} \cong \sfmod{\pair{\Lambda}{\psi+\theta}}{\Typ{\pair{\Lambda}{\psi}}},\quad \Lambda\in L/K, \
      \pair{\Lambda}{\psi+\theta}\in\ZZ
      \ \text{and}\ \pair{\Lambda}{\psi}\neq\ZZ,
    \end{equation}
    where \(\pair{\Lambda}{\psi}\) is the coset in \(\RR/\ZZ\) formed
    by pairing all representatives of \(\Lambda\) with \(\psi\).
    \label{itm:ffr2}
  \end{enumerate}
\end{prop}
The embeddings are well known and the identifications of \(\latva\)-modules with \(\bgva\)-modules follow by comparing characters and was shown in \cite[Proposition 4.7]{WoOsp19} and \cite[Proposition 4.1]{AdaBG19}.

\begin{thm}
  \leavevmode
  \begin{enumerate}
  \item Let \(\scr{1}=\Res \vop{\psi}{z}\), then
    \(\ker \brac*{\scr{1} :\latva\to \LFock{\psi}}=\img \phi_1\),
    where \(\phi_1: \bgva\to \latva\) is the embedding of
    \cref{thm:ffr}.\ref{itm:ffr1}, that is, \(\scr{1}\) is a screening
    operator for the free field realisation \(\phi_1\) of
    \(\bgva\). Further the sequence
    \begin{equation}\label{eq:felder1}
            {\cdots \overset{\scr{1}}{\lra} \LFock{-\psi}
              \overset{\scr{1}}{\lra} \LFock{0}
              \overset{\scr{1}}{\lra} \LFock{\psi}
              \overset{\scr{1}}{\lra} \cdots} 
    \end{equation}
    is exact and is therefore a Felder complex.
    \label{itm:scr1}
  \item Let \(\scr{2}=\Res \vop{-\psi}{z}\), then
    \(\ker \brac*{\scr{2} :\latva\to \LFock{-\psi}}=\img \phi_2\),
    where \(\phi_2: \bgva\to \latva\) is the embedding of
    \cref{thm:ffr}.\ref{itm:ffr2}, that is, \(\scr{2}\) is a screening
    operator for the free field realisation \(\phi_2\) of
    \(\bgva\). Further the sequence
    \begin{equation}
            {\cdots \overset{\scr{2}}{\lra} \LFock{\psi}
              \overset{\scr{2}}{\lra} \LFock{0}
              \overset{\scr{2}}{\lra} \LFock{-\psi}
              \overset{\scr{2}}{\lra} \cdots} 
    \end{equation}
    is exact and is therefore a Felder complex.
    \label{itm:scr2}
  \end{enumerate}
\end{thm}
\begin{proof}
  We prove part \ref{itm:scr1} only, as part \ref{itm:scr2} follows
  analogously. The \ope{} of \(\vop{\psi}{z}\) with the images of
  \(\beta\) and \(\gamma\) in \(\latva\) are
  \begin{equation}
    \vop{\psi}{z}\beta(w)\sim0,\qquad
    \vop{\psi}{z}\gamma(w)\sim -\frac{\vop{-\theta}{w}}{(z-w)^2},
  \end{equation}
  which are total derivatives in \(z\) implying that \(\scr{1}=\Res
  \vop{\psi}{z}\) is a screening operator and that \(\img
  \phi_1\subset \ker \scr{1}\). Therefore, \(\scr{1}\) commutes with \(\bgva\)
  and hence defines a \(\bgva\)-module map \(\LFock{0}\to\LFock{\psi}\). The
  identification \eqref{eq:FFR} 
  implies \(\LFock{0}\cong\sfmod{}{\Typ{0}^-}\) and
  \(\LFock{\psi}\cong \sfmod{2}{\Typ{0}^-}\). By comparing
  composition factors we see that the kernel must be either \(\img
  \phi_1\cong\VacMod\) or all of \(\LFock{0}\), so it is sufficient to show that the
  map \(\scr{1}:\LFock{0}\to\LFock{\psi}\) is non-trivial. A quick
  calculation reveals that
  \begin{equation}
    \scr{1}\ \ket{-\psi-\theta}=\ket{-\theta},
  \end{equation}
  and thus \(\scr{1}\) is not trivial. By comparing the composition
  factors of the sequence \eqref{eq:felder1} we also see that the
  sequence is an exact complex if each arrow is non-zero. Finally, the
  arrows are non-zero because
  \begin{equation}
    \scr{1} \ket{-\psi+m\theta}=\ket{m\theta},\qquad \forall m\in\ZZ.
  \end{equation}
\end{proof}

\begin{rmk}
  The existence of Felder complexes will not specifically be needed for any of the
  results that follow, however, it is interesting to note that the bosonic ghosts admit such
  complexes. These complexes were crucial in \cite{RidBos14} for
  computing the character formulae needed for the standard module formalism via resolutions of simple modules.
\end{rmk}

\section{Projective modules}
\label{sec:proj}

In this section we construct reducible yet indecomposable modules $\Stag$
on which the \(L_0\) operator has rank 2 Jordan blocks. We further
prove that the modules \(\sfmod{\ell}{\Stag}\) and
\(\sfmod{\ell}{\Typ{\lambda}}\) are both projective and injective, and
that in particular the \(\sfmod{\ell}{\Stag}\) are projective covers
and injective hulls of $\sfmod{\ell}{\VacMod}$ for any $\ell \in
\ZZ$. We refer readers unfamiliar with homological algebra concepts such as
injective and projective modules or extension groups to the book
\cite{HilHom97} and recall the following result for later use.
\begin{prop} \label{thm:homo}
  For a module $\module{R}$ which is both projective and injective, the Hom-Ext sequences terminate.
  That is, if we have the short exact sequence
  \begin{equation}
    \dses{\module{A}}{}{\module{R}}{}{\module{B}},
  \end{equation}
  for modules $\module{A}$, $\module{B}$.
  then this implies that the following two sequences are exact, for any module $\module{M}$.
  \begin{align}
    0 \lra {\Homgrp{\module{M}}{\module{A}}} \lra {\Homgrp{\module{M}}{\module{R}}} \lra {\Homgrp{\module{M}}{\module{B}}} \lra {\Extgrp{}{\module{M}}{\module{A}}} \lra 0, \\
    0 \lra {\Homgrp{\module{B}}{\module{M}}} \lra {\Homgrp{\module{R}}{\module{M}}} \lra {\Homgrp{\module{A}}{\module{M}}} \lra {\Extgrp{}{\module{B}}{\module{M}}} \lra 0.
  \end{align}
  Furthermore, $\Homgrp{\module{R}}{-}$ and $\Homgrp{-}{\module{R}}$ are exact covariant and exact contravariant functors respectively.
\end{prop}
This proposition assists with the calculation of Hom and Ext groups,
when all but one of the dimensions in the sequence are known. Using
the fact that the Euler characteristic (the alternating sum of the
dimensions of the coefficients) of an exact sequence vanishes, there is only one possibility for the remaining group.

\begin{prop}\label{thm:VacExt}
  The first extension groups of simple modules in \(\catFLE\)
  satisfy
  \begin{equation}
    \label{eq:VextDim}
    \dExt{\sfmod{k}{\VacMod}}{\sfmod{\ell}{\VacMod}}=
    \begin{cases}
      1,& |k-\ell|=1\\
      0,& \text{otherwise}
    \end{cases},\qquad
    \dExt{\sfmod{k}{\Typ{\lambda}}}{\Mmod}=\dExt{\Mmod}{\sfmod{k}{\Typ{\lambda}}}=0,
    \end{equation}
    where \(\lambda\in \RR/\ZZ,\ \lambda\neq \ZZ,\ k,\ell\in \ZZ\)
    and \(\Mmod\) is any module in \(\catFLE\). In particular the
    simple modules \(\sfmod{k}{\Typ{\lambda}}\) are both projective
    and injective in \(\catFLE\).
\end{prop}
\begin{proof}
  To conclude that \(\sfmod{k}{\Typ{\lambda}}\) is projective in
  \(\catFLE\) it is sufficient to show that
  \(\dExt{\Typ{\lambda}}{\Mmod}=0\) for all simple objects \(\Mmod\in\catFLE\). Injectivity in \(\catFLE\) then
  follows by applying the \({}^\ast\) functor and noting that
  \({\Typ{\lambda}}^\ast\cong \Typ{\lambda}\). Let
  \(\Mmod\in\catFLE\) be simple, then a necessary condition for the
  short exact sequence
  \begin{equation}
    \dses{\Mmod}{}{\Nmod}{}{\Typ{\lambda}},\qquad \Mmod\in\catFLE
  \end{equation}
  being non-split is that the respective ghost and conformal weights
  of \(\Typ{\lambda}\) and \(\Mmod\) differ only by integers. For
  simple \(\Mmod\) this rules out \(\Mmod=\sfmod{\ell}{\VacMod}\)
  or \(\Mmod=\sfmod{\ell}{\Typ{\mu}}\), \(\mu\neq \lambda\). So we
  consider \(\Mmod=\sfmod{\ell}{\Typ{\lambda}}\). Assume \(\ell=0\),
  let \(j\in\lambda\) and let \(v\) be a non-zero vector in the ghost
  and conformal weight \([j,0]\) space of the submodule
  \(\Mmod = \Typ{\lambda}\subset\Nmod\) and let \(w\in\Nmod\) be a
  representative of a non-zero coset in the \([j,0]\) weight space of
  the quotient \(\Nmod/\Typ{\lambda}\). Without loss of generality, we
  can assume that \(w\) is a \(J_0\)-eigenvector and a generalised
  \(L_0\)-eigenvector. A necessary condition for the indecomposability
  of \(\Nmod\), is the existence of an element \(U\) in the universal
  enveloping algebra 
  \(\UEA{\bglie}\) such that \(Uv=w\).
  Since \(v\) has minimal generalised conformal weight all
  positive modes annihilate \(v\), thus \(Uv\) can be expanded as a
  sum of products of \(\beta_0\) and \(\gamma_0\) with each summand
  containing as many \(\beta_0\) as \(\gamma_0\) factors, that is,
  \(Uv=f(J_0)v\) can be expanded as a polynomial in \(J_0\) acting on
  \(v\). Since \(\Nmod\in\catFLE\), \(J_0\) acts semisimply, hence
  \(f(J_0)v\propto v\). Since \(v\) is not a scalar multiple of \(w\),
  this contradicts the indecomposability
  of \(\Nmod\). Thus the exact sequence splits or, equivalently, the
  corresponding extension group vanishes.

  Assume $\Mmod = \sfmod{\ell}{\Typ{\lambda}}$ with \(\ell\neq0\), then by applying the \({}^\ast\) and
  \(\sfaut\) functors, we have
  \(\Extgrp{}{\Typ{\lambda}}{\sfmod{\ell}{\Typ{\lambda}}}=
    \Extgrp{}{\sfmod{\ell}{\Typ{\lambda}}}{\Typ{\lambda}}=
    \Extgrp{}{\Typ{\lambda}}{\sfmod{-\ell}{\Typ{\lambda}}}\).
  Thus the sign of \(\ell\) can be chosen at will and we can assume
  without loss of generality that \(\ell\ge1\). Further, from the formulae for the conformal weights of
  spectral flow twisted modules \eqref{eq:twistwts}, the conformal
  weights of \(\Typ{\lambda}\) and \(\sfmod{\ell}{\Typ{\lambda}}\)
  differ by integers if and only if \(\ell\cdot \lambda = \ZZ\).
  Let \(j\in \lambda\) be the minimal representative satisfying that the
  space of ghost weight \(j\) in \(\sfmod{\ell}{\Typ{\lambda}}\) has
  positive least conformal weight. The least conformal weight of
  the ghost weight \(j-1\) space is a negative integer, which we
  denote by \(-k\). See \cref{fig:projproof} for an illustration of how the
  weight spaces are arranged. Let \(v\in
  \Nmod\) be a non-zero vector of ghost weight \(j\) and
  generalised \(L_0\) eigenvalue 0, and hence a representative of
  a non-trivial coset of ghost and conformal weight \([j,0]\) in
  \(\Typ{\lambda}\cong
  \Nmod/\sfmod{\ell}{\Typ{\lambda}}\). Further let \(w\in
  \sfmod{\ell}{\Typ{\lambda}}\subset \Nmod\) be a non-zero vector
  of ghost and conformal weight \([j-1,-k]\). Both \(v\) and \(w\)
  lie in one-dimensional weight spaces and hence span them. If
  \(\Nmod\) is indecomposable, then there must exist an element
  \(U\) of ghost and conformal weight \([-1,-k]\) in \(\UEA{\bglie}\), such
  that \(U v=w\). We pick a \PBW{} ordering such that generators with larger
  mode index are placed to the right of those with lesser index
  and \(\gamma_n\) is placed to the right of \(\beta_n\) for any
  \(n\in\ZZ\). Thus \(Uv = \sum_{i=1}^k U^{(i)} \gamma_{i} v\),
  where \(U^{(i)}\) is an element of \(\UEA{\bglie}\) of ghost and
  conformal weight \([0,i-k]\). In \(\Typ{\lambda}\), \(\gamma_0\)
  acts bijectively on the space of conformal weight 0 vectors,
  hence there exists a \(\tilde{v} \in \Nmod\) such that \(\gamma_0
  \tilde{v} = v\). Since at ghost weight \(j\) the conformal
  weights of \(\Nmod\) are non-negative, we have \(\gamma_n
  \tilde{v}=0\), \(n\ge 1\) and thus \(Uv=\sum_{i=1}^k U^{(i)}
  \gamma_{i} \gamma_0\tilde{v}
  =\sum_{i=1}^k U^{(i)} \gamma_0\gamma_{i} \tilde{v}=0\),
  contradicting the indecomposability of \(\Nmod\).

  Next we consider the extensions of spectral flows of the vacuum module.
  By judicious application of the \({}^\ast\) and \(\sfaut\) functors,
  we can identify
  \(\Extgrp{}{\sfmod{k}{\VacMod}}{\sfmod{\ell}{\VacMod}}=
  \Extgrp{}{\VacMod}{\sfmod{k-\ell}{\VacMod}}=\Extgrp{}{\VacMod}{\sfmod{\ell-k}{\VacMod}}\). So
  without loss of generality, it is sufficient to consider the extension groups
  \(\Extgrp{}{\VacMod}{\sfmod{\ell}{\VacMod}}\) or equivalently short exact sequences of the form
  \begin{equation}
    \dses{\sfmod{\ell}{\VacMod}}{}{\Mmod}{}{\VacMod},\qquad \ell\in\ZZ_{\ge0}, \ \Mmod\in\catFLE.
  \end{equation}
  Let \(\sfmod{\ell}{\vac}\in\sfmod{\ell}{\VacMod}\subset \Mmod\)
  denote the the spectral flow image of the highest weight vector of
  \(\VacMod\) and let \(\omega\in\Mmod\) be a \(J_0\)-eigenvector and a choice of
  representative of the highest weight
  vector in \(\VacMod\cong \Mmod/\sfmod{\ell}{\VacMod}\).
  We first show that these sequences necessarily  split if
  \(\ell\neq1\).
  Assume \(\ell=0\), then the exact sequence can only be non-split
  if there exists a ghost and conformal weight \([0,0]\) element \(U\)
  in \(\UEA{\bglie}\) such that \(U \omega=
  a\sfmod{\ell}{\vac}-b\omega\), \(a,b\in\CC\), \(a\neq 0\). Without loss of
  generality we can replace \(U\) by \(\tilde{U}=U-b\wun\) to obtain
  \(\tilde{U}\omega=a\sfmod{\ell}{\vac}\). 
  Since the conformal weights of \(\VacMod\) are
  bounded below by 0, they satisfy the same bound in \(\Mmod\) and
  \(\beta_{n}\omega=\gamma_{n}\omega=0\), \(n\ge 1\), so \(\tilde{U}\omega\)
  can be expanded as a sum of products of \(\beta_0\) and \(\gamma_0\) acting
  on \(\omega\), with
  each summand containing the same number of \(\beta_0\) and
  \(\gamma_0\) factors. Equivalently, \(\tilde{U}\omega\) can be
  expanded as a polynomial in \(J_0\) acting on \(\omega\). Since
  \(\omega\) is a \(J_0\)-eigenvector \(\tilde{U}\omega\propto \omega\).
  Since \(\omega\) is not a scalar multiple of \(\sfmod{\ell}{\vac}\), \(\tilde{U}\omega=0\)
  contradicting indecomposability, and the exact sequence splits.

  Assume \(\ell \ge2\). The ghost and conformal weights of
  \(\sfmod{\ell}{\vac}\) are
  \([-\ell,-\frac{\ell(\ell+1)}{2}]\). Further, from the spectral flow
  formulae \eqref{eq:twistwts}, one can see that the weight spaces of
  ghost and conformal weight \([-1,h]\) of \(\sfmod{\ell}{\VacMod}\)
  vanish for \(h<\frac{(\ell+1)(\ell-2)}{2}\) and similarly the
  \([1,h]\) weight spaces  of \(\sfmod{\ell}{\VacMod}\) vanish for \(h<\frac{(\ell+1)(\ell+2)}{2}\). Since we are assuming
  \(\ell\ge 2\), \(\frac{(\ell+1)(\ell\pm2)}{2}\ge0\). Thus
  \(\gamma_n\omega=\beta_n\omega=0\), \(n\ge1\). If \(\Mmod\) is indecomposable,
  there must exist a ghost and conformal weight \([-\ell,-\frac{\ell(\ell+1)}{2}]\) element \(U\) in 
  \(\UEA{\bglie}\) such that \(U \omega=
  \sfmod{\ell}{\vac}\). Since the conformal
  weight of \(U\) is \(-\ell\), every summand of the expansion of
  \(U\omega\) into \(\beta\) and \(\gamma\) modes must contain factors
  of \(\gamma_n\) or \(\beta_n\) with \(n\ge1\) and we can choose a
  \PBW{} ordering where these modes are placed to the right. Thus \(U \omega=0\),
  contradicting indecomposability and the exact sequence splits.

  Assume \(\ell=1\), then \(\sfmod{}{\Typ{0}^+}\) provides an example
  for which the exact sequence does not split and the dimension of the
  corresponding extension group is at least 1. We show that it is also
  at most 1. Let \(\omega\) and \(\sfmod{}{\vac}\) be defined as for
  \(\ell\ge2\). By arguments analogous
  to those for \(\ell\ge2\), it follows that the \([1,h]\) weight space
  vanishes for \(h<0\) and the \([-1,h]\) weight space vanishes for
  \(h<-1\). Thus \(\beta_{n}\omega=\gamma_{n+1}\omega=0\), \(n\ge1\).
  The \([-1,-1]\) weight space of \(\sfmod{}{\VacMod}\) is
  one-dimensional and is hence spanned by \(\sfmod{}{\vac}\). If
  \(\Mmod\) is indecomposable, there must exist a ghost and conformal weight \([-1,-1]\) element \(U\) 
  in \(\UEA{\bglie}\) such that \(U \omega=
  \sfmod{}{\vac}\). Thus, \(U\omega\) can be
  expanded as
  \(f(J_0)\gamma_{1}\omega=f(0)\gamma_1\omega=a\vac\), where
  \(f(J_0)\) is a polynomial. Hence the
  isomorphism class of \(\Mmod\) is determined by the value of
  \(\gamma_1\omega\) in the one-dimensional \([-1,-1]\) weight space
  and \(\dExt{\VacMod}{\sfmod{}{\VacMod}}=1\).

\end{proof}

 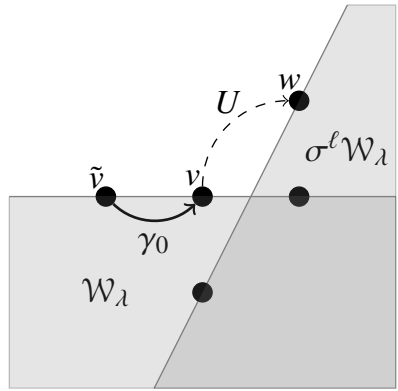
\begin{figure}[h]
	\centering
	\resizebox{!}{0.3\textwidth}{
		\begin{tikzpicture}[vec/.style={circle,draw=black,fill=black,inner sep=2pt,minimum size=5pt}]
		\draw[gray] (-0.5,-2) -- (1.5,2) ;
		\draw[gray] (2,0)  -- (-2,0);
		\path (-1,0) node[vec] (a) {} + (-0.1,0.2) node {$\tilde{v}$};
		\path (0,0) node[vec] (b) {}  + (-0.1,0.2) node {$v$};
		\path (1,0) node[vec] (c) {};
		\path (0,-1) node[vec] (d) {};
		\path (1,1) node[vec] (e) {} + (-0.1,0.2) node {$w$};
		\draw[->,thick] (a) to[out=-45,in=-135] node[below] {$\gamma_0$} (b);
		\draw[->,dashed] (b) to[out=90,in=180] node[above] {$U$} (e);
		\path (1.5,0.5) node {$\sfmod{\ell}{\Typ{\lambda}}$};
		\path (-1,-1) node {$\Typ{\lambda}$};
		\draw[fill=gray, opacity=0.2] (-2,0) -- (-2,-2) -- (2,-2) -- (2,0);
		\draw[fill=gray, opacity=0.2] (-0.5,-2) -- (2,-2) -- (2,2) -- (1.5,2);
		\end{tikzpicture}}
	\caption{This diagram is a visual aid for the proof of the
          inextensibility of the simple module \(\Typ{\lambda} \in \catFLE\),
          \(\lambda\in \RR/\ZZ, \lambda\neq \ZZ\). Here
          $\ell \ge 1$, $\ell \cdot \lambda = \ZZ$. The nodes
          represent the (spectral flows of) relaxed highest weight
          vectors of each module. Weight spaces are filled in
          grey. Conformal weight increases from top to bottom and
          ghost weight increases from right to left. 
	}
	\label{fig:projproof}
\end{figure}

Armed with the above results on extension groups, we can construct
indecomposable modules \(\sfmod{\ell}{\Stag}\in\catFLE\), which will
turn out to be projective covers and injective hulls of \(\sfmod{\ell}{\VacMod}\).
\begin{prop}\label{thm:projconstr}
  Recall that by the first free field realisation \(\phi_1\) of \cref{thm:ffr}, we can
  identify \(\LFock{\ell\psi}\cong\sfmod{\ell+1}{\Typ{0}^-}\). Define
  the \(\scr{1}\)-twisted action of \(\bgva\) on \(\LFock{-\psi}\oplus \LFock{0}\)
  by assigning
  \begin{equation}
    \label{eq:scrtwist}
    \beta(z) \mapsto \phi_1(\beta(z))=\vop{\psi+\theta}{z},\qquad \gamma(z)\mapsto \phi_1(\gamma(z))-\frac{\vop{-\theta}{z}}{z}=\normord{\psi(z)\vop{-\psi-\theta}{z}}-\frac{\vop{-\theta}{z}}{z},
  \end{equation}
  and determining the action of all other fields in \(\bgva\) through
  normal ordering and taking derivatives, where any vertex operator
  \(\vop{\lambda}{z}\) whose Heisenberg weight \(\lambda\) is in the
  coset \([\psi]=[-\theta]\) is defined to act as \(0\) on
  \(\LFock{0}\) and as usual
  on \(\LFock{-\psi}\).
  \begin{enumerate}
  \item The assignment is well-defined, that is, it represents the
    \ope{s} of \(\bgva\), and hence defines an action of
    \(\bgva\) on \(\LFock{-\psi}\oplus \LFock{0}\), where
    \(\oplus\) is meant as a direct sum of vector spaces without
    considering the module structure. Denote the
    module with this \(\scr{1}\)-twisted action by \(\Stag\).
    \label{itm:twist}
  \item The composite fields \(J(z)=\normord{\beta(z)\gamma(z)}\),
    \(T(z)=-\normord{\beta(z)\partial\gamma(z)}\) act as
    \begin{align}\label{eq:fieldids}
      J(z)\mapsto \phi_1(J(z))=-\theta(z),\qquad T(z)\mapsto
      \phi_1(T(z))+\frac{\vop{\psi}{z}}{z}=
\frac{\normord{\psi(z)^2}-\normord{\theta(z)^2}}{2}-\partial\frac{\psi(z)-\theta(z)}{2}+\frac{\vop{\psi}{z}}{z}.
    \end{align}
    The zero mode \(J_0\) therefore acts semisimply and \(L_0\) has
    rank 2 Jordan blocks. The vectors \(\ket{-\psi}, \ket{-\psi-\theta},
    \ket{\theta}, \ket{0}\in\Stag\) satisfy the relations
    \begin{equation}\label{eq:algaction}
    \beta_0\ket{-\psi}=\ket{\theta},\quad
    \gamma_1\ket{-\psi}=-\ket{-\psi-\theta},\quad
    \gamma_0\ket{\theta}=-\ket{0},\quad
    \beta_{-1}\ket{-\psi-\theta}=\ket{0},\quad
    L_0\ket{-\psi}=\ket{0}.
  	\end{equation}
    \label{itm:fieldformulae}
  \item The module \(\Stag\) is indecomposable and satisfies the
    non-split exact sequences
    \begin{subequations}
    \begin{align}
      &\dses{\sfmod{}{\Typ{0}^-}}{}{\Stag}{}{\Typ{0}^{-}},\label{eq:Psecs1}\\
      &\dses{{\Typ{0}^+}}{}{\Stag}{}{\sfmod{}{\Typ{0}^{+}}},\label{eq:Psecs2}
    \end{align}
    \end{subequations}
    which implies that its composition factors are
    \(\sfmod{\pm1}{\VacMod}\) and \(\VacMod\) with multiplicities 1
    and 2, respectively.
    \label{itm:moddecom}
  \item \(\Stag\) is an object in \(\catFLE\).
    \label{itm:FLEmod}
  \end{enumerate}
\end{prop}
See \cref{fig:weights} for an illustration of how the composition
factors of \(\Stag\) are linked by the action of \(\bgva\).

  \begin{figure}[h]
	\centering
	\vspace{3mm}
	\resizebox{!}{0.35\textwidth}{
		\begin{tikzpicture}[scale=1.7, vec/.style={circle,draw=black,fill=black,inner sep=2pt,minimum size=5pt}]
		\clip (-2.25,-1.5) rectangle + (4.5,3);
		\path (0,0.1) node[vec] (w) {} + (0.3,0.15) node {$\ket{-\psi}$};
		\path (0,-0.1) node[vec] (x) {} + (0.09,-0.21) node {$\ket{0}$};
		\path (-1,0) node[vec] (y1) {} + (-0.1,-0.2) node {$\ket{\theta}$};
		\path (1,1) node[vec] (x2) {} + (-0.2,0.2) node {$\ket{-\psi-\theta}$};
		\path (-1.75,-0.75) node {$\sfmod{-1}{\VacMod}$};
		\path (1.75,0.75) node {$\sfmod{}{\VacMod}$};
		\path (0.75,-0.75) node {${\VacMod}$};
		\draw[gray, opacity = 0.5] (3,0.1) -- (w) -- (-3.1,-3)
		(3,-0.1) -- (x) -- (-2.85,-3)
		(-4,0) -- (y1) -- (2,-3)
		(-1,-3) -- (x2) -- (3,3);
		\draw[->,thick] (y1) to[out=-45,in=-165] node[above] {$\gamma_0$} (x);
		\draw[->,thick] (x2) to[out=-90,in=-15]  (x);
		\path (1.1,0.3) node {$\beta_{-1}$};
		\draw[->,thick] (w) to[out=165,in=45] node[above] {$\beta_{0}$} (y1);
		\draw[->,thick] (w) to[out=75,in=-165] node[above] {$\gamma_{1}$} (x2);
		\end{tikzpicture}}
	\caption{The composition factors of $\Stag$ with the nodes
		representing the spectral flows of the highest weight vectors of $\sfmod{\ell}{\VacMod}$ for $-1 \le \ell \le 1$. The arrows give the
		action of $\bglie$ modes on the \hwvs{} of each factor. In this
		diagram, ghost weight increases to the left and conformal weight increases
		downwards. Note that there are two copies of $\VacMod$,
		illustrated by a small vertical shift in their weights. 
	}
	\label{fig:weights}
\end{figure}

\begin{proof}
  Part \ref{itm:twist} follows from \cite{FjeLog02}, where a general procedure was
  given for twisting actions by screening operators.
  The field identifications \eqref{eq:fieldids} of Part \ref{itm:fieldformulae} follow by
  evaluating definitions introduced there, while the relations \eqref{eq:algaction} follow by applying
  the field identifications.

  To conclude the first exact sequence of Part \ref{itm:moddecom} note
  that the action of \(\beta\) and \(\gamma\) closes on
  \(\LFock{0}\cong\sfmod{}{\Typ{0}^{-}}\), because \(\vop{-\theta}{z}\) acts trivially and
  quotienting by \(\LFock{0}\) leaves only \(\LFock{-\psi}\cong
  \Typ{0}^{-}\).

  To conclude the second exact sequence, let \(\vac\) be the highest
  weight vector of \(\VacMod\) and let \(\sfmod{\ell}{\vac}\) be the
  spectral flow images of \(\vac\). Then \(\ket{0}\in \LFock{0}\cong \sfmod{-1}{\Typ{0}^-}\) can
  be identified with \(\vac\) in the \(\VacMod\) composition factor of
  \(\sfmod{-1}{\Typ{0}^-}\) and \(\ket{-\psi-\theta}\) can be
  identified with \(\sfmod{}{\vac}\) in the \(\sfmod{}{\VacMod}\)
  composition factor. Further, 
  \(\ket{-\psi}\in\LFock{-\psi}\cong \Typ{0}^-\) can
  be identified with \(\vac\) in the \(\VacMod\) composition factor
  and \(\ket{\theta}\) can be
  identified with \(\sfmod{-1}{\vac}\) in the \(\sfmod{-1}{\VacMod}\)
  composition factor. 
  See \cref{fig:weights} for a diagram of the action of \(\beta\) and \(\gamma\)
  modes on \(\Stag\) and how they connect the different composition factors. It therefore
  follows that \(\ket{0}\) generates an indecomposable module whose
  composition factors are \(\sfmod{-1}{\VacMod}\) and \(\VacMod\),
  with \(\VacMod\) as a submodule and \(\sfmod{-1}{\VacMod}\) as a
  quotient. The module therefore satisfies the same non-split exact
  sequence \eqref{eq:Weqseq} as \(\Typ{0}^+\) does and since the extension
  groups in \eqref{eq:VextDim} are one-dimensional, this submodule is
  isomorphic to \(\Typ{0}^+\). After quotienting by the submodule
  generated by \(\ket{\theta}\), the formulae above imply that the
  quotient is isomorphic to \(\sfmod{}{\Typ{0}^+}\) and the second
  exact sequence of Part \ref{itm:moddecom} follows.

  Part \ref{itm:FLEmod} follows because \(J_0\) acts diagonalisably
  on \(\Stag\) and because \(\Stag\) has only finitely many
  composition factors all of which lie in \(\catR\) or \(\sfmod{}{\catR}\).
\end{proof}

\begin{thm} \label{thm:pproj}
  For every \(\ell \in \ZZ\) the indecomposable module $\sfmod{\ell}{\Stag}$
  is projective and injective in \(\catFLE\), and hence is a projective
  cover and an injective hull of the simple module \(\sfmod{\ell}{\VacMod}\).
\end{thm}
\begin{proof}
  Since spectral flow is an exact invertible functor, it is sufficient
  to prove projectivity and injectivity of \(\Stag\), rather than all spectral
  flow twists of \(\Stag\).
  We first show that \(\Stag\) is injective by showing that
  \(\dExt{\Wmod}{\Stag}=0\) for any simple module
  \(\Wmod\in\catFLE\).  Following that we will show
  \(\Stag^\ast=\Stag\), which, since \({}^\ast\) is an exact
  invertible contravariant functor, implies \(\Stag\) is also projective.

  A necessary condition for the
  non-triviality of such an extension is ghost weights differing only
  by integers.
  We therefore need not consider extensions by 
  $\sfmod{\ell}{\Typ{\lambda}}$, $\lambda \neq \ZZ$, so we 
  restrict our attention to
  short exact sequences of the form
  \begin{equation}
    \dses{\Stag}{}{\Mmod}{}{\sfmod{\ell}{\VacMod}}.
  \end{equation}
If the above extension is non-split, then there must exist a
  subquotient of \(\Mmod\) which is a non-trivial extension of
  \(\sfmod{\ell}{\VacMod}\) by one of the composition factors of
  \(\Stag\). By \cref{thm:VacExt} the above sequence must split if
  \(|\ell|\ge3\) and we therefore only consider \(|\ell|\le 2\).

  If \(\ell=2\), then the composition factor of \(\Stag\) non-trivially
extending \(\sfmod{2}{\VacMod}\) must be \(\sfmod{}{\VacMod}\). If the
extension is non-trivial, then this subquotient must be isomorphic to
\(\sfmod{2}{\Typ{0}^-}\). Further, if \(\sfmod{2}{\vac}\) is the
spectrally flowed highest weight vector of \(\sfmod{2}{\VacMod}\)
and \(\ket{-\psi-\theta}\in\Stag\) (see \cref{fig:weights}) is the
spectrally flowed highest weight vector of the \(\sfmod{}{\VacMod}\)
composition factor of \(\Stag\), then \(\beta_{-2}\sfmod{2}{\vac}=a
\ket{-\psi-\theta}\), \(a\in \CC\setminus\{0\}\).
The relations \eqref{eq:algaction} thus imply
\begin{equation}
  a\ket{0} = a\beta_{-1}\ket{-\psi-\theta}=a\beta_{-1}\beta_{-2}\sfmod{2}{\vac}=a\beta_{-2}\beta_{-1}\sfmod{2}{\vac}.
\end{equation}
However, \(\beta_{-1}\sfmod{2}{\vac}\) has conformal and ghost weight
\([-1,-2]\) and this weight space vanishes for both \(\Stag\) and
\(\sfmod{2}{\VacMod}\). Thus \(\beta_{-1}\sfmod{2}{\vac}\) and hence
\(a=0\), which is a contradiction. 

If \(\ell=1\), then the composition factor of \(\Stag\) non-trivially
extending \(\sfmod{}{\VacMod}\) must be \(\VacMod\). There are two
such composition factors in \(\Stag\). Any such non-trivial extension
must be isomorphic to \(\sfmod{}{\Typ{0}^-}\). If the non-trivial extension
involves the composition factor whose spectrally flowed highest weight
vector is represented by \(\ket{-\psi}\), then \(\beta_{-1}\sfmod{}{\vac}=a
\ket{-\psi}\), \(a\in \CC\setminus\{0\}\). The relations
\eqref{eq:algaction} thus imply
\begin{equation}
  a\ket{\theta} = a\beta_{0}\ket{-\psi-}=a\beta_{0}\beta_{-1}\sfmod{}{\vac}=a\beta_{-1}\beta_{0}\sfmod{}{\vac}.
\end{equation}
However, \(\beta_{0}\sfmod{}{\vac}=0\), so
\(a=0\), which is a contradiction.
If the non-trivial extension
involves the composition factor whose spectrally flowed highest weight
vector is represented by \(\ket{0}\), then there would exist \(a\in
\CC\setminus\{0\}\) such that
\(\beta_{-1}\sfmod{}{\vac}=a\ket{0}\). But then, by the relations
\eqref{eq:algaction}, \(\beta_{-1}\brac*{\sfmod{}{\vac}-a}
\ket{-\psi-\theta}=0\). Hence \(\brac*{\sfmod{}{\vac}-a}
\ket{-\psi-\theta}\) generates a direct summand isomorphic to
\(\sfmod{}{\VacMod}\), making the extension trivial.

If \(\ell=0\), then the composition factor of \(\Stag\) non-trivially
extending \(\VacMod\) must be \(\sfmod{}{\VacMod}\) or
\(\sfmod{-1}{\VacMod}\).
If there is a subquotient isomorphic to a non-trivial extension of
\(\VacMod\) by \(\sfmod{-1}{\VacMod}\), that is, isomorphic to
\(\Typ{0}^-\), then there exists \(a\in
\CC\setminus\{0\}\) such that
\(\beta_{0}\vac=a\ket{\theta}\).
But then, by the relations
\eqref{eq:algaction}, \(\beta_{0}\brac*{\vac-a}
\ket{\theta}=0\). Hence \(\brac*{\vac-a}\ket{\theta}\) generates a direct summand isomorphic to
\(\VacMod\), making the extension trivial. An analogous argument rules
out the existence of subquotient isomorphic a non-trivial extension of
\(\VacMod\) by \(\sfmod{-1}{\VacMod}\).

The cases \(\ell=-2\) and \(\ell=-1\) follow the same reasoning as
\(\ell=2\) and \(\ell=1\), respectively.

Now that we have established that \(\Stag\) is injective, we can apply
the functors \(\Homgrp{\Typ{0}^-}{-}\) and \(\Homgrp{\sfmod{}{\Typ{0}^+}}{-}\)
  to the short exact sequences \eqref{eq:Psecs1} and \eqref{eq:Psecs2}, respectively, to deduce
\(\dExt{\Typ{0}^-}{\sfmod{}{\Typ{0}^-}}=1=\dExt{\sfmod{}{\Typ{0}^+}}{\Typ{0}^+}\).
The indecomposable module \(\Stag\) is therefore the unique module
making the short exact sequences \eqref{eq:Psecs1} and \eqref{eq:Psecs2} non-split. By applying the
functor \(^\ast\) to these exact sequences, we see that \(\Stag^\ast\)
also satisfies these same sequences and hence \(\Stag\cong
\Stag^\ast\). This in turn implies \(\Extgrp{}{\Stag}{-}=0\) and hence
that \(\sfmod{\ell}{\Stag}\) is projective for all \(\ell\in \ZZ\).
\end{proof}

\section{Classification of indecomposables}
\label{sec:indecomp}

In this section, we give a classification of all indecomposable
modules in category $\catFLE$.
We already know any simple module is isomorphic to either $\sfmod{m}{\Typ{\lambda}}$ or
$\sfmod{m}{\Vmod{}}$, and we also know that the
$\sfmod{m}{\Typ{\lambda}}$ are inextensible due to being injective and
projective. We now complete the classification by
finding all the reducible indecomposables which can be built as finite
length extensions with composition factors isomorphic to spectral
flows of \(\VacMod\). To unclutter formulae, we use the notation $\module{M}_{n} = \sfmod{n}{\module{M}}$ for any module $\module{M}$.
The classification of indecomposable modules in $\catFLE$ closely
resembles the classification of indecomposable modules over the
Temperley-Lieb algebra with parameter at roots of unity given in
\cite{BelTem18}. Conveniently, the majority of the reasoning in
\cite{BelTem18} also applies to the $\bgva$-modules, with only minor
modifications --- primarily, that there are no exceptional cases to consider
for the bosonic ghost modules. 

The reducible yet indecomposable modules constituting the classification are the
spectral flows of the projective module $\Stag$, and two infinite
families. 
These two families, denoted $\Bmod{m}{n}$ and $\Tmod{m}{n}$, \(m,n\in\ZZ\), \(n\ge1\), are dual to each
other with respect to \(^\ast\), that is,
$\brac*{\Bmod{m}{n}}^\ast=\Tmod{m}{n}$, and further satisfy the
identifications $\Bmod{1}{} = \Tmod{1}{} = \Vmod{}$, $\Bmod{2}{} =
\sfmod{}{\Typ{0}^-}$ and $\Tmod{2}{} = \sfmod{}{\Typ{0}^+}$. The superscript \(m\)
is the number of composition factors or length of the module.
As a visual aid, we represent these indecomposable modules using Loewy diagrams.

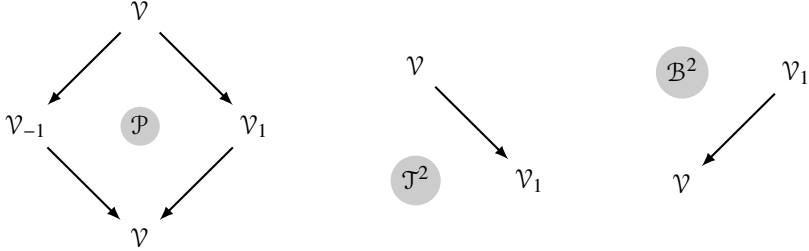
\begin{center} 
  \parbox{0.3\textwidth}{
    \begin{tikzpicture}[thick,>=latex,
      nom/.style={circle,draw=black!20,fill=black!20,inner sep=1pt}
      ]
      \node (top) at (0,1.5) [] {\(\Vmod{}\)};
      \node (left) at (-1.5,0) [] {\(\Vmod{-1}\)};
      \node (right) at (1.5,0) [] {\(\Vmod{1}\)};
      \node (bot) at (0,-1.5) [] {\(\Vmod{}\)};
      \node at (0,0) [nom] {\(\Stmod{}\)};
      \draw[->] (top) -- (left);
      \draw[->] (top) -- (right);
      \draw[->] (left) -- (bot);
      \draw[->] (right) -- (bot);
    \end{tikzpicture}
  }
  \parbox{0.2\textwidth}{
    \begin{tikzpicture}[thick,>=latex,
      nom/.style={circle,draw=black!20,fill=black!20,inner sep=1pt}
      ]
      \node (0) at (0,1.5) [] {\(\Vmod{}\)};
      \node (1) at (1.5,0) [] {\(\Vmod{1}\)};
      \node at (0,0) [nom] {\(\Tmod{2}{}\)};
      \draw[->] (0) -- (1);
    \end{tikzpicture}
  }
  \parbox{0.2\textwidth}{
    \begin{tikzpicture}[thick,>=latex,
      nom/.style={circle,draw=black!20,fill=black!20,inner sep=1pt}
      ]
      \node (0) at (0,0) [] {\(\Vmod{}\)};
      \node (1) at (1.5,1.5) [] {\(\Vmod{1}\)};
		\node at (0,1.5) [nom] {\(\Bmod{2}{}\)};
		\draw[->] (1) -- (0);
              \end{tikzpicture}
            }
\end{center}
Here the edges indicate the action of \(\bgva\) and the vertices represent
  the composition factors.

The indecomposable modules $\Bmod{m}{n}$ and $\Tmod{m}{n}$ can have either an even or an odd number of composition factors which are constructed inductively by different extensions.
Each chain is the result of extending either $\VacMod$, $\Bmod{2}{}$ or
$\Tmod{2}{}$ repeatedly by the length two indecomposables $\Bmod{2}{n}$ or
$\Tmod{2}{n}$, as either quotients or submodules.
For example, the even length module $\Bmod{2m}{}$ is constructed by repeatedly extending $\Bmod{2}{} = \sfmod{}{\Typ{0}^{-}}$ by spectrally flowed copies of itself, as submodules, as outlined in the diagram below.
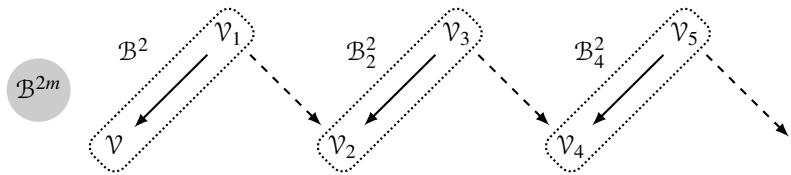
\begin{center}
		\begin{tikzpicture}[thick,>=latex,
		nom/.style={circle,draw=black!20,fill=black!20,inner sep=1pt}
		]
		\node (0) at (0,0) [] {\(\Vmod{}\)};
		\node (1) at (1.5,1.5) [] {\(\Vmod{1}\)};
		\node (2) at (3,0) [] {\(\Vmod{2}\)};
		\node (3) at (4.5,1.5) [] {\(\Vmod{3}\)};
		\node (4) at (6,0) [] {\(\Vmod{4}\)};
		\node (5) at (7.5,1.5) [] {\(\Vmod{5}\)};
		\node at (0.25,1.25) [] {\(\Bmod{2}{}\)};
		\node at (3.25,1.25) [] {\(\Bmod{2}{2}\)};
		\node at (6.25,1.25) [] {\(\Bmod{2}{4}\)};
		\draw[->] (1) -- (0);
		\draw[->, dashed] (1) -- (2);
		\draw[->] (3) -- (2);
		\draw[densely dotted,rounded corners=5] (0,0.4) -- (1.5,1.9) -- (1.9,1.5)  -- (0,-0.4) -- (-0.4,0) -- (0,0.4);
		\draw[densely dotted,rounded corners=5] (3,0.4) -- (4.5,1.9) -- (4.9,1.5)  -- (4.5,1.1)  -- (3,-0.4) -- (2.6,0) -- (3,0.4);
		\draw[->] (5) -- (4);
		\draw[->, dashed] (3) -- (4);
		\draw[densely dotted,rounded corners=5] (6,0.4) -- (7.5,1.9) -- (7.9,1.5)  -- (7.5,1.1)  -- (6,-0.4) -- (5.6,0) -- (6,0.4);
		\node (6) at (9,0) [] {};
		\draw[->, dashed] (5) -- (6);
		\node at (-1,0.75) [nom] {$\Bmod{2m}{}$};
		\end{tikzpicture}
\end{center}
The dotted boxes separate the component modules and the dashed lines
indicate a non-trivial action of the algebra present only in the extended modules.
Similarly for the odd length module $\Bmod{2m+1}{}$, the chain is built by repeated extensions of $\Bmod{1}{} = \Vmod{}$ by spectrally flowed copies of $\Tmod{2}{} = \sfmod{}{\Typ{0}^+}$ as quotients.
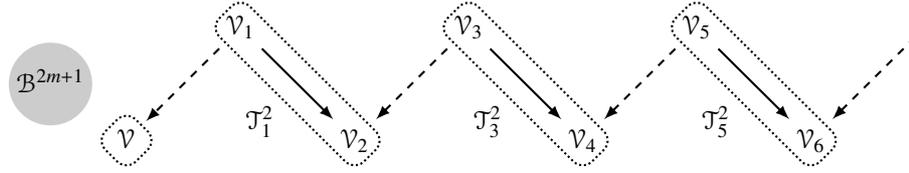
\begin{center}
  \begin{tikzpicture}[thick,>=latex,
    nom/.style={circle,draw=black!20,fill=black!20,inner sep=1pt}
    ]
    \node (0) at (0,0) [] {\(\Vmod{}\)};
    \node (1) at (1.5,1.5) [] {\(\Vmod{1}\)};
    \node (2) at (3,0) [] {\(\Vmod{2}\)};
    \node (3) at (4.5,1.5) [] {\(\Vmod{3}\)};
    \node (4) at (6,0) [] {\(\Vmod{4}\)};
    \node (5) at (7.5,1.5) [] {\(\Vmod{5}\)};
    \node (6) at (9,0) [] {\(\Vmod{6}\)};
    \node (7) at (10.5,1.5) [] {};
    \node at (1.75,0.25) [] {\(\Tmod{2}{1}\)};
    \node at (4.75,0.25) [] {\(\Tmod{2}{3}\)};
    \node at (7.75,0.25) [] {\(\Tmod{2}{5}\)};
    \draw[->,dashed] (1) -- (0);
    \draw[->] (1) -- (2);
    \draw[->, dashed] (3) -- (2);
    \draw[->] (3) -- (4);
    \draw[->] (5) -- (6);
    \draw[->, dashed] (5) -- (4);
    \draw[densely dotted,rounded corners=5] (0.2,0.2) -- (0,0.4) -- (-0.4,0) -- (0,-0.4) -- (0.4,0) -- (0.2,0.2);
    \draw[densely dotted,rounded corners=5]  (4.9,1.5) -- (6,0.4) -- (6.4,0) -- (6,-0.4) -- (5.6,0) -- (4.1,1.5) -- (4.5,1.9) -- (4.9,1.5);
    \draw[densely dotted,rounded corners=5]  (1.9,1.5) -- (3,0.4) -- (3.4,0) -- (3,-0.4) -- (2.6,0) -- (1.1,1.5) -- (1.5,1.9) -- (1.9,1.5);
    \draw[densely dotted,rounded corners=5]  (7.9,1.5) -- (9,0.4) -- (9.4,0) -- (9,-0.4) -- (8.6,0) -- (7.1,1.5) -- (7.5,1.9) -- (7.9,1.5);
    \draw[->, dashed] (7) -- (6);
    \node at (-1,0.75) [nom] {$\Bmod{2m+1}{}$};
  \end{tikzpicture}
\end{center}

When applying the \(^\ast\) functor, the composition factors stay the same, however,
all of the arrows corresponding to the action of the algebra between the
composition factors are reversed. Thus the top and bottom row are switched
in the Loewy diagram and \(\Bmod{}{}\) type indecomposables become
\(\Tmod{}{}\) type indecomposables. 
The letters \(\Tmod{}{}\) and \(\Bmod{}{}\) indicate the composition factor isomorphic to
\(\VacMod\) being either in the top or bottom row, respectively.
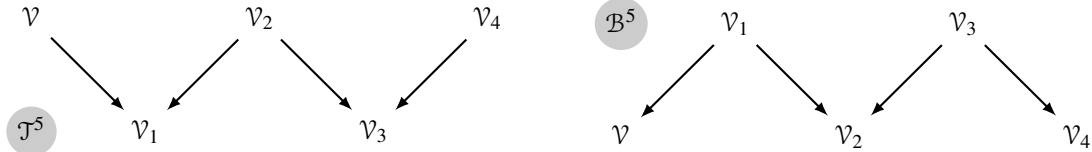
\begin{center}
  \parbox{0.45\textwidth}{
    \begin{tikzpicture}[thick,>=latex,
      nom/.style={circle,draw=black!20,fill=black!20,inner sep=1pt}
      ]
      \node (0) at (0,1.5) [] {\(\Vmod{}\)};
      \node (1) at (1.5,0) [] {\(\Vmod{1}\)};
      \node (2) at (3,1.5) [] {\(\Vmod{2}\)};
      \node (3) at (4.5,0) [] {\(\Vmod{3}\)};
      \node (4) at (6,1.5) [] {\(\Vmod{4}\)};
      \node at (0,0) [nom] {\(\Tmod{5}{}\)};
      \draw[->] (0) -- (1);
      \draw[->] (2) -- (1);
      \draw[->] (2) -- (3);
      \draw[->] (4) -- (3);		
    \end{tikzpicture}
  }
  \parbox{0.45\textwidth}{
    \begin{tikzpicture}[thick,>=latex,
      nom/.style={circle,draw=black!20,fill=black!20,inner sep=1pt}
      ]
      \node (0) at (0,0) [] {\(\Vmod{}\)};
      \node (1) at (1.5,1.5) [] {\(\Vmod{1}\)};
      \node (2) at (3,0) [] {\(\Vmod{2}\)};
      \node (3) at (4.5,1.5) [] {\(\Vmod{3}\)};
      \node (4) at (6,0) [] {\(\Vmod{4}\)};
      \node at (0,1.5) [nom] {\(\Bmod{5}{}\)};
      \draw[->] (1) -- (0);
      \draw[->] (1) -- (2);
      \draw[->] (3) -- (2);
      \draw[->] (3) -- (4);
    \end{tikzpicture}
  }
\end{center}

Recall from \cref{thm:VacExt} that non-split extensions only exist between composition factors $\Vmod{i}$, $\Vmod{j}$ with $\abs{i-j} = 1$. This explains the sequential order of the spectral flows of composition factors in the chains.  
We will show that these Loewy diagrams uniquely characterise the reducible
indecomposable modules, that is, no two non-isomorphic indecomposables
have the same Loewy diagram. This is essentially due to certain extension groups
being one-dimensional. These diagrams therefore provide a convenient
way for reading off all submodules and quotients of a given
indecomposable module and hence provide a shortcut for computing the
dimensions of \(\Hom\) groups.

\begin{thm} \label{thm:classify}
	\begin{enumerate}
	\item
	The initial identifications $\Bmod{1}{} = \Tmod{1}{} = \Vmod{}$, $\Bmod{2}{} = \sfmod{}{\Typ{0}^-}$ and $\Tmod{2}{} = \sfmod{}{\Typ{0}^+}$, along with the non-split short exact sequences below, uniquely characterise the modules $\Bmod{n}{}$ and $\Tmod{n}{}$.
	\begin{subequations}\label{eq:defseq}
	\begin{align} 
	&\dses{\Bmod{2n-1}{}}{}{\Bmod{2n+1}{}}{}{\Tmod{2}{2n-1}}, \label{eq:defseq1} \\ 
	&\dses{\Bmod{2}{2n-2}}{}{\Bmod{2n}{}}{}{\Bmod{2n-2}{}}, \label{eq:defseq2} \\ 
	&\dses{\Bmod{2}{2n-1}}{}{\Tmod{2n+1}{}}{}{\Tmod{2n-1}{}},\label{eq:defseq3} \\
	&\dses{\Tmod{2n-2}{}}{}{\Tmod{2n}{}}{}{\Tmod{2}{2n-2}}. \label{eq:defseq4}
	\end{align}
	\end{subequations}
\item
	Any reducible indecomposable module in $\catFLE$ is isomorphic to one of the following.
\begin{equation}\label{eq:indlist}
\Stmod{m} = \sfmod{m}{\Stmod{}}, \qquad
\Bmod{n}{m} = \sfmod{m}{\Bmod{n}{}}, \qquad
\Tmod{n}{m} = \sfmod{m}{\Tmod{n}{}}, \qquad m,n \in \ZZ, \ n\ge2.
\end{equation}
\end{enumerate}
\end{thm}

\cref{thm:classify} follows by first computing dimensions of Hom and
Ext groups to prove the existence of all indecomposables listed above,
and then showing that the list is closed under extensions by simple modules.

For a module $\Mmod$, we recall the following two well known substructures. The
first is the maximal semisimple submodule of $\Mmod$, called the socle and
which we denote $\soc\Mmod$. The second, called the head, is the maximal semisimple quotient of
$\Mmod$, defined to be the quotient of $\Mmod$ by its radical (the intersection of its maximal proper submodules), which we denote $\hd\Mmod$.
We also let $\Inje{\Mmod}$ and $\Proj{\Mmod}$ denote the injective hull and the projective cover of $\Mmod$ respectively.

\begin{prop}
  For any module \(\Mmod\in\catFLE\), we have
  \begin{equation}\label{eq:indechom}
    \Homgrp{\Vmod{n}}{\Mmod} \cong \Homgrp{\Vmod{n}}{\soc \Mmod}, \qquad \Homgrp{\Mmod}{\Vmod{n}} \cong \Homgrp{\hd \Mmod}{\Vmod{n}},
  \end{equation}
  and
  \begin{equation} \label{eq:sochd}
    \Inje{\Mmod} \cong \Inje{\soc \Mmod}, \qquad \Proj{\Mmod} \cong \Proj{\hd \Mmod } .
  \end{equation}
\end{prop}
This proposition allows us to find the Hom groups of indecomposables by examining their submodule and quotient structure and applying \eqref{eq:indechom}, and we can use Hom-Ext exact sequences to fill in the gaps. 
We also know that $\Proj{\Vmod{n}} = \Inje{\Vmod{n}} = \Stmod{n}$,
 and therefore knowledge of the submodules and quotients of the indecomposable modules immediately determines the injective hull and projective cover. Once these are known, we can construct injective and projective presentations which we use with the Hom-Ext exact sequence to determine the remaining Ext groups.
The dimensions of Hom and Ext groups involving indecomposables of
large length can be computed inductively from short length
indecomposables and so we prepare these here.
\begin{prop}\label{thm:homexttabs}
  The dimensions of Hom groups for the indecomposable modules
  \(\Vmod{m}\), \(\Bmod{2}{m}\), \(\Tmod{2}{m}\), \(\Stmod{m}\) are
  given by the following table.
  \begin{center}
    \bgroup
    \def\arraystretch{1.5}
    \begin{tabular}{cccccc}
      \multicolumn{1}{l}{}                             & \multicolumn{1}{l}{}                                & \multicolumn{4}{c}{$\Mmod$}     \\ \cline{2-6} 
      \multicolumn{1}{c|}{\multirow{5}{*}{$\Nmod$}} & \multicolumn{1}{c|}{$\dHom{\Nmod}{\Mmod}$} & \multicolumn{1}{c|}{\textbf{$\Vmod{m}$}} & \multicolumn{1}{c|}{\textbf{$\Tmod{2}{m}$}}    & \multicolumn{1}{c|}{\textbf{$\Bmod{2}{m}$}}    & \multicolumn{1}{c|}{\textbf{$\Stmod{m}$}}                      \\ \cline{2-6} 
      \multicolumn{1}{c|}{}                            & \multicolumn{1}{c|}{\textbf{$\Vmod{n}$}}           & \multicolumn{1}{c|}{$\delta_{n,m}$}                & \multicolumn{1}{c|}{$\delta_{n,m+1}$}               & \multicolumn{1}{c|}{$\delta_{n,m}$}                 & \multicolumn{1}{c|}{$\delta_{n,m}$}                                   \\ \cline{2-6} 
      \multicolumn{1}{c|}{}                            & \multicolumn{1}{c|}{\textbf{$\Tmod{2}{n}$}}    & \multicolumn{1}{c|}{$\delta_{n,m}$}                & \multicolumn{1}{c|}{$\delta_{n,m} +\delta_{n,m+1}$} & \multicolumn{1}{c|}{$\delta_{n,m}$}                 & \multicolumn{1}{c|}{$\delta_{n,m-1} +\delta_{n,m}$}                   \\ \cline{2-6} 
      \multicolumn{1}{c|}{}                            & \multicolumn{1}{c|}{\textbf{$\Bmod{2}{n}$}}    & \multicolumn{1}{c|}{$\delta_{n,m-1}$}              & \multicolumn{1}{c|}{$\delta_{n,m}$}                 & \multicolumn{1}{c|}{$\delta_{n,m-1} +\delta_{n,m}$} & \multicolumn{1}{c|}{$\delta_{n,m-1} +\delta_{n,m}$}                   \\ \cline{2-6} 
      \multicolumn{1}{c|}{}                            & \multicolumn{1}{c|}{\textbf{$\Stmod{n}$}}    & \multicolumn{1}{c|}{$\delta_{n,m}$}                & \multicolumn{1}{c|}{$\delta_{n,m} +\delta_{n,m+1}$} & \multicolumn{1}{c|}{$\delta_{n,m} +\delta_{n,m+1}$} & \multicolumn{1}{c|}{$\delta_{n,m-1} + 2\delta_{n,m} +\delta_{n,m+1}$} \\ \cline{2-6} 
    \end{tabular}
    \egroup
  \end{center}
\vspace{5mm}
  Further, the dimensions of Ext groups are given by the following table.
  \begin{center}
    \bgroup
    \def\arraystretch{1.5}
    \begin{tabular}{ccccc}
      \multicolumn{1}{l}{}                             & \multicolumn{1}{l}{}                                               & \multicolumn{3}{c}{$\Mmod$}                                                                                                                                      \\ \cline{2-5} 
      \multicolumn{1}{c|}{\multirow{4}{*}{$\Nmod$}} &
                                                         \multicolumn{1}{c|}{$\dExt{\Nmod}{\Mmod}$}                                      & \multicolumn{1}{c|}{\textbf{$\Vmod{m}$}}  & \multicolumn{1}{c|}{\textbf{$\Tmod{2}{m}$}}      & \multicolumn{1}{c|}{\textbf{$\Bmod{2}{m}$}}      \\ \cline{2-5} 
      \multicolumn{1}{c|}{}                            & \multicolumn{1}{c|}{\textbf{$\Vmod{n}$}}                          & \multicolumn{1}{c|}{$\delta_{n,m-1} +\delta_{n,m+1}$} & \multicolumn{1}{c|}{$\delta_{n,m+2}$}                 & \multicolumn{1}{c|}{$\delta_{n,m-1}$}                 \\ \cline{2-5} 
      \multicolumn{1}{c|}{}                            & \multicolumn{1}{c|}{\textbf{$\Tmod{2}{n}$}}                   & \multicolumn{1}{c|}{$\delta_{n,m+1}$}               & \multicolumn{1}{c|}{$\delta_{n,m+1} +\delta_{n,m+2}$} & \multicolumn{1}{c|}{0}                                \\ \cline{2-5} 
      \multicolumn{1}{c|}{}                            & \multicolumn{1}{c|}{\textbf{$\Bmod{2}{n}$}}                   & \multicolumn{1}{c|}{$\delta_{n,m-2}$}               & \multicolumn{1}{c|}{0}                                & \multicolumn{1}{c|}{$\delta_{n,m-2} +\delta_{n,m-1}$} \\ \cline{2-5} 
    \end{tabular}
    \egroup
  \end{center}
\end{prop}
\begin{proof}
  These dimensions follow from the exact sequences \eqref{eq:Weqseq},
  \cref{thm:VacExt} and \cref{thm:projconstr}, and judicious
  application of \cref{thm:homo}.
\end{proof}

The classification of indecomposable Temperley-Lieb algebra modules in
\cite{BelTem18} parametrises modules by finite sets of integers. The analogue
here is the subscript \(m\) in \eqref{eq:indlist} parametrising spectral flow,
which is an infinite index set. However, away from the end points of these
finite sets of integers the dimensions of Hom and Ext
groups for short length indecomposable modules in \cite[Propositions 2.17 and 2.18]{BelTem18} are equal to those for \(\bgva\)-modules after making the
identifications in the following table. 
\begin{center}
  \bgroup
  \def\arraystretch{1.5}
  \begin{tabular}{|l|l|l|l|l|l|l|}
    \hline
    \(\bgva\) & $\Vmod{m}$ & $\Bmod{n}{m}$   & $\Tmod{n}{m}$   & $\Bmod{2}{m}$ & $\Tmod{2}{m}$ & $\Stmod{m}$         \\ \hline
    \textbf{TL} & $\ImodTL{m}$ & $\BmodTL{n-1}{m}$ & $\TmodTL{n-1}{m}$ & $\CmodTL{m}$    & $\SmodTL{m}$    & $\PmodTL{m}, \ \JmodTL{m}$ \\ \hline
  \end{tabular}
  \egroup
\end{center}\vspace{2mm}
The dimensions of the remaining Hom and Ext groups, and therefore the
classification, follows from the same homological algebra reasoning as in \cite{BelTem18}, with no need for exceptions at the boundaries of the finite
sets in \cite{BelTem18}.
\begin{prop}{{\cite[Corollary 3.3, Propositions 3.6 and 3.7]{BelTem18}}}
  \label{thm:covhulpres}
The indecomposable modules \(\Bmod{n}{m}\), \(\Tmod{n}{m}\) have the
  following projective covers and hulls.
  \begin{center}
    \bgroup
    \def\arraystretch{1.5}
    \begin{tabular}{|c|c|c|c|c|}
      \hline
      $\module{M}$ & $\Bmod{2k+1}{m}$ & $\Bmod{2k}{m}$ & $\Tmod{2k+1}{m}$ & $\Tmod{2k}{m}$ \\
      \hline
      $\Proj{\module{M}}$ & $\bigoplus_{i=0}^{k-1} \Stmod{m+2i+1}$ & $\bigoplus_{i=0}^{k-1} \Stmod{m+2i+1}$ & $\bigoplus_{i=0}^k \Stmod{m+2i}$ & $\bigoplus_{i=0}^{k-1} \Stmod{m+2i}$ \\
      \hline
      $\Inje{\module{M}}$ & $\bigoplus_{i=0}^k \Stmod{m+2i}$ & $\bigoplus_{i=0}^{k-1} \Stmod{m+2i}$ & $\bigoplus_{i=0}^{k-1} \Stmod{m+2i+1}$ & $\bigoplus_{i=0}^{k-1} \Stmod{m+2i+1}$ \\
      \hline 
    \end{tabular}
    \egroup
  \end{center}\vspace{1mm}
  Further, projective and injective presentations are characterised by the following.
  \begin{center}
    \bgroup
    \def\arraystretch{1.5}
    \begin{tabular}{|c|c|c|c|c|}
      \hline
		$\module{M}$ & $\Bmod{2k+1}{m}$ & $\Bmod{2k}{m}$ & $\Tmod{2k+1}{m}$ & $\Tmod{2k}{m}$ \\
      \hline
      $\ker{\brac*{\Proj{\Mmod} \ra \Mmod}}$ & $\Bmod{2k-1}{m}$ & $\Bmod{2k}{m+1}$ & $\Tmod{2k+2}{m}$ & $\Tmod{2k}{m-1}$ \\
      \hline
      $\coker{\brac*{\Mmod \ra \Inje{\Mmod}}}$ & $\Bmod{2k+1}{m}$ & $\Bmod{2k}{m-1}$ & $\Tmod{2k}{m}$ & $\Tmod{2k}{m+1}$ \\
      \hline 
    \end{tabular}
    \egroup
  \end{center}
\end{prop}

This data now suffices to show that the extension groups corresponding
to the exact sequences \eqref{eq:defseq} of \cref{thm:classify} are
one-dimensional and hence uniquely characterise the indecomposable
\(\Bmod{}{}\) and \(\Tmod{}{}\) modules. The data can also be used to
show that any non-trivial extension of these indecomposable modules by
spectral flows of \(\VacMod\) will be a direct sum of modules in the
list \eqref{eq:indlist}. Hence \cref{thm:classify} follows.

For example, consider all possible extensions involving $\Bmod{3}{}$ and $\Vmod{n}$, starting with $\Extgrp{}{\Vmod{n}}{\Bmod{3}{}}$. Using the tables above, we start with the following injective presentation of $\Bmod{3}{}$
\begin{equation}
  \dses{\Bmod{3}{}}{}{\Stmod{}\oplus \Stmod{2}}{}{\Bmod{5}{-1}}.
\end{equation}
Applying the functor $\Homgrp{\Vmod{n}}{-}$, \cref{thm:homo} gives the Hom-Ext exact sequence
\begin{equation}
  0 \lra {\Homgrp{\Vmod{n}}{\Bmod{3}{}}} \lra {\Homgrp{\Vmod{n}}{\Stmod{}\oplus \Stmod{2}}} \lra {\Homgrp{\Vmod{n}}{\Bmod{5}{-1}}} \lra {\Extgrp{}{\Vmod{n}}{\Bmod{3}{}}} \lra 0.
\end{equation}
We can use \eqref{eq:indechom} to calculate these Hom groups, and the vanishing Euler characteristic implies $\dExt{\Vmod{n}}{\Bmod{3}{}} = \delta_{n,-1} + \delta_{n,1} + \delta_{n,3}$. These extensions are given by $\Bmod{4}{-1}$, $\Bmod{2}{}\oplus\Tmod{2}{1}$ and $\Bmod{4}{}$ for $n=-1$, 1 and 3, respectively.
Similarly apply the functor $\Homgrp{-}{\Vmod{n}}$ to the projective presentation
\begin{equation}
  \dses{\Vmod{1}}{}{\Stmod{1}}{}{\Bmod{3}{}}.
\end{equation}
The vanishing Euler characteristic then implies $\dExt{\Bmod{3}{}}{\Vmod{n}} = \delta_{n,1}$ with the extension being given by $\Stmod{1}$. Therefore we see that all extensions involving $\Bmod{3}{}$ and $\Vmod{n}$ return direct sums of classified indecomposable modules.

We end this section with some properties of the classified
indecomposable modules which will prove helpful in later sections.	
\begin{prop} \label{thm:dual}
  The evaluation of the \(^\ast\) functor of \cref{thm:duallist} on reducible indecomposable modules is given by
  \begin{align}
    \brac*{\Stmod{n}}^*  \cong \Stmod{n}, \qquad
    \brac*{\Bmod{m}{n}}^*  \cong \Tmod{m}{n}, \qquad
    \brac*{\Tmod{m}{n}}^*  \cong \Bmod{m}{n}.
  \end{align}
\end{prop}
\begin{proof}
The action of the \({}^\ast\) functor on the $\Bmod{}{}$ and $\Tmod{}{}$ modules
follows from their defining sequences \eqref{eq:defseq1} -- \eqref{eq:defseq4} being dual to each other, these sequences
uniquely characterising the \(\Bmod{}{}\) and \(\Tmod{}{}\) modules,
and proceeding by induction, starting with
$\brac*{\Typ{0}^{\pm}}^*=\Typ{0}^{\mp}$ from
\cref{thm:duallist}.\ref{itm:strdual}.
The self duality of $\Stmod{}$ is a consequence of \cref{thm:projconstr}.\ref{itm:moddecom}.
\end{proof}

\begin{cor}\label{thm:indseqs}
  The \(\Bmod{}{}\) and \(\Tmod{}{}\) indecomposable modules satisfy the
  non-split exact sequences.
  \begin{subequations}
    \begin{align}
      &\dses{\Vmod{}}{}{\Bmod{n}{}}{}{\Tmod{n-1}{1}},  \label{eq:xseq1}\\
      &\dses{\Vmod{2n}}{}{\Bmod{2n+1}{}}{}{\Bmod{2n}{}}, \label{eq:dseq1}\\
      &\dses{\Bmod{2n-1}{}}{}{\Bmod{2n}{}}{}{\Vmod{2n-1}}, \label{eq:dseq2}\\
      &\dses{\Bmod{n-2}{2}}{}{\Bmod{n}{}}{}{\Bmod{2}{}}, \label{eq:xseq3}\\
      &\dses{\Bmod{n-1}{1}}{}{\Tmod{n}{}}{}{\Vmod{}}, \label{eq:xseq2}\\
      &\dses{\Tmod{2n}{}}{}{\Tmod{2n+1}{}}{}{\Vmod{2n}}, \label{eq:dseq3}\\
      &\dses{\Vmod{2n-1}}{}{\Tmod{2n}{}}{}{\Tmod{2n-1}{}}. \label{eq:dseq4}
     \end{align}
  \end{subequations}
\end{cor}
\begin{proof}
  The above sequences being non-split is intuitively clear from the
  Loewy diagrams of the \(\Bmod{}{}\) and \(\Tmod{}{}\) indecomposables.
\end{proof}

\section{Rigid tensor category}
\label{sec:rigid}

In this section we prove that fusion furnishes category $\catFLE$ with the
structure of a rigid tensor category and define evaluation and coevaluation
maps for the simple projective modules to verify that these modules and maps
satisfy the conditions required for rigidity. We refer readers unfamiliar with
tensor categories or related notions such as rigidity to \cite{EtiTen15}.

\begin{thm} \label{thm:vtencat}
  Category \(\catFLE\) with the tensor structures defined by
  fusion is a braided tensor category.
\end{thm}

This theorem follows by verifying certain conditions which were proved to
be sufficient in \cite{HuaLog}, and \cite{HuaApp17}. To this end, we recall
some necessary definitions and results.

\begin{defn}
  \label{def:strong}
  Let \(\VOA{V}\) be a \va{} and let \(\Mmod\) be a module
  over \(\VOA{V}\). Let $A\le B$ be abelian groups.
  \begin{enumerate}
  \item The module $\Mmod$ is called {\em doubly-graded} 
    if both $\Mmod$ and \(\VOA{V}\) are 
    equipped with second gradations, in addition to conformal weight
    $h\in \CC$, which take values in $B$ and $A$, respectively. We will use
    the notations $\Mmod^{(j)}$ and 
    $\Mmod_{[h]}$ to denote the homogeneous spaces with respect to the
    additional grading or generalised conformal weight, respectively, 
    and denote the simultaneous homogeneous space by
    \(\Mmod^{(j)}_{[h]}=\Mmod^{(j)}\cap\Mmod_{[h]}\).
    The action of \(\VOA{V}\) on \(\Mmod\) is required to be
    compatible with the \(A\) and \(B\) gradation, that is,
    \begin{equation}
      v_n \VOA{V}^{(i)} \subset \VOA{V}^{(i+j)}, \quad  v_n \Mmod^{(k)} \subset \Mmod^{(j+k)}, \qquad v \in \VOA{V}^{(j)}, \ n \in \ZZ, \ i,j \in A, \ k \in B,
    \end{equation}
    and
    \begin{equation}
      \wun \in \VOA{V}^{(0)}_{[0]}, \qquad \omega \in \VOA{V}^{(0)}_{[2]},
    \end{equation}
    where $\wun$ is the vacuum vector and $\omega$ is the conformal vector.
  \item The module \(\Mmod\) is called \emph{lower bounded} if it is
      doubly graded and if for each \(j\in B\), \(\Mmod_{[h]}^{(j)}=0\) for
      \(\Re h\) sufficiently negative.
  \item The module \(\Mmod\)
    is called {\em strongly graded with respect to $B$} 
    if it is doubly graded; it is the direct sum of its homogeneous spaces,
    that is, 
    \begin{equation}
      \Mmod=\bigoplus_{\substack{h\in \CC\\ j\in B}}  \Mmod_{[h]}^{(j)},
    \end{equation}
    where the homogeneous spaces
    $\Mmod^{(j)}_{[h]}$ are all finite dimensional; and for fixed
    \(h\) and \(j\), \(\Mmod^{(j)}_{[h+k]}=0\), whenever \(k \in \ZZ \) is
    sufficiently negative.  The \va{} \(\VOA{V}\) is called {\em strongly
      graded with respect to $A$} if it is strongly graded as a module over
    itself. 
\item The module $\Mmod$ is called \emph{discretely strongly graded with
      respect to $B$} if all conformal weights are real and for any \(j\in
  B\), \(h\in \RR\) the space
  \begin{equation}
    \bigoplus_{\substack{\tilde{h}\in \RR\\ \tilde{h}\le h}}\Mmod^{(j)}_{[\tilde{h}]}
  \end{equation}
  is finite dimensional.
  \item For \(j\in B\), let $C_1(\Mmod)^{(j)}=\spn{u_{-h}w\in \Mmod^{(j)}\st u  \in
      \VOA{V}_{[h]}, h>0,\ w \in \Mmod}$. 
    A strongly graded module \(\Mmod\) is called \emph{graded $C_1$-cofinite}
    if $(\Mmod/C_1(\Mmod))^{(j)}$ is finite dimensional for all \(j\in B\).
  \end{enumerate}
\end{defn}

\begin{defn}
  Let \(A\le B\) be abelian groups.
  Let \(\VOA{V}\) be a \va{} graded by \(A\) and let \(\Mmod_1, \Mmod_2\) and
  \(\Mmod_3\) be modules over \(\VOA{V}\), graded by \(B\).
  Denote by \(\Mmod_3\set{x}\sqbrac*{\log x}\) the space of formal power
  series in \(x\) and \(\log x\) with coefficient in \(\Mmod_3\), where the
  exponents of \(x\) can be arbitrary complex numbers and with only finitely
  many \(\log x\) terms.
  A \emph{grading compatible logarithmic intertwining operator of type
  \(\itype{\Mmod_1}{\Mmod_2}{\Mmod_3}\)} is a linear map
  \begin{align}
    \mathcal{Y}:\Mmod_1&\to \Hom(\Mmod_2,\Mmod_3)\set{x}\sqbrac*{\log
                                        x}\nonumber\\
    m_1&\mapsto
                    \mathcal{Y}\brac*{m_1,x}=\sum_{\substack{s\ge0\\t\in
    \CC}}\brac*{m_1}_{t,s}x^{-t-1}\brac*{\log x}^s
  \end{align}
  satisfying the following properties.
  \begin{enumerate}
  \item Truncation: For any \(m_i\in \Mmod_i, i=1,2\), and \(s\ge0\)
	\begin{equation}
	\brac*{m_1}_{t+k,s}m_2 = 0
	\end{equation}   
	for sufficiently large $k \in \ZZ$. 
  \item \(L_{-1}\)-derivation: For any \(m_1\in \Mmod_1\),
    \begin{equation}
      \mathcal{Y}(L_{-1}m_1,x)=\frac{\dd}{\dd x}\mathcal{Y}(m_1,x).
    \end{equation}
  \item Jacobi identity:
    \begin{align}
      x_0^{-1} \delta \brac*{\frac{x_1-x_2}{x_0}} Y(v,x_1) \iop{m_1}{x_2}
      m_2 = & x_0^{-1} \delta \brac*{\frac{-x_2 +x_1}{x_0}} \iop{m_1}{x_2}
              Y(v,x_1) m_2 \nonumber\\
            &
              + x_2^{-1} \delta \brac*{\frac{x_1 - x_0}{x_2}} \iop{Y(v,x_0)
              m_1}{x_2} m_2 ,
                    \label{eq:intjac}
    \end{align}
    where \(Y\) denotes field map encoding the action of \(\VOA{V}\) on either
    \(\Mmod_1,\Mmod_2\) or \(\Mmod_3\) and \(\delta\) denotes the algebraic
    delta distribution, that is the formal power series
    \begin{equation}
      \delta
      \brac*{\frac{y-x}{z}} = \sum_{\substack{r \in \ZZ \\s \ge 0}}
      \binom{r}{s} (-1)^s x^s y^{r-s} z^{-r}.
    \end{equation}
  \item Grading compatibility: For any \(m_i\in \Mmod_i^{(j_i)},\ j_i\in
      B,\ i=1,2\), \(t\in \CC\) and \(s\ge0\) 
	\begin{equation}
	\brac*{m_1}_{t,s}m_2 \in \Mmod_3^{(j_1+j_2)}.
	\end{equation}  
  \end{enumerate}
\end{defn}

\begin{defn}\label{def:HLZddual}
  Let \(A\le B\) be abelian groups.
  Let \(\VOA{V}\) be a \va{} graded by \(A\) and let \(\Mmod_1\) and
  \(\Mmod_2\) be modules over \(\VOA{V}\), graded by \(B\). 
  We define the following properties for functionals \(\psi\in \Hom(\Mmod_1\otimes\Mmod_2,\mathbb{C})\).
  \begin{enumerate}
  \item \(P(w)\)-compatibility:
    \begin{enumerate}
    \item \emph{Lower truncation}: For any \(v\in \VOA{V}\),
      \(v_n\psi=0\), for any sufficiently large \(n\in\ZZ\).
    \item For any \(v\in \VOA{V}\) and \(f\in \CC[t,t^{-1},(t^{-1}-w)^{-1}]\) the identity
      \begin{equation}\label{eq:pzcomp}
        v f(t)\psi= v \iota_+\brac*{f(t)}\psi
      \end{equation}
      holds. Here \(\iota_+\) means expanding about \(t=0\) such that the
      exponents of \(t\) are bounded below and the action of \(\VOA{V}\otimes
      \CC[t,t^{-1},(t^{-1}-w)^{-1}]\) or \(\VOA{V}\otimes
      \CC((t))\) on \(\psi\) is characterised by
      \begin{equation}
        \dpair{vg(t)\psi}{m_1\otimes m_2}=\dpair{\psi}{\iota_+\circ
          T_w\brac*{v^{\opp}g(t^{-1})} m_1\otimes
            m_2}+ \dpair{\psi}{m_1\otimes\iota_+\brac*{v^{\opp} g(t^{-1})}
            m_2},
          \label{eq:ddaction}
        \end{equation}
        where \(m_i\in\Mmod_i\), \(v\in\VOA{V}\), \(g\in
        \CC[t,t^{-1},(t^{-1}-w)^{-1}]\),  
        \(T_w\) replaces
        \(t\) by \(t+w\), \(v^{\opp}=\ee^{t^{-1}L_1}(-t^2)^{L_0}vt^{-2}\), and (assuming \(v\)
        has conformal weight \(h\)) \(vt^n m_i= v_{n-h+1}m_i\). 
      \end{enumerate}
      Denote by \(\comp{\Mmod_1}{\Mmod_2}\) the vector space of all \(P(w)\)-compatible functionals.
  \item \(P(w)\)-local grading restriction:
    \begin{enumerate}
    \item The functional \(\psi\) is a finite sum of vectors that are both
      \(B\)-homogeneous and \(L_0\) generalised
      eigenvectors.
    \item Denote the smallest subspace of
      \(\Hom(\Mmod_1\otimes\Mmod_2,\mathbb{C})\) containing \(\psi\) and 
      stable under \(V\otimes \CC[t,t^{-1}]\) by \(\Mmod_\psi\). Then
      \(\Mmod_\psi\) must satisfy for any \(r\in \CC, b\in B\)
      \begin{equation}
        \dim\brac*{{\Mmod_\psi}_{\sqbrac*{r}}^{\brac*{b}}}< \infty,\qquad \text{and}\qquad \dim\brac*{{\Mmod_\psi}_{\sqbrac*{r+k}}^{\brac*{b}}}=0,
      \end{equation}
      for sufficiently large \(k\in\ZZ\).
    \end{enumerate}
    Denote by \(\lgr{\Mmod_1}{\Mmod_2}\) the vector space of all
    \(P(w)\)-local grading restricted functionals.
  \end{enumerate}
  Define \(\Mmod_1 \ddfuse\Mmod_2=\comp{\Mmod_1}{\Mmod_2}\cap \lgr{\Mmod_1}{\Mmod_2}\).
\end{defn}

\begin{rmk}
The variable $w$ in $P(w)$ denotes the insertion point of the tensor product
constructed in \cite{HuaLog}, where it is usually denoted $z$ and hence the
tensor product is referred to as the $P(z)$-tensor product.
\end{rmk}

\begin{thm}[{Huang-Lepowsky-Zhang \cite[Part IV, Theorem 5.44, 5.45, 5.50]{HuaLog}}]
  Let \(A\le B\) be abelian groups. Let \(\VOA{V}\) be a \va{} graded by \(A\)
  with a choice of module category \(\categ{C}\) which is closed under
  restricted duals
  and let \(\Mmod_1, \Mmod_2\in\categ{C}\) be graded by \(B\). Then
  \(\comp{\Mmod_1}{\Mmod_2}\) and \(\Mmod_1\ddfuse \Mmod_2\) are modules over
  \(\VOA{V}\). Further, if \(\Mmod_1\ddfuse \Mmod_2 \in \categ{C}\), then
  \(\Mmod_1\fuse \Mmod_2\cong \brac*{\Mmod_1\ddfuse \Mmod_2}^\prime.\)
  \label{thm:fusiondef}
\end{thm}

In \cite{HuaLog} \(\Mmod_1\ddfuse \Mmod_2\) is originally defined as the image
of all intertwining operators with \(\Mmod_1\) and  \(\Mmod_2\) as factors,
but it is then shown that this is equivalent to the definition given above.
The construction of fusion products through \cref{def:HLZddual} is
sometimes called the HLZ double dual construction. In addition to the primary
reference \cite{HuaLog}, the authors also recommend the survey \cite{KanRid18},
which relates this construction of fusion to others in the literature.

\begin{thm}[{Huang-Lepowsky-Zhang \cite[Part VIII, Theorem 12.15]{HuaLog}, Huang
  \cite[Theorem 3.1]{HuaApp17}}] \label{thm:huang}
  For any \va{} and module category $\categ{C}$ satisfying the
    conditions below, fusion equips category \(\categ{C}\) with the structures
    of an additive braided tensor category.
  \begin{enumerate}
  \item The \va{} and all its modules in \(\categ{C}\) are strongly graded and all logarithmic
    intertwining operators are grading compatible.
    \cite[Part III, Assumption 4.1]{HuaLog}.\label{itm:grading}
  \item $\categ{C}$ is a full subcategory of the category of 
    strongly graded modules and is closed under the contragredient functor and
    under taking finite direct sums 
    \cite[Part IV, Assumption 5.30]{HuaLog}.
    \label{itm:sums}
  \item All objects in $\categ{C}$ have real weights and the non-semisimple part of $L_0$ acts on them nilpotently \cite[Part V, Assumption 7.11]{HuaLog}.\label{itm:real}
  \item  \(\categ{C}\) is closed under images {of module homomorphisms}  \cite[Part VI, Assumption 10.1.7]{HuaLog}.\label{itm:closure}
  \item The convergence and extension properties for either
    products or iterates holds \cite[Part VII, Theorem 11.4]{HuaLog}. \label{itm:prodsits}
  \item For any objects $\Mmod_1, \Mmod_2 \in \categ{C}$, let $\Mmod_{v}$ be
    the doubly graded $\VOA{V}$-module generated 
    by a \(B\)-homogeneous generalised \(L_0\) eigenvector
    $v \in \comp{\Mmod_1}{\Mmod_2}$.
    If $\Mmod_{v}$ is 
lower bounded
    then $\Mmod_{v}$ is 
 strongly graded
    and an object in 
    $\categ{C}$ \cite[Theorem 
    3.1]{HuaApp17}. Here the action of $L_0$ is defined by 
    \eqref{eq:ddaction}, 
    given in \cref{def:HLZddual}. 
    \label{itm:cycmods}
  \end{enumerate}
\end{thm}

Conditions \ref{itm:grading} -- \ref{itm:closure} of \cref{thm:huang} hold by
construction for category \(\catFLE\), so all that remains is verifying Conditions
\ref{itm:prodsits} and \ref{itm:cycmods}.

\begin{thm}
    \label{thm:gradconext}
  Let \(A\le B\) be abelian groups, let \(\VOA{V}\) be a doubly
  \(A\)-graded \voa{} and let \(\overline{\VOA{V}}\) be a vertex subalgebra
  of \(\VOA{V}^{(0)}\). Further, let \(\Wmod_i,\ i=0,1,2,3,4\) be doubly
  \(B\)-graded \(\VOA{V}\)-modules. Finally let \(\mathcal{Y}_1\),
  \(\mathcal{Y}_2\), \(\mathcal{Y}_3\) and \(\mathcal{Y}_4\) be logarithmic
  grading compatible intertwining operators
  of types \(\binom{\Wmod_0}{\Wmod_1,\ \Wmod_4}\), \(\binom{\Wmod_4}{\Wmod_2,\ \Wmod_3}\),
  \(\binom{\Wmod_0}{\Wmod_4,\ \Wmod_3}\) and \(\binom{\Wmod_4}{\Wmod_1,\ \Wmod_2}\)
  respectively.
  If the modules  \(\Wmod_i,\ i=0,1,2,3\) (note \(i=4\) is excluded) are
  discretely strongly graded, and graded \(C_1\)-cofinite as
  \(\overline{\VOA{V}}\)-modules,
  then \(\mathcal{Y}_1\), \(\mathcal{Y}_2\) satisfy the convergence and
  extension property for products and \(\mathcal{Y}_3\), \(\mathcal{Y}_4\)
  satisfy the convergence and extension property for iterates.
\end{thm}

The above theorem follows from the proof of \cite[Theorem 7.2]{Yang18},
however, in \cite{Yang18} some assumptions are made on the category of
strongly graded modules (see \cite[Assumption 7.1, Part 3]{Yang18}) which do not
hold for \(\bgva\). Fortunately, the proof of \cref{thm:gradconext} does not
depend at all on any categorical considerations or even on the details of the
intertwining operators \(\mathcal{Y}_i\) beyond their types. It merely depends on certain
finiteness properties of the modules \(\Wmod_i\). We reproduce the proof of Yang
in \cref{sec:suffconvext}, with some minor tweaks to the arguments, to show that the
conclusion of \cref{thm:gradconext} holds, without making any assumptions on the
category of all strongly graded modules.

\begin{lem} \label{thm:conext}
  The convergence and extension properties for products and iterates holds for $\catFLE$.
\end{lem}
\begin{proof}
  If, in the assumptions of \cref{thm:gradconext}, we set $\VOA{V} = \bgva$ and
  grade by ghost weight, so that $A = \ZZ$, then the modules of \(\catFLE\)
  are graded by \(B=\RR\). We further choose $\overline{\VOA{V}}
  = \bgva^{(0)}$, that is, the vertex subalgebra given by the ghost weight
  0 subspace of \(\bgva\). The lemma then follows by verifying that all
  modules in \(\catFLE\) are discretely strongly graded and graded
  \(C_1\)-cofinite as modules over \(\overline{\VOA{V}}\).
  
  All modules in \(\catFLE\) are discretely strongly graded by ghost 
  weight $j \in \RR$. To prove this, we need to check that the
  simultaneous ghost and conformal weight spaces are finite
  dimensional and that every ghost weight homogeneous space has lower bounded
  conformal weights. The simultaneous ghost and conformal weight spaces of
  objects in \(\catR\) and therefore also those of
  \(\sfmod{\ell}{\catR}\) are finite dimensional by
  construction. Thus, since the objects of \(\catFLE\) are finite
  length extensions of those in \(\sfmod{\ell}{\catR}\), the objects
  of \(\catFLE\) also have finite dimensional simultaneous ghost and
  conformal weight spaces. Similarly we have that the objects in
  $\catFLE$ are graded lower bounded and therefore discretely strongly graded. 
  
  Next we need to decompose objects of $\catFLE$ as
  $\overline{\VOA{V}}$-modules,. It is known that $\overline{\VOA{V}}$ is
  generated by $\set{\normord{\beta(z) \brac*{ \partial^n \gamma(z)}}, \ n
    \ge 0}$ and is isomorphic to $W_{1+\infty} \cong W_{3,-2} \otimes
  \mathcal{H}$ where $W_{3,-2}$ is the singlet algebra at $c=-2$ and
  $\mathcal{H}$ is a rank 1 Heisenberg algebra \cite{Wan98, Lin09}. Note
    that the conformal vector of \(W_{1+\infty}\) is usually chosen so as
    to have a central charge of \(1\). Since we require \(\overline{\VOA{V}}\) to
    embed conformally into \(\bgva\), that is, to have the same conformal
    vector as \(\bgva\) and the central charge of \(\bgva\) is \(2\), we
    choose conformal vector of our Heisenberg algebra \(\mathcal{H}\) so
    that its central charge is 4 (the conformal structure of \(W_{3,-2}\)
      is unique). Fortunately, this does not complicate
    matters, as the simple modules over \(\mathcal{H}\) are just Fock
    spaces regardless of the central charge or conformal vector.
    The tensor factors of \(W_{1+\infty}\) decompose nicely
    with respect to the free field realisation of
    \cref{thm:ffr}.\ref{itm:ffr2}. The Heisenberg algebra $\mathcal{H}$ is
    generated by $\theta(z)$ and the singlet algebra $W_{3,-2}$ is a vertex subalgebra
    of the Heisenberg algebra generated by $\psi(z)$.  
  
 We denote Fock spaces over the rank 1 Heisenberg algebras generated by
  \(\psi\) and \(\theta\), respectively, by the same symbol \(\Fock{\mu}\),
  where, the index \(\mu\in\CC\) indicates the respective eigenvalues of the
  zero modes \(\psi_0\) and \(\theta_0\).
  All simple $\overline{\VOA{V}} \cong W_{1 + \infty}$ modules can be
    constructed via its free field realisation as 
  $\Vmod{\pair{\lambda}{\psi}} \otimes \Fock{\pair{\lambda}{\theta}}$
  \cite[Corollary 6.1]{AdaW01}, where $\Vmod{\pair{\lambda}{\psi}}$, as a $W_{3,-2}$-module, is the simple
  quotient of the submodule of $\Fock{\pair{\lambda}{\psi}}$ generated by the
  highest weight vector. 
  The homogeneous space $\brac*{\sfmod{\ell}{\VacMod}}^{(j)}$ is simple,
  as a $\overline{\VOA{V}}$-module \cite[Lemma 4.1]{Wan98}, see also
  \cite{Kac95, Mat94}. 
 Recall from \cref{thm:ffr}.\ref{itm:ffr2} that with \(K=\zspn{\psi,\theta}\)
  and \(\Lambda\in L/K\), we can construct the simple projective \(\bgva\)-modules as \({\sfmod{\pair{\Lambda}{\psi+\theta}}{\Typ{\pair{\Lambda}{\psi}}}} \cong {\LFock{\Lambda}}\).
  To identify the homogeneous space $\brac*{\sfmod{\pair{\Lambda}{\psi+\theta}}{\Typ{\pair{\Lambda}{\psi}}}}^{(j)}$ as a $\overline{\VOA{V}}$-module,
    we use the fact that $J(z) = - \theta(z)$, thus the \(\RR\)-grading on
      \(\LFock{\Lambda}\) is given by the eigenvalue of
      \(-\theta_0\). Therefore, for \(j\in\RR\),
  \begin{equation}
  \brac*{\sfmod{\pair{\Lambda}{\psi+\theta}}{\Typ{\pair{\Lambda}{\psi}}}}^{(j)}
  \cong \LFock{\Lambda}^{(j)}\cong \brac*{\bigoplus_{\lambda \in \Lambda}
    \Fock{\pair{\lambda}{\psi}} \otimes \Fock{\pair{\lambda}{\theta}}}^{(j)} =
  \begin{dcases}
    \Fock{\pair{\Lambda}{\psi+\theta}+j} \otimes \Fock{-j}, & j \in
    \pair{\Lambda}{\psi},\\
    0,& j \notin \pair{\Lambda}{\psi}.
  \end{dcases}
  \end{equation}
For $\pair{\Lambda}{\psi+\theta}+j \notin \ZZ$,
$\Fock{\pair{\Lambda}{\psi+\theta}+j}$  
is irreducible as a
  $\mathcal{W}_{3,-2}$ module, by \cite[Section 3.2]{CreFal13}, see also
  \cite[Section 5]{Ada07}. Thus,
  $\brac*{\sfmod{\pair{\Lambda}{\psi+\theta}}{\Typ{\pair{\Lambda}{\psi}}}}^{(j)}
  \cong \Fock{\pair{\Lambda}{\psi+\theta}+j}
  \otimes \Fock{-j}=\Vmod{\pair{\Lambda}{\psi+\theta}+j} \otimes \Fock{-j}
  $. The finite length modules of $W_{3,-2}$ are all $C_1$-cofinite
  \cite[Corollary 14]{CreLog16}, as are $\mathcal{H}$-modules, since Fock
  spaces have \(C_1\)-codimension 1.
  Therefore all $\overline{\VOA{V}}$-modules appearing as the homogeneous spaces of modules
  in \(\catFLE\) are $C_1$-cofinite and the lemma follows. 
\end{proof}

\begin{lem} \label{thm:closurecomps}
  Let \(\Mmod_1,\Mmod_2\) be modules in \(\catFLE\), let \(\Wmod\) be an
  indecomposable smooth (or weak)
  \(\bgva\) module
  and let \(\mathcal{Y}\) be a surjective logarithmic intertwining operator of
  type \(\binom{\Wmod}{\Mmod_1,\ \Mmod_2}\).
  \begin{enumerate}
  \item The logarithmic intertwining operator $\mathcal{Y}$ is grading
          compatible and the module \(\Wmod\) is doubly
          graded. \label{itm:gradcomp}
      \item  If \(\Mmod_1\in \sfmod{k}{\catR} ,\Mmod_2 \in \sfmod{\ell}{\catR}\),
    then 
    \(\Wmod \in \catFLE\)
    and \(\Wmod\) has composition factors only in $\sfmod{k+\ell}{\catR}$ and
     $\sfmod{k+\ell-1}{\catR}$.
    \label{itm:projfuscalc}
  \item If \(\Mmod_1\) has composition factors only in \(\sfmod{k}\catR\) and
    \(\sfmod{k-1}{\catR}\), and has composition factors only in
    \(\sfmod{\ell}\catR\) and 
    \(\sfmod{\ell-1}{\catR}\), 
    \(\Wmod\in \catFLE\) and \(\Wmod\)
    has composition factors only in
    \(\sfmod{k+\ell+i}{\catR}, -3\le i\le 0\). 
    \label{itm:closureprt}
  \end{enumerate}
\end{lem}
\begin{proof}
  Due to the compatibility of fusion with spectral flow, see
  \cref{thm:specflow}, it is sufficient to only consider \(k=\ell=0\).
 We prove Part \ref{itm:gradcomp} first.
  Let \(\Mmod_1,\Mmod_2\), be
  modules in \(\catFLE\). 
  Let \(v\in \bgva\) be the vector corresponding to the field \(J(z)\)
  and take the residue with respect
    to \(x_0\) and \(x_1\) in the Jacobi identity \eqref{eq:intjac}. This yields
    \begin{equation}
      J_0\iop{m_1}{x_2} m_2=\iop{m_1}{x_2} J_0 m_2 + \iop{J_0 m_1}{x_2}m_2.
    \end{equation}
    Hence, since the fusion factors \(\Mmod_i\) are graded by ghost weight, the
    fusion product will be too. This means that the intertwining operator will
    be grading compatible and \(\Wmod\) must be doubly graded.
  
    Next we prove Part \ref{itm:closureprt}.
    Assume that \(\Mmod_1,\Mmod_2\) have composition factors only in \(\catR\)
    and \(\sfmod{-1}{\catR}\).
    Note that \(J_n,\ n\ge1\) acts locally nilpotently on any object in
  \(\catFLE\) and that \(\beta_{n-\ell},\ \gamma_{n+\ell},\ n\ge1\) act
  locally nilpotently on any object in \(\sfmod{\ell}{\catR}\) (recall that
  local nilpotence is one of the defining properties of \(\sfmod{\ell}{\catR}\)).
  We first show that \(J_{n},\ \beta_{n+1},\ \gamma_n,\ n\ge 1\) acting
  locally nilpotently on \(\Mmod_1, \Mmod_2\) implies that \(J_{n},\ \beta_{n+3},\ \gamma_n,\ n\ge 1\) act
  locally nilpotently on \(\Wmod\). 
  Let \(h\) be the conformal weight of \(v=\beta,\gamma\) or \(J\),
  multiply both sides of the Jacobi identity \eqref{eq:intjac} by \(x_0^kx_1^{n+h-1},\ 
  n,k\in\ZZ\) and take residues with respect to \(x_0\) and \(x_1\). This yields
  \begin{align}
    \sum_{s \ge 0} \binom{k}{s} (-1)^s x_2^s v_{n-s} \iop{m_1}{x_2}
    m_2 =  & \sum_{s \ge 0} \binom{k}{s} (-1)^s x_2^{k-s}
             \iop{m_1}{x_2} v_{n-k+s} m_2 \nonumber\\
           & + \sum_{s \ge 0} \binom{s-n+k-h}{s} (-1)^s x_2^{n-k+h-s-1}
             \iop{v_{s-h+k+1} m_1}{x_2} m_2 .
                 \label{eq:specjac}
  \end{align}
  Set \(v=\gamma\)
  (and thus \(h=0\)) and \(k=0\) in
  \eqref{eq:specjac} to obtain
  \begin{align}
    \gamma_{n} \iop{m_1}{x_2} m_2 =  \iop{m_1}{x_2} \gamma_{n} m_2  + \sum_{s = 0}^{n} \binom{s-n}{s} (-1)^s x_2^{n-s-1} \iop{\gamma_{s+1} m_1}{x_2} m_2 .
  \end{align}
  This implies the local nilpotence of \(\gamma_n,\ n\ge
  1\) on \(\iop{m_1}{x_2} m_2\) from its local nilpotence on \(m_1\) and \(m_2\).
  Next consider \(v=J\)
  (and thus \(h=1\)) and 
  \(k=1\) in
  \eqref{eq:specjac} to obtain
  \begin{align}
    \brac{J_{n} - x_2 J_{n-1}} \iop{m_1}{x_2} m_2 =  \iop{m_1}{x_2} \brac{J_{n} - x_2 J_{n-1}} m_2  + \sum_{s = 0}^{n} \binom{s-n}{s} (-1)^s x_2^{n-s-1} \iop{J_{s+1} m_1}{x_2} m_2.
  \end{align}
  Since \(J_k,\ k\ge1\) is nilpotent on both \(m_1\) and \(m_2\), we
  see that \(J_{n}-x_2 J_{n-1}\) is nilpotent for \(n\ge 2\). Recall
  that the series expansion of the intertwining operator
  \begin{equation}
    \iop{m_1}{x_2}m_2 = \sum_{\substack{t \in\CC \\s \ge 0}} (m_1)_{(t,s)}m_2 x_2^{-t-1} (\log x_2)^s
  \end{equation}
  satisfies a lower truncation condition, that is, for fixed \(s\), if there exists a \(u\in \CC\)
  satisfying \(m_{(u,s)}\neq0\), then there exists a minimal representative
  \(t\in u+\ZZ\) such that \(m_{(t,s)}\neq0\) and \(m_{(t',s)}=0\) for
  all \(t'<t\). Since \(J_{n}-x_2 J_{n-1}\) is nilpotent on
  \(\iop{m_1}{x_2}m_2\) it is also nilpotent on the leading term
  \(m_{(t,s)}\). By comparing coefficients of \(x_2\) and \(\log x_2\)
  it then follows that \(J_n,\ n\ge 2\) acts nilpotently on
  \(m_{(t,s)}\) and by induction also on all coefficients of higher
  powers of \(x_2\). To show that \(J_1\) acts locally nilpotently,
  assume that \(m_1\) has \(J_0\)-eigenvalue
  \(j\) and set \(n=1, \ k=0\) in \eqref{eq:specjac}
  to obtain
  \begin{align}
    J_1 \iop{m_1}{x_2} m_2 =  \iop{m_1}{x_2} J_1 m_2 +
    x_2 j \iop{m_1}{x_2}m_2 + \sum_{s\ge 1}(-1)^s\binom{s-2}{s}x_2^{1-s}\iop{J_s m_1}{x_2} m_2.
  \end{align}
  Thus \(J_1-x_2 j\) is nilpotent, which by the previous leading
  term argument implies that \(J_1\) is too.
  Finally, consider \(v=\beta\)
  (and thus \(h=1\)) and 
  \(k=2\) in \eqref{eq:specjac} to obtain
  \begin{align}
    \brac{\beta_{n} - 2x_2 \beta_{n-1}+x_2^2 \beta_{n-2}} \iop{m_1}{x_2}
    m_2&=  \iop{m_1}{x_2} \brac{\beta_{n} - 2x_2 \beta_{n-1}+ x_2^2
         \beta_{n-2}} m_2  \nonumber\\
    &\qquad+ \sum_{s \ge 0} \binom{s-n+1}{s} (-1)^s x_2^{n-s-2} \iop{\beta_{s+2} m_1}{x_2} m_2.
  \end{align}
  By leading term arguments analogous to those used for \(J_n\), this
  implies that \(\beta_n\) acts locally nilpotently for \(n\ge 4\).
  
  Consider the subspace \(V\subset \Wmod\) annihilated by
  \(\beta_{n+3}\), \(\gamma_n\), \(n\ge1\). Then \(V\) is a module over four commuting copies of the Weyl
  algebra respectively generated by the pairs \((\beta_0,\gamma_0),\
  (\beta_1,\gamma_{-1}),\ (\beta_2,\gamma_{-2}),\
  (\beta_3,\gamma_{-3})\). Further, \(V\) is closed under the action
  of \(J_n,\ n\ge 1\) and restricted to acting on \(V\), the first few
  \(J_n\) modes expand as
  \begin{align}
    \label{eq:Jexp}
    J_3= \beta_3\gamma_0,\qquad J_2=\beta_2\gamma_0 +
    \beta_3\gamma_{-1},\qquad J_1= \beta_1\gamma_0+ \beta_2\gamma_{-1}+\beta_{3}\gamma_{-2}.
  \end{align}
  We show that on any composition factor of
  \(V\) at least three of the four Weyl algebras have a generator
  acting nilpotently and that thus the induction of such a composition
  factor is an object in one of the categories \(\sfmod{i}{\catR},\
  -3\le i\le 0\). Let \(C_0\otimes C_1\otimes C_2\otimes C_3\) be
  isomorphic to a
  composition factor of \(V\), where \(C_i\) is a simple module over
  the Heisenberg algebra generated by the pair
  \((\beta_i,\gamma_{-i})\). Since \(J_1,J_2,J_3\) act locally
  nilpotently on \(V\) they must also do so on \(C_0\otimes C_1\otimes
  C_2\otimes C_3\) using the expansions \eqref{eq:Jexp}. If we assume
  that neither \(\beta_3\) nor \(\gamma_0\) act locally nilpotently on
  \(C_3\) and \(C_0\), respectively, that is there exist \(c_3\in
  C_3\) and \(c_0\in C_0\) such that \(\UEA{\beta_3}c_3\)
  and \(\UEA{\gamma_0}c_0\) are both infinite dimensional, and choose
  \(c_1,c_2,\) to be non-zero vectors in \(C_1\) and \(C_2\),
  respectively. Then \(\UEA{J_3} (c_0\otimes c_1\otimes c_2\otimes
  c_3)\) will be infinite dimensional contradicting the local
  nilpotence of \(J_3\). So assume \(\beta_3\) acts locally
  nilpotently but \(\gamma_0\) does not, and let \(c_3\in C_3\) be
  annihilated by \(\beta_3\) and \(c_0,c_1,c_2\) be non-zero vectors
  in \(C_0,C_2,C_3\), respectively. On this vector \(J_2\) evaluates
  to
  \begin{equation}
    J_2 (c_0\otimes c_1\otimes c_2\otimes c_3)= \gamma_0 c_0\otimes
    c_1\otimes \beta_2 c_2\otimes c_3.
  \end{equation}
  By the same reasoning as before, unless either \(\beta_2\) or \(\gamma_0\)
  act nilpotently, we have a contradiction to the nilpotence of
  \(J_2\), so \(\beta_2\) must act nilpotently on \(c_2\). Repeating
  this argument for \(J_1\) and assuming \(\beta_2 c_2=0\) we have a
  contradiction to the nilpotence of \(J_1\) unless \(\beta_1\) acts
  nilpotently. The composition factor isomorphic to \(C_0\otimes
  C_1\otimes C_2\otimes C_3\) thus induces to an object in
  \(\catR\). Repeating the previous arguments, assuming that
  \(\gamma_0\) acts locally nilpotently but \(\beta_3\) does not,
  implies that \(\gamma_{-1}\) and \(\gamma_{-2}\) must act locally
  nilpotently to avoid contradictions to the local nilpotence of
  \(J_1,J_2,J_3\). Such a composition factor would induce to a module
  in \(\sfmod{-3}{\catR}\). Finally assume both \(\beta_3\) and
  \(\gamma_0\) act locally nilpotently, then analogous arguments to
  those used above applied to the action of \(J_1\) imply
  that at least one of \(\beta_2\) or \(\gamma_{-1}\) act locally
  nilpotently. Such a composition factor would induce to an object in
  \(\sfmod{-2}{\catR}\) or \(\sfmod{-1}{\catR}\), respectively.

  The final potential obstruction to \(\Wmod\)
  lying in
  \(\catFLE\) is that such a submodule might not be finite
  length. However, if \(\Wmod\) had infinite length, it would have
  to admit indecomposable subquotients of arbitrary finite length, yet by the
  classification of indecomposable modules in \cref{thm:classify},
  a finite length indecomposable module with composition factors only in
  \(\sfmod{i}{\catR},\ -3\le i\le 0\) has length at most 5. Therefore 
  \(\Wmod\in\catFLE\).

  Part \ref{itm:projfuscalc} follows by a similar but simplified version of the
  above arguments. \(J_n\) and \(\gamma_n\) continue to satisfy the
  same nilpotence conditions as above, however for \(\beta\) one
  needs to reconsider \eqref{eq:specjac} with \(k=1\) to conclude
  that \(\beta_n,\ n\ge2\) is nilpotent. The remainder of the
  argument follows analogously.
\end{proof}

\begin{proof}[Proof of \cref{thm:vtencat}.]
  We verify that the assumptions of \cref{thm:huang} hold, in numerical
  order. \cref{thm:huang} thus implies that category \(\catFLE\) is an additive braided tensor
  category. Additionally, since category \(\catFLE\) is abelian, it is a braided tensor category. 
  \begin{enumerate}
  	\item All modules in category \(\catFLE\) are strongly graded by ghost 
  	weight $j \in \RR$. Further, by \cref{thm:closurecomps}.\ref{itm:gradcomp}, all logarithmic intertwining operators are grading compatible.
  	\item By \cref{thm:duallist}, category $\catFLE$ is closed
  	under taking restricted duals. Closure under finite direct
  	sums holds by construction, since category \(\catFLE\) is abelian.
  	\item The modules in $\catFLE$ have real conformal weights by definition. The only modules on which the non semi-simple part of $L_0$ acts non-trivially are $\sfmod{m}{\Stmod{n}}$, for which it squares to zero.
  	\item Closure under images of module homomorphisms holds by
          construction, since category $\catFLE$ is abelian.
 	\item The convergence and extension properties hold by \cref{thm:conext}.
 	\item Since the \(P(w)\)-tensor product is right exact, by \cite[Part IV, Proposition
          4.26]{HuaLog}, and since category \(\catFLE\) has sufficiently many 
          projectives, that is, every module can be realised as a quotient of a direct
          sum of indecomposable projectives, we can without loss of generality
          assume \(\Mmod_1\) and \(\Mmod_2\) are indecomposable projective
          modules, as Condition \ref{itm:cycmods} holding for projective
          modules implies that it also holds for their quotients. Further, due
          to the compatibility of fusion with spectral
          flow, we can pick \(\Mmod_1\) and \(\Mmod_2\) to be isomorphic to
          \(\Typ{\lambda}\) or \(\Stag\).
          Let \(\nu\in \comp{\Mmod_1}{\Mmod_2}\) be doubly homogenous and
            assume that the module \(\Mmod_\nu\) generated by \(\nu\) is lower
            bounded. By assumption, the functional \(\nu\) therefore satisfies all the
            properties of \(P(w)\)-local 
            grading restriction except for the
            finite dimensionality of the doubly homogeneous spaces of
            \(\Mmod_\nu\). We need to show the finite dimensionality of these
            doubly homogeneous spaces and that \(\Mmod_\nu\) is an object
            in \(\catFLE\). Since \(\Mmod_\nu\) is finitely generated (cyclic
            even) it is at most a finite direct sum.
            To see this, assume the module admits an infinite
            direct sum. Then the partial sums define an ascending filtration
            whose union is the entire module. Hence after some finite number of
            steps all generators must appear within this filtration, but if this
            finite sum contains all generators, it must be equal to the entire
            module and hence all later direct summands must be zero.
            Denote the direct summands by \(\Mmod_{\nu,i}, \ i \in I\), where
            \(I\) is some finite index set.
            By \cite[Part IV, Proposition 5.24]{HuaLog} there exists a
            smooth 
            \(\bgva\) module \(\Wmod_{\nu,i}\) such that
            \(\Wmod_{\nu,i}^\prime\cong \Mmod_{\nu,i}\) and a surjective intertwiner
            of type \(\binom{\Wmod_{\nu,i}}{\Mmod_1,\ \Mmod_2}\). Hence, by
            \cref{thm:closurecomps}, \(\Wmod_{\nu,i}\in\catFLE\). In particular,
            since category \(\catFLE\) is closed under taking restricted duals and
            all its objects have finite dimensional
            doubly homogeneous spaces, we have \(\Mmod_{\nu,i}\in \catFLE\)
            and \(\Mmod_\nu\in\catFLE\).
  \end{enumerate}
\end{proof}

\begin{rmk}
  Note that the above proof did not make any use of \(\Mmod_\nu\) being lower bounded to conclude that
  \(\Mmod_\nu\in\catFLE\) and that membership of category \(\catFLE\) implies
  lower boundedness.
\end{rmk}

\begin{lem}\label{thm:stagfusion}
  For \(\lambda\in \RR/\ZZ\), \(\lambda \neq \ZZ,\) the fusion product
  \(\sfmod{\ell}{\Typ{\lambda}} \fuse \sfmod{k}{\Typ{-\lambda}}\) has exactly
  one direct summand isomorphic to \(\sfmod{\ell+k-1}{\Stag}\).
\end{lem}

We will prove the above lemma by showing that
  \( \Typ{\lambda} \ddfuse \Typ{-\lambda}\) has
  exactly one submodule isomorphic to \(\Stag\). This requires finding linear
  functionals which satisfy \(P(w)\) compatibility. This is very difficult
  to do in practice, since \eqref{eq:pzcomp} needs to be checked for every
  vector \(v\in V\). Fortunately there
  is a result by Zhang which cuts this down to generators. Zhang originally
  formulated the theorem below for a related type of fusion product called the
  \(Q(z)\)-tensor product, so we have translated his result to the
  \(P(w)\)-tensor product, which we use here.
  
\begin{thm}[{Zhang \cite[Theorem 4.7]{ZhaHLZgen08}}]
  Let \(A\le B\) be abelian groups.
  Let \(\VOA{V}\) be a \va{} graded by \(A\) with a set of strong generators
  \(S\) and let \(\Mmod_1\) and 
  \(\Mmod_2\) be modules over \(\VOA{V}\), graded by \(B\). A functional \(\psi\in
  \Hom(\Mmod_1\otimes\Mmod_2,\mathbb{C})\) is said to satisfy the \emph{strong lower
  truncation condition for a vector \(v\in \VOA{V}\)}, if there exists an \(N\in
  \NN\) such that for all \(n,m\in \ZZ\), with \(m\ge N\), we have
  \begin{equation}
    vt^{m+n}(t^{-1}-w)^n\psi=0.
  \end{equation}
  Then \(\psi\in
  \Hom(\Mmod_1\otimes\Mmod_2,\mathbb{C})\) satisfies the
  \(P(w)\)-compatibility condition if and only if it satisfies the strong lower
  truncation condition for all elements of \(S\).
  \label{thm:pzcompatgens}
\end{thm}

We further prepare some helpful identities.
\begin{lem}
  Let \(\Mmod_1, \Mmod_2\in\catFLE\), \(m_i\in \Mmod_i,\ i=1,2\), and
  \(\psi\in\comp{\Mmod_1}{\Mmod_2}\), then we have the identities.
\begin{align}
  \ang*{J_n\psi,m_1\otimes m_2}&=\delta_{n,0}\ang*{\psi,m_1\otimes
                                 m_2}-\sum_{i\ge0}\binom{-n}{i}w^{-n-i}\ang*{\psi,J_i
                                 m_1\otimes m_2}-\ang*{\psi, m_1 \otimes
     J_{-n}m_2},\quad n\in\ZZ,\label{eq:rhwrels1}\\
  \ang*{L_0\psi,m_1\otimes m_2}&=\ang*{\psi,L_0m_1\otimes m_2}+w\ang*{\psi,L_{-1}m_1\otimes m_2}+\ang*{\psi, m_1 \otimes L_0 m_2},\label{eq:rhwrelsL0}\\
    \ang*{\beta t^{k+n}(t^{-1}-w)^{n}\psi,m_1\otimes m_2}&=-\sum_{i\ge0}\binom{-k-n}{i}w^{-k-n-i}\ang*{\psi,
       \beta_{n+i}m_1\otimes m_2} \nonumber \\ &\qquad -\sum_{i\ge 0}
       \binom{n}{i}(-w)^{n-i}\ang*{\psi,m_1\otimes \beta_{i-k-n}m_2},\quad
 k,n\in\ZZ,                                                          \label{eq:rhwrels3}\\
 \ang*{\gamma t^{k+n}(t^{-1}-w)^{n}\psi,m_1\otimes m_2}&=\sum_{i\ge0}\binom{-k-n-2}{i}w^{-k-n-2-i}\ang*{\psi,
	\gamma_{n+i+1}m_1\otimes m_2} \nonumber \\ & \qquad +\sum_{i\ge0}
\binom{n}{i}(-w)^{n-i} \ang*{\psi,m_1\otimes
	\gamma_{i-k-n-1}m_2},\quad k,n\in\ZZ.
\label{eq:rhwrels4}
\end{align}
\end{lem}
\begin{proof}
  These identities follow by evaluating \eqref{eq:ddaction} for the fields
  \(\beta,\gamma, J\) and \(T\).
\end{proof}

\begin{proof}[Proof of \cref{thm:stagfusion}]
  We shall use the HLZ double dual construction of \cref{def:HLZddual}.
  By the compatibility of fusion with spectral flow, \cref{thm:specflow}, it
  is sufficient to consider the case \(\ell=k=0\).
  Note since \(\sfmod{-1}{\Stag}\) is both projective and injective, it must
  be a direct summand if it appears as either a quotient or a subspace.
  Further, by \cref{thm:closurecomps}, all composition factors must lie in
  categories \(\sfmod{i}{\catR},\ i=-1,0\). This implies that the
  composition factors of \(\Typ{\lambda}\ddfuse \Typ{-\lambda}\) must all lie in
  \(\sfmod{i}{\catR},\ i=0,1\).
  Note further, that \((\sfmod{-1}{\Stag})'\cong \Stag\) and so we seek to
  find a copy of \(\Stag\) within \(\Typ{\lambda}\ddfuse \Typ{-\lambda}\). We do so
  by considering a certain characterising two dimensional subspace of
  \(\Stag\).
  For a \(\bgva\)-module \(\Mmod\) consider the subspace
  \begin{equation}
    K(\Mmod)=\set{m\in\Mmod\st \beta_{n}m=\gamma_{n+1}m=J_1m=J_0m=0,\ n\ge1}.
  \end{equation}
  From the expansions of \(T(z)\) and \(J(z)\) in terms of the fields \(\beta\)
  and \(\gamma\), it follows that for any \(m\in K(\Mmod)\), \(L_0^2 m=L_{n}m=J_n m=0,\
  n\ge1\). In particular, in the notation of \cref{fig:weights},
  \(K(\Stag)=\spn{\ket{0},\ket{-\psi}}\) and thus \(K(\Stag)\) is two
  dimensional and \(L_0\) has a rank 2 Jordan block of generalised eigenvalue 0
  on this space. Further, \(\Stag\) is the only indecomposable module with
  composition factors in categories \(\sfmod{i}{\catR}$, $i=0, 1\)
  admitting \(L_0\) Jordan blocks. The remaining indecomposable modules with
  composition factors in categories \(\sfmod{i}{\catR}$, $i=0, 1\) all
  have \(K(\Mmod)\) subspaces of dimension zero or one.

 Let \(\psi\in \Hom\brac*{\Typ{\lambda}\otimes\Typ{-\lambda},\CC}\) satisfy
  \(\beta t^{k+n}(t^{-1}-w)^n\psi=\gamma t^{k+n}(t^{-1}-w)^n\psi=0\) for all
  \(m\ge1\). Thus by \cref{thm:pzcompatgens}, \(\psi\) satisfies the
  \(P(w)\)-compatibility property and \(\beta_m\psi=\gamma_{m+1}\psi=0,\
  m\ge1\). If in addition \(\psi\) is doubly 
  homogeneous, then \(\psi\) lies in \(\Typ{\lambda}\ddfuse \Typ{-\lambda}\). By assumption
  the \lhs{s} of \eqref{eq:rhwrels3} and \eqref{eq:rhwrels4} vanish for $k\ge1$. These
  relations imply that the value of \(\psi\) on any vector in
  \(\Typ{\lambda}\otimes\Typ{-\lambda}\) is determined by its value on tensor products of
  relaxed highest weight vectors, because negative modes on one factor can be
  traded for less negative modes on the other factor. For example, for
  \(k=1, \ n=0\) in \eqref{eq:rhwrels3}, we have
  the relation
  \begin{equation}
    \ang{\psi, m_1\otimes
      \beta_{-1}m_2}=-\sum_{i\ge0}\binom{-1}{i}w^{-1-i}\ang*{\psi,\beta_i m_1\otimes m_2}.  
  \end{equation}
  Let \(u_{\pm j}\in\Typ{\pm\lambda},\ j\in\pm\lambda\) be a choice of
  normalisation of relaxed highest weight vectors satisfying
  \(u_{\pm j-1}=\gamma_0 u_{\pm j}\). This implies \(\beta_0 u_{\pm j}=\pm j
  u_{1\pm j}\).
  Since the negative \(\beta\) and \(\gamma\) modes act freely on the simple
  projective modules \(\Typ{\lambda}\) and \(\Typ{-\lambda}\), there are no
  relations in addition to those coming from 
  \(\beta t^{k+n}(t^{-1}-w)^n\psi=\gamma t^{k+n}(t^{-1}-w)^n\psi=0\) for all
  \(m\ge1\) .
  Thus there is a linear isomorphism
  \begin{equation}
    \set{\psi\in \Typ{\lambda}\ddfuse \Typ{-\lambda}\st
      \beta_n\psi=\gamma_{n+1}\psi=0,\ n\ge1}\overset{\cong}{\longrightarrow}
    \Hom\brac*{\spn{u_j\otimes u_{-i}},\CC}.
  \end{equation}
 Clearly, \(K\brac*{\Typ{\lambda}\ddfuse \Typ{-\lambda}}\) is a
  subspace of \(\set{\psi\in \Typ{\lambda}\ddfuse \Typ{-\lambda}\st
      \beta_n\psi=\gamma_{n+1}\psi=0,\ n\ge1}\) and so we impose the remaining
    two relations, the vanishing of \(J_0\) and \(J_1\), via
    \eqref{eq:rhwrels1}. The vanishing of \(J_0\psi\) implies
    \begin{equation}
      0=\ang*{J_0\psi,u_{j}\otimes u_{-i}}=\ang*{\psi,u_j\otimes
        u_{-i}}-\ang*{\psi,J_0 u_j\otimes u_{-i}}-\ang*{\psi,u_j\otimes J_{0}u_{-i}}
      =(1-j+i)\ang*{\psi,u_j\otimes u_{-i}}.
    \end{equation}
    Thus \(\psi\) vanishes on \(u_j\otimes u_{-i}\) unless \(i=j-1\). The
    vanishing of \(J_1\psi\) implies 
    \begin{equation} \label{eq:gamma0}
      (2j-1)\ang*{\psi,u_{j}\otimes u_{1-j}}-j\ang*{\psi,u_{j+1}\otimes u_{-j}}+(1-j)\ang*{\psi,u_{j-1}\otimes u_{2-j}}=0,
    \end{equation}
    where we have used \(J_{-1}u_{1-j}=\brac*{\gamma_{-1}\beta_0+\beta_{-1}\gamma_0}u_{1-j}\).
    Thus \(\psi\) is completely characterised by its value on a two
    pairs of relaxed highest weight vectors, say \(u_{j}\otimes u_{1-j}\) and \(u_{j+1}\otimes u_{-j}\).
    Therefore, the subspace \(K\brac*{\Typ{\lambda}\ddfuse\Typ{-\lambda}}\)
    is two dimensional. Next we show that
    that \(L_0\) has a rank two Jordan block on it when acting on this space.
    Let \(\psi\in
    K\brac*{\Typ{\lambda}\ddfuse\Typ{-\lambda}}\). If \(\psi\neq0\), then there exist
    \(a,b\in\CC\), not both zero, such that
    \begin{equation}
      \ang*{\psi,u_{j}\otimes u_{1-j}}=a,\qquad \ang*{\psi,u_{j+1}\otimes u_{-j}}=b.
    \end{equation}
    The evaluation of \(L_0\psi\) on \(u_{j}\otimes u_{1-j}\) and
    \(u_{j+1}\otimes u_{-j}\) is then
    \begin{equation}
      \ang*{L_0\psi,u_{j}\otimes u_{1-j}}=j(a-b),\qquad \ang*{L_0\psi,u_{j+1}\otimes u_{-j}}=-j(a-b).
    \end{equation}
    Therefore if \(a\neq b\) (a choice which we can make as $K\brac*{\Typ{\lambda}\ddfuse\Typ{-\lambda}}$ is two dimensional), the vectors \(\psi\) and \(L_0\psi\) are linearly
    independent and span \(K\brac*{\Typ{\lambda}\ddfuse\Typ{-\lambda}}\), which also
    shows that \(L_0\) has a rank two Jordan block.
  \end{proof}

\begin{rmk}\leavevmode
  In \cite[Section 7]{RidBos14} the above fusion product was computed using
  the NGK algorithm up to certain conjectured additional
  conditions. 
  In light of the 
  survey \cite{KanRid18} explaining the equivalence of
  the HLZ double dual construction and the NGK algorithm, the
  authors thought it appropriate to supplement the NGK calculation of
  \cite{RidBos14} with an HLZ double dual calculation here.
\end{rmk}

\begin{prop}\label{thm:rigidprojsimples}
  For all \(\ell\in\ZZ\) and \(\lambda\in \RR/\ZZ\), \(\lambda\neq\ZZ\), the simple module
  \(\sfmod{\ell}{\Typ{\lambda}}\) is rigid in category \(\catFLE\),
  with tensor dual given by \(\brac*{\sfmod{\ell}{\Typ{\lambda}}}^\vee=\sfmod{1-\ell}{\Typ{-\lambda}}\).
\end{prop}
\begin{proof}
  Recall that an object \(\Mmod\) in a tensor category is \emph{rigid} if
  there exists an object \(\Mmod^\vee\) (called a tensor dual of \(\Mmod\)) and two morphisms $e_\Mmod:  \Mmod^\vee
  \fuse \Mmod \ra \VacMod$ and $i_\Mmod: \VacMod \ra \Mmod \fuse \Mmod^\vee$,
  respectively, called evaluation and coevaluation, such that
  the compositions
  \begin{subequations}\label{eq:rigid}
    \begin{align} 
      &{\Mmod \cong \VacMod \fuse \Mmod \overset{i_\Mmod \otimes 1}{\lra}
        \brac*{\Mmod \fuse_{w_2} \Mmod^\vee } \fuse_{w_1} \Mmod
        \overset{\mathcal{A}^{-1}
        }{\lra} \Mmod \fuse_{w_2} \brac*{ \Mmod^\vee  \fuse_{w_1} \Mmod } \overset{1 \otimes e_\Mmod}{\lra} \Mmod \fuse \VacMod \cong \Mmod},\label{eq:rigid1}\\
      &{\Mmod^\vee \cong  \Mmod^\vee \fuse \VacMod \overset{1 \otimes
        i_\Mmod}{\lra} \Mmod^\vee \fuse_{w_2} \brac*{ \Mmod \fuse_{w_1}
        \Mmod^\vee } \overset{\mathcal{A}
        }{\lra} \brac*{ \Mmod^\vee \fuse_{w_2} \Mmod } \fuse_{w_1} \Mmod^\vee \overset{e_\Mmod \otimes 1}{\lra} \VacMod \fuse \Mmod^\vee \cong \Mmod^\vee},\label{eq:rigid2}
    \end{align}
  \end{subequations}
  yield the identity maps $1_\Mmod$ and $1_{\Mmod^\vee}$, respectively. Here
  \(w_1,w_2\) are distinct non-zero complex numbers satisfying
  \(|w_2|>|w_1|\) and \(|w_2|>|w_2-w_1|\); \(\fuse_w\) indicates
  the relative positioning of insertion points of fusion factors, that
  is, the right most factor will be inserted at 0, the middle factor
  at \(w_1\) and the left most at \(w_2\); 
  Technically there exist distinct notions of left and right duals and the above
  properties are those for left duals. We prove below that
  \(\Mmod=\sfmod{\ell}{\Typ{\lambda}}\) is left rigid. Right rigidity follows
  from left rigidity due to category \(\catFLE\) being braided.

  For $\Mmod = \sfmod{\ell}{\Typ{\lambda}}$ we take the tensor dual to be
  $\Mmod^\vee = \sfmod{1-\ell}{\Typ{-\lambda}}$ and
  we will construct the evaluation and coevaluation morphisms using
  the first free field realisation \eqref{eq:bosonisation} given in \cref{thm:ffr}.\ref{itm:ffr1}.
  In particular, we have
  \begin{equation}
     \sfmod{\ell}{\Typ{\lambda}} \cong \LFock{\lambda(\theta+\psi) +
       (\ell-1)\psi}, \qquad \sfmod{1-\ell}{\Typ{-\lambda}} \cong
     \LFock{-\lambda(\theta+\psi) - \ell \psi}, \qquad
     \ell\in\ZZ,\ \lambda\in \RR/\ZZ,\ \lambda\neq\ZZ.
  \end{equation}
  We denote fusion over the lattice \va{} \(\latva\) of the
  free field realisation by \(\ffuse\) to distinguish it from fusion
  over \(\bgva\). Recall that the fusion product of Fock spaces over the lattice \va{}
  \(\latva\) of the free field realisation just adds Fock space weights. 
  Thus the fusion product over \(\latva\) of the modules corresponding to
  \(\sfmod{\ell}{\Typ{\lambda}}\) and \(\sfmod{1-\ell}{\Typ{-\lambda}}\)
  is given by
  \begin{equation}
    \LFock{-\lambda(\theta+\psi) - \ell \psi} \ffuse \LFock{\lambda(\theta+\psi)+ (\ell-1) \psi}  \cong \LFock{-\psi} \cong \Typ{0}^-.
  \end{equation}
  Therefore we have the \(\latva\)-module map $\mathcal{Y}:
   \LFock{-\lambda(\theta+\psi) - \ell \psi}  \ffuse
   \LFock{\lambda(\theta+\psi) +(\ell - 1) \psi}\ra \LFock{-\psi}$ given
  by the intertwining operator that maps the kets in the Fock space
  \(\LFock{\lambda(\theta+\psi) +(\ell - 1) \psi}\) to vertex operators, that is, operators of the form \eqref{eq:vopformula}. Since
  \(\latva\)-module maps are also \(\bgva\)-module maps by restriction
  and since the fusion product of two modules over a vertex subalgebra
  is a quotient of the fusion product over the larger \va{},
  \(\mathcal{Y}\) also defines a \(\bgva\)-module map
  \( \LFock{-\lambda(\theta+\psi) - \ell \psi}\fuse
  \LFock{\lambda(\theta+\psi)+ (\ell-1) \psi} 
   \ra \LFock{-\psi}
  \cong \Typ{0}^-\). Furthermore, the screening operator
  \(\scr{1}=\oint \vop{\psi}{z}\dd z\)
  defines a \(\bgva\)-module map $\scr{1}: \LFock{-\psi} \ra
  \LFock{0} $ with the image being the bosonic ghost \va{}
  \(\bgva\). Up to a normalisation factor, to be determined later, we define the
  evaluation map for \(\Mmod=\sfmod{\ell}{\Typ{\lambda}}\) to be the
  composition of \(\mathcal{Y}\) and the screening operator \(\scr{1}\).
  \begin{equation}
    e_\Mmod = \scr{1} \circ \mathcal{Y}: \Mmod^\vee \fuse \Mmod \ra \VacMod .
  \end{equation}
 To define the coevaluation we need to identify a submodule of \(\Mmod\fuse
  \Mmod^\vee\) isomorphic to \(\VacMod\). By \cref{thm:stagfusion}, we know
  that \(\Mmod\fuse \Mmod^\vee\) has a direct summand isomorphic to \(\Stag\),
  which by \cref{thm:projconstr} we know has a submodule isomorphic to
  \(\VacMod\). It is this copy of \(\VacMod\) which the coevaluation shall map to.
  Since \(\VacMod\) is the vector space underlying the \va{} \(\bgva\) and any
  \va{} is generated from its vacuum vector, we 
  characterise the coevaluation map by the image of the vacuum vector.
  \begin{equation}
    i_M:\ \vac \lra \ket{0} \overset{\scr{1}^{-1}}{\lra} \ket{-\psi}
    \overset{}{\lra} \vop{(j-1)\psi+(j-\ell)\theta}{w}\ket{-j\psi-(j-\ell)\theta}
    \overset{\scr{1}}{\lra}
    \oint_{w} \scr{1}\brac*{z} \vop{(j-1)\psi+(j-\ell)\theta}{w}\ket{-j\psi-(j-\ell)\theta} \dd z ,
  \end{equation}
  where the first arrow is the inclusion of \(\VacMod\) into \( \LFock{0} \cong \Typ{0}^- \subset \Stag\), \(\scr{1}^{-1}\) denotes picking
  preimages of \(\scr{1}\) and \(j\) the unique representative of the coset
  \(\lambda\) satisfying \(0< j< 1\). Note that the ambiguity of
  picking preimages of \(\scr{1}\) in the second arrow is undone by
  reapplying \(\scr{1}\) in the fourth arrow and hence the map is
  well-defined. This map maps to \(\LFock{0}\), which is a
  submodule of \(\Stag\) as shown in \cref{thm:projconstr}.

  Note that since the modules \(\Mmod\) and \(\Mmod^\vee\) considered here are
  simple, the compositions of coevaluations and evaluations
  \eqref{eq:rigid} are proportional to the identity by Schur's
  lemma. Rigidity therefore follows, if we can show that the
  proportionality factors for \eqref{eq:rigid1} and \eqref{eq:rigid2}
  are equal and non-zero.

  We determine the proportionality factor for \eqref{eq:rigid1} by
  applying the map to the ket \(\ket{(j-1)\psi+(j-\ell)\theta}\in\LFock{\lambda(\psi+\theta)+(\ell-1)\theta}\cong\sfmod{\ell}{\Typ{\lambda}}\).
  Following the sequence of maps in \eqref{eq:rigid1} we get
\begin{gather}
  \ket{(j-1)\psi+(j-\ell)\theta}\ra
  \ket{0}\fuse \ket{(j-1)\psi+(j-\ell)\theta}\ra
  \oint_{w_1,w_2}
  \scr{1}\brac*{z}\vop{(j-1)\psi-(j-\ell)\theta}{w_2}\vop{-j\psi-(j-\ell)\theta}{w_1}\ket{(j-1)\psi-(j-\ell)\theta}\dd
  z\nonumber\\
  \ra \oint_{0,w_1} \oint_{w_1,w_2} \scr{1}\brac*{z_2} \scr{1}\brac*{z_1} \vop{(j - 1) \psi + (j-\ell) \theta}{w_2} \vop{-j\psi - (j-\ell) \theta}{w_1} \ket{(j -1) \psi + (j-\ell) \theta} \dd z_1 \dd z_2,
\end{gather}
  where \(\oint_{0,w_2}\) denotes a contour about \(0\) and \(w_2\)
  but not \(w_1\), \(\oint_{w_1,w_2}\) denotes a contour about \(w_1\) and \(w_2\)
  but not \(0\).
  The proportionality factor is obtained by pairing the above with the
  dual of the Fock space highest weight vector, which we denote by an empty bra
  \(\bra{}\),
  and thus equal to the matrix element
  \begin{align}
    I(w_1,w_2)  &= \oint_{0,w_1} \oint_{w_1,w_2} \bra{} \scr{1}\brac*{z_2} \scr{1}\brac*{z_1} \vop{(j - 1) \psi + (j-\ell) \theta}{w_2} \vop{-j\psi - (j-\ell) \theta}{w_1} \ket{(j -1) \psi + (j-\ell) \theta} \dd z_1 \dd z_2 \nonumber\\
       &= f(w_1, w_2)\oint_{0,w_1} \oint_{w_1,w_2} (z_2 - z_1) z_{2}^{j -1}  (z_2 - w_2)^{j - 1} (z_2 - w_1)^{-j}  z_1^{j -1}(z_1 - w_2)^{j -1} (z_1 - w_1)^{-j}    \dd z_1 \dd z_2 \nonumber\\
       &= f(w_1, w_2) \bigg( \oint_{0,w_1} z^j (z - w_2)^{j-1}
         (z - w_1)^{-j} \dd z \oint_{w_1,w_2} z^{j-1} (z -
         w_2)^{j-1} (z - w_1)^{-j} \dd z  \nonumber\\
       & \qquad \qquad \qquad -  \oint_{0,w_1} z^{j-1} (z - w_2)^{j-1} (z - w_1)^{-j} \dd z \oint_{w_1,w_2} z^{j} (z - w_2)^{j-1} (z - w_1)^{-j} \dd z \bigg),  \label{eq:example}
  \end{align}
 where
  \begin{equation}
    f(w_1, w_2) = (w_2 - w_1)^{\ell^2+j(1-2\ell)} w_{2}^{(j -1)(2j-\ell-1)} w_{1}^{\ell^2+j(1-2\ell)}. 
  \end{equation}
  Note that the second equality of \eqref{eq:example} is where the
  associativity isomorphisms are used to pass from compositions (or
    products) of vertex operators to their operator product expansions (also
    called iterates). For intertwining operators,
  associativity amounts to the analytic continuation of their series 
  expansions and then reexpanding in a different domain. On the \lhs{} of the
  second equality the intertwining operators (or here specifically vertex
  operators) are in radial ordering, while on the \rhs{} they have been
  analytically continued and then reexpanded as operator product expansions.
  By an analogous argument the proportionality factor produced by the
  sequence of maps \eqref{eq:rigid2} is the matrix element
  \begin{align}
    \tilde{I}(w_1,w_2)  &= \oint_{0,w_1} \oint_{w_1,w_2} \bra{}
                          \scr{1}\brac*{z_2} \scr{1}\brac*{z_1}
                          \vop{-j\psi - (j-\ell) \theta}{w_2}\vop{(j -
                          1) \psi + (j-\ell) \theta}{w_1} \ket{-j \psi
                          - (j-\ell) \theta} \dd z_1 \dd
                          z_2\nonumber\\
    &=f(w_1, w_2) \bigg( \oint_{0,w_1} z^j (z - w_2)^{j-1}
         (z - w_1)^{-j} \dd z \oint_{w_1,w_2} z^{j-1} (z -
         w_2)^{j-1} (z - w_1)^{-j} \dd z  \nonumber\\
       & \qquad \qquad \qquad -  \oint_{0,w_1} z^{j-1} (z - w_2)^{j-1} (z - w_1)^{-j} \dd z \oint_{w_1,w_2} z^{j} (z - w_2)^{j-1} (z - w_1)^{-j} \dd z \bigg).
  \end{align}
  Since both matrix elements are equal,
  \(I(w_1,w_2)=\tilde{I}(w_1,w_2)\), rigidity follows by showing that
  they are non-zero.

  We evaluate the four integrals appearing in \(I(w_1,w_2)\). We
  simplify the first integral using the substitution \(z=w_1x\).
  \begin{align}
    \oint_{0,w_1} z^j (z - w_2)^{j-1} (z - w_1)^{-j} \dd z 
	& = - w_{2}^{j-1} w_1 \oint_{0,1} x^j (1-x)^{-j} \left(1- \frac{w_1}{w_2}x \right)^{j-1} \dd x \nonumber\\
	& = - \left( e^{2\pi i j } -1 \right) w_{2}^{j - 1} w_1 \int_{0}^{1} x^{j} (1-x)^{-j} \left( 1- \frac{w_1}{w_2}x \right)^{j -1} \dd x \nonumber\\
	& = - \left( e^{2\pi i j } -1 \right) w_{2}^{j - 1} w_1 \ \bfn{1 + j}{1- j} \ \tFo{1-j}{1+ j}{2} {\frac{w_1}{w_2}},
  \end{align}
where the second equality follows by deforming the contour about 0 and
1 to a dumbbell or dog bone contour, whose end points vanish because
the contributions from the end points are $O(\varepsilon^{1+j})$ and
$O(\varepsilon^{1-j})$ respectively, and \(0<j<1\); and the third equality is the
integral representation of the hypergeometric function and \(B\) is
the beta function. Similarly,
  \begin{align}
    \oint_{0,w_1} z^{j-1} (z - w_2)^{j-1} (z - w_1)^{-j} \dd z 
	& = - \left( e^{2\pi i j } -1 \right) w_{2}^{j - 1} \ \mathrm{B}(j, 1- j) \ \tFo{1-j}{j}{1} {\frac{w_1}{w_2}}.
  \end{align}
For the integrals with contours about \(w_1\) and \(w_2\) we use the
substitution \(z=w_2-(w_2-w_1)x\) and then again obtain integral representations of
the hypergeometric function.
\begin{align}
  \oint_{w_1,w_2} z^{j-1} (z -
         w_2)^{j-1} (z - w_1)^{-j} \dd z&=(-1)^j\brac*{e^{2\pi i j}-1}
                                          w_2^{j-1}\bfn{j}{1-j} \tFo{1-j}{j}{1}{\frac{w_2-w_1}{w_2}},\nonumber\\
   \oint_{w_1,w_2} z^{j} (z -
         w_2)^{j-1} (z - w_1)^{-j} \dd z&=(-1)^j\brac*{e^{2\pi i j}-1}
                                          w_2^{j}\bfn{j}{1-j}  \tFo{-j}{j}{1}{\frac{w_2-w_1}{w_2}}.
\end{align}
Note that for the three integrals above, the end point contributions of the contour also vanish due to being $O(\varepsilon^{j})$ and $O(\varepsilon^{1-j})$ for 0 and 1 respectively.
  
Making use of the hypergeometric and beta function identities
  \begin{align}
    \tFo{1-\mu}{1+ \mu}{2} {\frac{w_2}{w_1}} &= \frac{w_1}{w_2} \ \tFo{-\mu}{ \mu}{1} {1 -\frac{w_2}{w_1}} ,\nonumber\\
    \tFo{1-\mu}{\mu}{1} {1-\frac{w_2}{w_1}} &= \tFo{1-\mu}{\mu}{1} {\frac{w_2}{w_1}} ,\nonumber\\
    \mathrm{B} (1+ \mu, 1-\mu) = \mu \ \mathrm{B}(\mu, 1-\mu) &= \frac{\pi \mu}{\sin(\pi \mu)},
  \end{align}
  the proportionality factor \(I(w_1,w_2)\) simplifies to
  \begin{equation}
    I(w_1,w_2) = (-1)^{j}f(w_1, w_2) \left( e^{2\pi i j } -1 \right)^2 w_{2}^{2j-1} \frac{\pi^2 (j-1) }{\sin(\pi j)^2} \ \tFo{-j}{ j}{1} {\frac{w_2-w_1}{w_2}} \ \tFo{1-j}{j}{1}{\frac{w_2}{w_1}} .
  \end{equation}
  Since $j \notin \ZZ$, \(I(w_1,w_2)\) can only vanish, if one of the
  hypergeometric factors does.
  We specialise the complex numbers \(w_1,w_2\), such that
  \(w_2=2w_1\). Then,
  \begin{equation}
    \tFo{1-j}{j}{1}{\frac{w_1}{w_2}} = \tFo{1-j}{j}{1}{\tfrac{1}{2}} = \frac{\Gam{\frac{1}{2}}\Gam{1}}{\Gam{1 - \frac{j}{2}}\Gam{\frac{1}{2} + \frac{j}{2}}} \neq 0 ,
  \end{equation}
  and the relationship between contiguous functions implies
  \begin{align}
    \tFo{-j}{j}{1} {\frac{w_2-w_1}{w_2}} &= \tfrac{1}{2} \left( \tFo{1-j}{ j}{1} {\tfrac{1}{2}} + \tFo{-j}{1+ \mu}{1} {\tfrac{1}{2}}  \right)\\ 
                                             & = \frac{\Gam{\frac{1}{2}}\Gam{1}}{\Gam{1 - \frac{j}{2}}\Gam{\frac{1}{2} + \frac{j}{2}}}
                                               \neq 0.
  \end{align}
  Thus \(I(w_1,w_2)\neq 0\) and we can rescale the evaluation map by
  \(I(w_1,w_2)^{-1}\) so that the sequences of maps \eqref{eq:rigid}
  are equal to the identity maps on \(\Mmod\) and \(\Mmod^\vee\). Thus
  \(\sfmod{\ell}{\Typ{\lambda}}\) is rigid.
\end{proof}

\section{Fusion product formulae}
\label{sec:fusion}

In this section we determine the decomposition of all fusion products
in category \(\catFLE\). A complete list of fusion products among
representatives of each spectral flow orbit is collected in
\cref{thm:fusionlist}, while the proofs of these decomposition formulae
have been split into the dedicated Subsections \ref{sec:simplefusion}
and \ref{sec:fusindecomp}. To simplify some of the decomposition
formulae we introduce dedicated notation for certain sums of spectral
flows of the projective module \(\Stag\). Consider the polynomial of
spectral flows
\begin{equation}
f_n(\sfaut)=\sum_{k=1}^n \sfaut^{2k-1},\qquad n\in\NN,
\end{equation}
and let
\begin{equation}
\Stsum{n}{}=f_n(\sfaut)\Stag = \bigoplus_{k=1}^n\Stmod{2k-1}, \qquad n\in\NN.
\end{equation}
Further, let
\begin{equation} \label{eq:polys}
\Stsum{n}{k}=\sfmod{k}{\Stsum{n}{}}, \qquad
\Stsum{m,n}{k}=\sfaut^{k-1} f_m(\sfaut)\Stsum{n}{}
=\bigoplus_{r=1}^{m+n-1} N_{r}\; \Stmod{k + 2r-1},
        \quad N_{r} = \min\{r,m,n,m+n-r\},
\qquad
m,n\in \NN, \quad k \in \ZZ.
\end{equation}

\begin{thm}\label{thm:fusionlist}\leavevmode
  \begin{enumerate}
    \item Category \(\catFLE\) under fusion is a rigid braided tensor
      category.
      \label{itm:rigid}
    \item The following is a list of all non-trivial fusion products, those
      not involving the fusion unit (the vacuum module \(\VacMod\)), in
      category \(\catFLE\) among representatives for each 
      spectral flow orbit. All other fusion products are determined from
      these through spectral flow and the compatibility of spectral flow with
      fusion as given in \cref{thm:specflow}.
      
      Since \(\catFLE\) is rigid, the fusion product of a projective module \(\module{R}\) with any
      indecomposable module \(\module{M}\) is given by
\begin{equation}
  \module{R} \fuse \module{M} \cong \bigoplus_{\module{S}} [\module{M}:\module{S}]\; \module{R}\fuse\module{S} ,
\end{equation}
where the summation index runs over all isomorphism classes of
composition factors of \(\module{M}\) and \([\module{M}:\module{S}]\) is the multiplicity
of the composition factor $\module{S}$ in \(\module{M}\).

For all \(\lambda,\mu\in \RR/\ZZ\), $\lambda, \mu, \lambda + \mu \neq \ZZ$, 
  \begin{align}\label{eq:simpleprojfuse}
    \begin{split}
      {\Typ{\lambda}} \fuse {\Typ{\mu}} &\cong {\Typ{\lambda + \mu}} \oplus \sfmod{-1}{\Typ{\lambda + \mu}},\\
      {\Typ{\lambda}} \fuse {\Typ{-\lambda}} &\cong \sfmod{-1}{\Stag}.
    \end{split}
  \end{align}
For $m, n \in \ZZ$, $m \ge n$, such that the lengths of indecomposables below
are positive, we have the following fusion product formulae.
\begin{subequations}
\begin{align}
  \Bmod{2m+1}{} \fuse \Bmod{2n+1}{} &\cong \Bmod{2m + 2n +1}{} \oplus
                                      \Stsum{m,n}{1}&
  \Tmod{2m+1}{} \fuse \Tmod{2n+1}{} &\cong \Tmod{2m + 2n +1}{} \oplus \Stsum{m,n}{1}\nonumber\\
  \Bmod{2m+1}{} \fuse \Bmod{2n}{} &\cong \Bmod{2n}{} \oplus \Stsum{m,n}{1}&
  \Tmod{2m+1}{} \fuse \Tmod{2n}{} &\cong \Tmod{2n}{} \oplus \Stsum{m,n}{1}\nonumber\\
  \Bmod{2m}{} \fuse \Bmod{2n}{} &\cong \Bmod{2n}{2m-1} \oplus \Bmod{2n}{}
                                  \oplus \Stsum{m-1,n}{1}&
  \Tmod{2m}{} \fuse \Tmod{2n}{} &\cong \Tmod{2n}{2m-1} \oplus \Tmod{2n}{} \oplus \Stsum{m-1,n}{1} \label{eq:indfus1}\\
  \Tmod{2m+1}{} \fuse \Bmod{2n+1}{} &\cong \Tmod{2m-2n+1}{2n}  \oplus
                                      \Stsum{m+1,n}{}&
  \Bmod{2m+1}{} \fuse \Tmod{2n+1}{} &\cong \Bmod{2m-2n+1}{2n}  \oplus \Stsum{m+1,n}{}\nonumber\\
  \Tmod{2m}{} \fuse \Bmod{2n+1}{} &\cong \Tmod{2m}{2n}  \oplus \Stsum{m,n}{}&
  \Bmod{2m}{} \fuse \Tmod{2n+1}{} &\cong \Bmod{2m}{2n}  \oplus \Stsum{m,n}{} \nonumber\\
  \Tmod{2m}{} \fuse \Bmod{2n}{} &\cong \Stsum{m,n}{}&\Bmod{2m}{} \fuse \Tmod{2n}{} &\cong \Stsum{m,n}{}
                                                                                     \label{eq:indfus2}
\end{align}
\end{subequations}
\end{enumerate}
\end{thm}
We split the proof of \cref{thm:fusionlist} into multiple
parts. \cref{thm:fusionlist}.\ref{itm:rigid} is shown in
\cref{thm:rigid}. The fusion formulae \eqref{eq:simpleprojfuse},\eqref{eq:indfus1},\eqref{eq:indfus2} are
determined in
\cref{thm:nonstagfusion,thm:stagfusion,thm:pureindfus,thm:mixindfus}, respectively.

\begin{rmk}
  The fusion product formulae of \cref{thm:fusionlist} projected onto the
  Grothendieck group match the conjectured Verlinde formula of
  \cite[Corollaries 7 and 10]{RidBos14}, thereby proving that category
  \(\catFLE\) satisfies the standard module formalism version of the Verlinde
  formula. It will be an interesting future problem to find a more conceptual and
  direct proof for the validity of the Verlinde formula, rather than a proof by
  inspection.
\end{rmk}

\subsection{Fusion products of simple projective modules}
\label{sec:simplefusion}

In this section we determine the fusion products of the simple
projective modules.

\begin{prop}\label{thm:nonstagfusion}
  For \(\lambda,\mu\in \RR/\ZZ\), \(\lambda,\mu,\lambda + \mu \notin \ZZ\), we have
  \begin{equation}
    {\Typ{\lambda}} \fuse {\Typ{\mu}} \cong {\Typ{\lambda + \mu}} \oplus \sfmod{-1}{\Typ{\lambda + \mu}}. 
  \end{equation}
\end{prop}
\begin{proof}
  Since \(\Typ{\lambda}\) and \(\Typ{\mu}\) both lie in category
  \(\catR\), we know, by \cref{thm:closurecomps},
  that the composition factors of the fusion product lie in
  categories \(\catR\) or \(\sfmod{-1}{\catR}\). Further, since
  \(J(z)\) is a conformal weight 1 field, its corresponding weight, the ghost
  weight, adds under fusion. Therefore the only possible composition
  factors are \(\Typ{\lambda+\mu}\) and
  \(\sfmod{-1}{\Typ{\lambda+\mu}}\). Since these composition factors
  are both projective and injective, they can only appear as direct
  summands and all that remains is to determine their multiplicity.
  In \cite{AdaBG19} Adamovi{\'c} and Pedi{\'c} computed dimensions of spaces
  of intertwining operators for fusion products of the simple
  projective modules. In particular, \cite[Corollary 6.1]{AdaBG19} states that
  \begin{equation}
    \label{eq:dimintops}
    \dim\binom{\module{M}}{\Typ{\lambda},\ \Typ{\mu}}=1,
  \end{equation}
  if \(\module{M}\) is isomorphic to
  \(\sfmod{\ell}{\Typ{\lambda+\mu}}$, $\ell=0,-1\). Thus the proposition follows.
\end{proof}

\begin{rmk}
  To prove the above proposition directly without citing the
  literature, we could have used the two free field realisations in
  \cref{sec:ffr} to construct intertwining operators of the type appearing in
  equation \eqref{eq:dimintops}, thereby showing that the dimension
  of the corresponding space of intertwining operators is at least
  1. This was also done in \cite{AdaBG19}. An upper bound of 1 can
  then easily be determined by calculations involving either the HLZ
  double dual construction (similar to the calculations done in
  \cref{thm:stagfusion}) or the NGK algorithm. 
\end{rmk}

\begin{prop}
  For \(\lambda\in \RR/\ZZ\), \(\lambda \neq \ZZ\), we have
  \begin{equation}
    \Typ{\lambda} \fuse \Typ{-\lambda} \cong \sfmod{-1}{\Stag}.
  \end{equation}
\end{prop}
\begin{proof}
  By \cref{thm:rigidprojsimples}, \(\Typ{\lambda}\) is rigid and hence its
  fusion product with a projective module must again be projective.
  Further, by \cref{thm:closurecomps}, all
  composition factors must lie in categories \(\sfmod{\ell}{\catR}$, $\ell=-1, 0\). Finally, since ghost weights add under fusion, the
  ghost weights of the fusion product must lie in \(\ZZ\). Thus the fusion product must be isomorphic to a direct
  sum of some number of copies of \(\sfmod{-1}{\Stag}\). By
  \cref{thm:stagfusion}, we know there is exactly one such summand.
\end{proof}

\begin{prop}\label{thm:rigid}
  Category \(\catFLE\) is rigid.
\end{prop}
\begin{proof}
  Category \(\catFLE\) has sufficiently many injective and projective
  modules, that is,  all simple modules have projective covers and
  injective hulls, and all projectives are injective and
  vice-versa. Further, the simple projective 
  modules \(\sfmod{\ell}{\Typ{\lambda}}\) are rigid and generate the
  non-simple projective modules under fusion, so all projective modules are
  rigid. Catefory \(\catFLE\) is therefore a Frobenius category and hence any
  for short exact sequence with two rigid terms (whose duals are also rigid)
  the third term is also rigid. This implies that all modules are rigid and hence so is category
  \(\catFLE\).
\end{proof}

\begin{cor}\label{thm:stardual}
  Let \(\Mmod,\Nmod\in \catFLE\), then
  \begin{equation}
    \Mmod^* \fuse \Nmod^* \cong \brac{\Mmod \fuse \Nmod}^* .
  \end{equation}
\end{cor}
\begin{proof}
 Due to rigidity, the tensor duality functor ${}^\vee$ defines an equivalence of
 categories and is therefore exact.
 Further, the tensor duality functor satisfies
 \begin{equation}
   \Mmod^\vee \fuse \Nmod^\vee \cong \brac*{\Mmod \fuse \Nmod}^\vee.
 \end{equation}
 This also implies that $\Vmod{k}^\vee = \Vmod{-k}$.
 We see that the tensor dual $\Mmod^\vee$ agrees with $\sfmod{}{\brac*{\Mmod'}}$ on all simple modules in $\catFLE$.
 As both $\brac{-}^\vee$ and $\sfmod{}{\brac*{-}'}$ are exact contravariant
 invertible functors, we have \(\Mmod^\vee\cong \sfmod{}{\brac*{\Mmod'}}\) for
 any module in \(\catFLE\).
 Recalling \(\brac{-}^\ast=c\brac{-}'\), we further have \(M^\ast\cong \sfaut{}{\conjaut{M^\vee}}\). \cref{thm:specflow} then implies
 \begin{equation}
   \Mmod^* \fuse \Nmod^* \cong
   \brac*{\sfaut{}{\conjaut{\Mmod^\vee}}} \fuse \brac*{\sfaut{}{\conjaut{\Nmod^\vee}}}
   \cong 
  \sfaut{}{\conjaut{\brac*{\Mmod\fuse\Nmod}^\vee}}
  \cong \brac{\Mmod \fuse \Nmod}^*.
\end{equation}
\end{proof}

\subsection{Fusion products of reducible indecomposable modules} \label{sec:fusindecomp}

In this section we calculate the remaining fusion product formulae involving indecomposable
modules in $\catFLE$. The main tool for determining these fusion products is
that category \(\catFLE\) is rigid by \cref{thm:rigid}. Hence fusion is biexact
and projective modules form a tensor ideal.
We begin by calculating certain basic fusion products from which the remainder can be determined
inductively.
\begin{lem}
    \begin{align}
      \Tmod{2}{} \fuse \Bmod{2}{} &\cong \Stmod{1}, \nonumber\\
      \Bmod{2}{} \fuse \Bmod{2}{} &\cong \Bmod{2}{} \oplus \Bmod{2}{1}, \nonumber\\
      \Tmod{2}{} \fuse \Tmod{2}{} &\cong \Tmod{2}{} \oplus \Tmod{2}{1}.
    \end{align}
\end{lem}
\begin{proof}
  Taking the short exact sequence \eqref{eq:Weqseq1} for
  $\Typ{0}^+=\Tmod{2}{-1}$ and fusing it with \(\Typ{0}^-=\Bmod{2}{-1}\)
  yields the short exact sequence
  \begin{equation}\label{eq:fuseq1}
    \dses{\Typ{0}^{-}}{}{\Typ{0}^{+} \fuse \Typ{0}^{-}}{}{\sfmod{-1}{\Typ{0}^{-}}}.
  \end{equation}
  Similarly, fusing the short exact sequence \eqref{eq:Weqseq2} for
  $\Typ{0}^-$ with \(\Typ{0}^+\) yields
  \begin{equation}\label{eq:fuseq2}
    \dses{\sfmod{-1}{\Typ{0}^{+}}}{}{\Typ{0}^{-} \fuse \Typ{0}^{+}}{}{{\Typ{0}^{+}}}.
  \end{equation}
 If either of the above exact sequences splits there is a
 contradiction, because if \(\sfmod{-1}{\Typ{0}^+}\) and \(\Typ{0}^+\) are direct
   summands of \(\Typ{0}^{+} \fuse \Typ{0}^{-}\), \eqref{eq:fuseq1} is not
   exact, and if \(\Typ{0}^-\) and \(\sfmod{-1}{\Typ{0}^-}\) are direct
     summands, \eqref{eq:fuseq2} is not exact.
 Hence both sequences must be non-split. As can be read
    off from the tables in \cref{thm:homexttabs},
    \(\dExt{\sfmod{-1}{\Typ{0}^{-}}}{\Typ{0}^{-}}=\dExt{\Typ{0}^{+}}{\sfmod{-1}{\Typ{0}^{+}}}=1\).
 There is only one candidate for the middle coefficient of these exact
 sequences, namely $\sfmod{-1}{\Stag}$. Thus the first fusion rule follows.
 The other two fusion products by are determined by fusing \(\Typ{0}^\pm\) with
 the short exact sequences for \(\Typ{0}^\pm\). The extension groups
 corresponding to these fused exact sequences are zero-dimensional and hence
 the sequences split and the lemma follows.
\end{proof}

We further prepare the following Ext group dimensions for later use.
\begin{lem}\label{prop:loewy}
  The indecomposable modules \(\Tmod{2n+1}{}\), \(\Bmod{m}{2n+1}\),
  \(\Bmod{2n}{}\) and \(\Bmod{m}{2n}\) satisfy
  \begin{align}
    \dExt{\Tmod{2n+1}{}}{\Bmod{m}{2n+1}} = \dExt{\Bmod{2n}{}}{\Bmod{m}{2n}} = 1.
  \end{align}
  The corresponding extensions are given by $\Tmod{2n+m+1}{}$ and $\Bmod{2n+m}{}$ respectively.
\end{lem}
\begin{proof}
	We start with the following presentation of $\Tmod{2n+1}{}$
	\begin{equation}
	\dses{\Tmod{2n+2}{}}{}{\Proj{\Tmod{2n+1}{}}}{}{\Tmod{2n+1}{}}.
	\end{equation}
	Applying the functor $\Homgrp{-}{\Bmod{m}{2n+1}}$ yields
	\begin{equation}
	0 \lra {\Homgrp{\Tmod{2n+1}{}}{\Bmod{m}{2n+1}}} \lra {\Homgrp{\Proj{\Tmod{2n+1}{}}}{\Bmod{m}{2n+1}}} \lra {\Homgrp{\Tmod{2n+2}{}}{\Bmod{m}{2n+1}}} \lra {\Extgrp{}{\Tmod{2n+1}{}}{\Bmod{m}{2n+1}}} \lra 0.
	\end{equation} 
	The first coefficient vanishes due to $\Tmod{2n+1}{}$ and
        $\Bmod{m}{2n+1}$ having no common composition factors. The
        second coefficient can be shown to vanish
        using the projective cover formulae
        in \cref{thm:covhulpres}
        and reading off Hom group dimensions from the Loewy diagrams. For the
        third coefficient, the only composition factor common to both
        \(\Tmod{2n+2}{}\) and \(\Bmod{m}{2n+1}\) is $\Vmod{2n+1}$, which occurs
        as a quotient for $\Tmod{2n+2}{}$ and a submodule for
        $\Bmod{m}{2n+1}$, so this gives rise to a one dimensional Hom group. 
	The vanishing Euler characteristic then implies that $\dExt{\Tmod{2n+1}{}}{\Bmod{m}{2n+1}} = 1$ as expected.
	Furthermore, we can examine $\Tmod{2n+m+1}{}$ to see that it has a $\Bmod{m}{2n}$ submodule which yields $\Tmod{2n+1}{}$ when quotiented out, therefore this is the unique extension characterised by $\Extgrp{}{\Tmod{2n+1}{}}{\Bmod{m}{2n+1}}$.
	
	We can follow the same procedure starting with the projective presentation of $\Bmod{2n}{}$ to obtain the following exact sequence
	\begin{equation}
	0 \lra {\Homgrp{\Bmod{2n}{}}{\Bmod{m}{2n}}} \lra {\Homgrp{\Proj{\Bmod{2n}{}}}{\Bmod{m}{2n}}} \lra {\Homgrp{\Bmod{2n}{1}}{\Bmod{m}{2n}}} \lra {\Extgrp{}{\Bmod{2n+1}{}}{\Bmod{m}{2n}}} \lra 0.
	\end{equation} 
	By the same argument as above we can calculate the Hom groups, and vanishing Euler characteristic implies $\dExt{\Bmod{2n}{}}{\Bmod{m}{2n}} = 1$. Similarly we see that $\Bmod{2n+m}{}$ provides an extension of $\Bmod{2n}{}$ by $\Bmod{m}{2n}$ and must therefore be the unique one.
\end{proof}

We can now determine fusion products when
one factor has length 2 and the other has arbitrary length.
\begin{lem} \label{lem:b1t1}
  The fusion products of length 2 indecomposables with any
  indecomposable of types \(\Bmod{}{}\) or \(\Tmod{}{}\) satisfy the
  following decomposition formulae.
  \begin{align}
    \Bmod{2n+1}{} \fuse \Bmod{2}{} &\cong \Bmod{2}{} \oplus
                                     \Stsum{n}{1}&
    \Tmod{2n+1}{} \fuse \Tmod{2}{} &\cong \Tmod{2}{} \oplus \Stsum{n}{1}  \nonumber\\
    \Bmod{2n+2}{} \fuse \Bmod{2}{} &\cong \Bmod{2}{2n+1} \oplus \Bmod{2}{} \oplus \Stsum{n}{1}&
    \Tmod{2n+2}{} \fuse \Tmod{2}{} &\cong \Tmod{2}{2n+1} \oplus \Tmod{2}{} \oplus \Stsum{n}{1} \nonumber\\
    \Bmod{2n+1}{} \fuse \Tmod{2}{} &\cong \Tmod{2}{2n} \oplus \Stsum{n}{}& 
    \Tmod{2n+1}{} \fuse \Bmod{2}{} &\cong \Bmod{2}{2n} \oplus \Stsum{n}{}\nonumber\\
    \Bmod{2n}{} \fuse \Tmod{2}{} &\cong  \Stsum{n}{}&
    \Tmod{2n}{} \fuse \Bmod{2}{} &\cong  \Stsum{n}{}
  \end{align}
\end{lem}
\begin{proof}
  We prove the left column of identities. The right column then follows from
  \cref{thm:stardual} and applying the \(^\ast\) functor.
  We start with the short exact sequence \eqref{eq:defseq1} satisfied by $\Bmod{2n+1}{}$,
  \begin{equation}
    \dses{\Bmod{2n-1}{}}{}{\Bmod{2n+1}{}}{}{\Tmod{2}{2n-1}}.
  \end{equation}
  We then take the fusion product with $\Bmod{2}{}$,
  \begin{equation}
    \dses{\Bmod{2n-1}{} \fuse \Bmod{2}{}}{}{\Bmod{2n+1}{} \fuse \Bmod{2}{}}{}{\Stmod{2n}}.
  \end{equation}
  Because $\Stmod{2n}$ is projective, the sequence splits and we have the recurrence relation
  \begin{equation}
    \Bmod{2n+1}{} \fuse \Bmod{2}{} \cong \brac{\Bmod{2n-1}{} \fuse \Bmod{2}{}} \oplus \Stmod{2n}.
  \end{equation}
  Then, the first fusion product formula of the lemma 
  follows by induction with $\Bmod{1}{} =
  \VacMod$ as the base case.
  
  We next consider the short exact sequence \eqref{eq:dseq2} and fuse it with $\Bmod{2}{}$ to obtain
  \begin{align}
    \dses{\Bmod{2n+1}{} \fuse \Bmod{2}{}}{}{\Bmod{2n+2}{} \fuse \Bmod{2}{} }{}{\Bmod{2}{2n+1}}.
  \end{align}
  Since $\Extgrp{}{\Bmod{2}{2n+1}}{\Bmod{2}{}} = 0$, by the tables in \cref{thm:homexttabs}, this
  sequence splits and we obtain the second fusion product of the lemma.
  
  For the final two fusion products, we perform the same exercises with
  different exact sequences. For the third and fourth fusion products 
  we use \eqref{eq:xseq3}, with odd and even length respectively. Fusing
  with $\Tmod{2}{}$ gives the short exact sequences 
  \begin{align}
    \dses{{\Bmod{2n-1}{2} \fuse \Tmod{2}{}}}{}{\Bmod{2n+1}{} \fuse \Tmod{2}{}}{}{\Stmod{1}}, \nonumber\\
	\dses{{\Bmod{2n}{2} \fuse \Tmod{2}{}}}{}{\Bmod{2n+2}{} \fuse \Tmod{2}{}}{}{\Stmod{1}}. 
  \end{align}
  In both cases, the sequences split because $\Stmod{1}$ is projective.
\end{proof}

We now use \cref{lem:b1t1} to prove the fusion product formulae
  \eqref{eq:indfus1} of \cref{thm:fusionlist}. 
  \begin{prop}\label{thm:pureindfus}
    The fusion products of indecomposable modules of types
    \(\Bmod{}{}\) and \(\Tmod{}{}\) with themselves satisfy the
    decomposition formulae below, for $m \ge n$.
  \begin{align}
    \Bmod{2m+1}{} \fuse \Bmod{2n+1}{} &\cong \Bmod{2m + 2n +1}{} \oplus
                                        \Stsum{m,n}{1}&
                                                        \Tmod{2m+1}{} \fuse \Tmod{2n+1}{} &\cong \Tmod{2m + 2n +1}{} \oplus \Stsum{m,n}{1}\nonumber\\
    \Bmod{2m+1}{} \fuse \Bmod{2n}{} &\cong \Bmod{2n}{} \oplus \Stsum{m,n}{1}&
                                                                              \Tmod{2m+1}{} \fuse \Tmod{2n}{} &\cong \Tmod{2n}{} \oplus \Stsum{m,n}{1}\nonumber\\
    \Bmod{2m}{} \fuse \Bmod{2n}{} &\cong \Bmod{2n}{2m-1} \oplus \Bmod{2n}{}
                                    \oplus \Stsum{m-1,n}{1}&
                                                             \Tmod{2m}{} \fuse \Tmod{2n}{} &\cong \Tmod{2n}{2m-1} \oplus \Tmod{2n}{} \oplus \Stsum{m-1,n}{1}
   \end{align}
\end{prop}
\begin{proof}
  We prove the left column of identities. The right column then follows from
  \cref{thm:stardual} and applying the \(^\ast\) functor.
  First, for both superscripts odd, we take two short exact sequences
  \eqref{eq:xseq3} and \eqref{eq:defseq1} for $\Bmod{2n+1}{}$ and fuse with
  $\Bmod{2m+1}{}$ to find 
  \begin{align}
    \dses{{\Bmod{2n-1}2{} \fuse \Bmod{2m+1}{}}}{}{\Bmod{2n+1}{} \fuse
    \Bmod{2m+1}{}}{}{\Bmod{2}{} \oplus \Stsum{m}{1}}, \nonumber\\ 
    \dses{\Bmod{2n-1}{} \fuse \Bmod{2m+1}{}}{}{\Bmod{2n+1}{} \fuse \Bmod{2m+1}{}}{}{\Tmod{2}{2n + 2m -1} \oplus \Stsum{m}{2n-1}}. 
  \end{align}
  Now comparing these exact sequences, and using the fact that $\Stag$ is
  projective, we find that the sequences cannot both split, as they would give
  different direct sums. For the first short exact sequence, we use \cref{prop:loewy}, to find $\dExt{\Bmod{2}{}}{\Bmod{2m+2n-1}{2}} = 1 $, with
  the extension being given by $\Bmod{2m+2n+1}{}$
  so we can determine the fusion product formulae inductively to get
  \begin{align}
    \Bmod{2m+1}{} \fuse \Bmod{3}{} &\cong \Bmod{2m + 3}{} \oplus \Stsum{m}{1}, \nonumber\\
    \Bmod{2m+1}{} \fuse \Bmod{5}{} &\cong \Bmod{2m + 5}{} \oplus \brac*{1+\sfaut^2} \Stsum{m}{1} , \nonumber\\
    \Bmod{2m+1}{} \fuse \Bmod{2n+1}{} &\cong \Bmod{2m + 2n +1}{} \oplus \bigoplus_{k=1}^{m} \Stsum{n}{2k-1} = \Bmod{2m + 2n+1}{} \oplus \Stsum{m,n}{1}. 
  \end{align}
  We can deduce the remaining rules from short exact sequences that relate
  even and odd $\Bmod{}{}$s. Firstly, we take the two short exact sequences
  \eqref{eq:xseq3} and \eqref{eq:defseq2},  and fuse them with $\Bmod{2m+1}{}$ to get
  \begin{align}
    \dses{{\Bmod{2m+1}{2} \fuse \Bmod{2n}{}}}{}{\Bmod{2m+1}{} \fuse \Bmod{2n+2}{}}{}{\Bmod{2}{} \oplus \Stsum{m}{1}}, \nonumber\\
    \dses{{\Bmod{2}{2n}} \oplus \Stsum{m}{2n+1}}{}{\Bmod{2m+1}{} \fuse \Bmod{2n+2}{}}{}{\Bmod{2m+1}{} \fuse \Bmod{2n}{}}.
  \end{align}
  Either of these exact sequences splitting would lead to a contradiction,
  hence both must be non-split. Further, by \cref{prop:loewy} we find $\dExt{\Bmod{2}{}}{\Bmod{2n}{2}} = \dExt{\Bmod{2n}{}}{\Bmod{2}{2n}} = 1$,
  with the corresponding non-split extension given by $\Bmod{2n+2}{}$. Therefore
  \begin{equation}
    \Bmod{2m+1}{} \fuse \Bmod{2n}{} \cong \Bmod{2n}{} \oplus \Stsum{m,n}{1}.
  \end{equation}
  Finally we fuse \eqref{eq:dseq2} with $\Bmod{2n}{}$ to find
  \begin{align}
    \dses{\Bmod{2m+1}{} \fuse \Bmod{2n}{}}{}{\Bmod{2m+2}{} \fuse \Bmod{2n}{}}{}{\Bmod{2n}{2m+1}}. 
  \end{align}
  For $m \ge n$, $\dExt{\Bmod{2n}{2m+1}}{\Bmod{2n}{}} = 0$, which
  follows because the composition factors are separated by at least two units of
  spectral flow and $\Extgrp{}{\Vmod{n}}{\Vmod{m}} = 0$ for $\abs{n-m} > 1$, the above sequence splits.
  In the case when $m=n-1$, we have that $\Extgrp{}{\Bmod{2n}{2n-1}}{\Vmod{k}} =
  0$ for all the composition factors of $\Bmod{2n}{}$ , that is, $(0 \le k \le
  2n-1)$.
  Hence $\dExt{\Bmod{2n}{2n-1}}{\Bmod{2n}{}} = 0$ and the above sequence
  again splits. Thus,
  \begin{equation}
    \Bmod{2m+2}{} \fuse \Bmod{2n}{} \cong \Bmod{2n}{2m+1} \oplus \Bmod{2n}{}
    \oplus \Stsum{m,n}{1},\qquad m\ge n-1.    
  \end{equation}
\end{proof}

\begin{prop}\label{thm:mixindfus}
  The fusion products of indecomposable modules of types
  \(\Bmod{}{}\) and \(\Tmod{}{}\) with each other satisfy the
  decomposition formulae below, for $m \ge n$.
  \begin{align}
    \Tmod{2m+1}{} \fuse \Bmod{2n+1}{} &\cong \Tmod{2m-2n+1}{2n}  \oplus
                                        \Stsum{m+1,n}{}&
                                                         \Bmod{2m+1}{} \fuse \Tmod{2n+1}{} &\cong \Bmod{2m-2n+1}{2n}  \oplus \Stsum{m+1,n}{}\nonumber\\
    \Tmod{2m}{} \fuse \Bmod{2n+1}{} &\cong \Tmod{2m}{2n}  \oplus \Stsum{m,n}{}&
                                                                                \Bmod{2m}{} \fuse \Tmod{2n+1}{} &\cong \Bmod{2m}{2n}  \oplus \Stsum{m,n}{} \nonumber\\
    \Tmod{2m}{} \fuse \Bmod{2n}{} &\cong \Stsum{m,n}{}&\Bmod{2m}{} \fuse \Tmod{2n}{} &\cong \Stsum{m,n}{}
  \end{align}
\end{prop}
\begin{proof}
  We prove the left column of identities. The right column then follows from \cref{thm:stardual} and applying the \(^\ast\) functor to each module.
  We start with sequences \eqref{eq:defseq1} and \eqref{eq:xseq3} for odd length $\Bmod{}{}$, and fuse them with $\Tmod{2m+1}{}$ to find
  \begin{align}
    \dses{\Tmod{2m+1}{} \fuse \Bmod{2n-1}{}}{}{\Tmod{2m+1}{} \fuse \Bmod{2n+1}{}&}{}{\Tmod{2}{2n-1} \oplus \Stsum{m}{2n}}, \\
    \dses{\Tmod{2m+1}{} \fuse \Bmod{2n-1}{2}}{}{\Tmod{2m+1}{} \fuse \Bmod{2n+1}{}&}{}{\Bmod{2}{2m} \oplus \Stsum{m}{}}.
  \end{align}
  Specialising to n=1 we have
  \begin{align}
    \dses{\Tmod{2m+1}{}}{}{\Tmod{2m+1}{} \fuse \Bmod{3}{}&}{}{\Tmod{2}{1} \oplus \bigoplus_{k=1}^{m} \sfmod{2k+1}{\Stag}}, \\
    \dses{\Tmod{2m+1}{2}}{}{\Tmod{2m+1}{} \fuse \Bmod{3}{}&}{}{\Bmod{2}{2m} \oplus \bigoplus_{k=1}^{m} \sfmod{2k-1}{\Stag}}.
  \end{align}
  Since $\Stmod{}$ is projective, its spectral flows must appear as direct
  summands in the middle coefficient of the above exact sequences. Thus, 
  \begin{equation}
    \Tmod{2m+1}{} \fuse \Bmod{3}{} \cong \Amod \oplus \bigoplus_{k=1}^{m+1}  \sfmod{2k-1}{\Stag} = \Amod \oplus \Stsum{m+1}{}.
  \end{equation}
  Therefore the module $\Amod$ satisfies the exact sequences
  \begin{align}
    &\dses{\Tmod{2m+1}{}}{}{\Amod \oplus \Stmod{1}}{}{{\Tmod{2}{1}}}, \nonumber\\
    &\dses{\Tmod{2m+1}{2}}{}{\Amod \oplus \Stmod{2m+1}}{}{\Bmod{2}{2m}}. 
  \end{align}
  Because either of these sequences splitting would lead to a contradiction
  and the corresponding extension groups are one-dimensional, the sequences
  uniquely characterise the fusion product. Proceeding by induction, we obtain
  \begin{align}
    \Tmod{2m+1}{} \fuse \Bmod{3}{} &\cong \Tmod{2m-1}{2} \oplus \Stsum{m+1}{},\nonumber\\
    \Tmod{2m+1}{} \fuse \Bmod{5}{} &\cong \Tmod{2m-3}{4} \oplus \brac*{1+\sfaut^2} \Stsum{m+1}{},\nonumber\\
    \Tmod{2m+1}{} \fuse \Bmod{2n+1}{} &\cong \Tmod{2m-2n+1}{2n}  \oplus \Stsum{m+1,n}{}.
  \end{align}
  Next we take two short exact sequences \eqref{eq:dseq4} and \eqref{eq:xseq2}, for $\Tmod{2m}{}$ and fuse them with $\Bmod{2n+1}{}$ to get
  \begin{align}
    &\dses{\Bmod{2n+1}{2m-1}}{}{\Tmod{2m}{}\fuse\Bmod{2n+1}{}}{}{\Tmod{2m-2n-1}{2n}  \oplus \Stsum{m,n}{}}, \nonumber\\
    &\dses{\Bmod{2m+2n-1}{1}\oplus \Stsum{m-1,n}{2}}{}{\Tmod{2m}{}\fuse\Bmod{2n+1}{}}{}{\Bmod{2n+1}{}}.
  \end{align}
  Again either of these sequences splitting would lead to a contradiction, and
  by \cref{prop:loewy}, 
  $\dExt{\Tmod{2m-2n-1}{2n}}{\Bmod{2n+1}{2m-1}} = 1$, with the extension being given by $\Tmod{2m}{2n}$, so the second fusion rule follows.
  Finally, fusing \eqref{eq:defseq4} with $\Bmod{2n}{}$, we have 
  \begin{align}
    \dses{\Tmod{2m-2}{} \fuse \Bmod{2n}{}}{}{\Tmod{2m}{} \fuse \Bmod{2n}{}}{}{\Stsum{n}{2m-2}}, \\
    \Tmod{2m}{} \fuse \Bmod{2n}{} \cong \bigoplus_{k=1}^{m} \Stsum{n}{2k-2} = \Stsum{m,n}{}.
  \end{align}
\end{proof}

\appendix

\section{Sufficient conditions for convergence and extension -- Proof of \cref{thm:gradconext}}
\label{sec:suffconvext}

In this section we give a proof of \cref{thm:gradconext} by reviewing
reasoning presented by Yang in \cite{Yang18} and 
showing that certain assumptions on the category of strongly graded modules (see
\cite[Assumption 7.1, Part 3]{Yang18}) are not required, if one only wishes
  to conclude that convergence and extension properties hold. Instead all that
is required is that the modules considered satisfy suitable finiteness
conditions. 
This appendix closely follows the logic of \cite[Sections
5 \& 6]{Yang18} and also \cite[Section 2]{HuaDif04}.

Throughout this section
let \(A \le B\) be abelian groups.
Further, let \(\VOA{V}\) be an \(A\)-graded \va{} with a vertex
subalgebra 
\(\overline{\VOA{V}}\subset\VOA{V}^{(0)}\). In this section only, all mode
expansions of fields from a \voa{} \(\VOA{V}\) will be of the form \(Y(v,z)=\sum_{n\in\ZZ}
v_{n}z^{-n-1}\) regardless of the conformal weight of \(v\in\VOA{V}\), that is,
\(v_{n}\) refers to the coefficient of \(z^{-n-1}\) rather than the one which
shifts conformal weight by \(-n\).

\begin{defn}
  Let \(\Wmod_0, \Wmod_1, \Wmod_2, \Wmod_3,\Wmod_4\) be \(B\)-graded \(\VOA{V}\)-modules.
  \begin{enumerate}
  \item We say that two \(B\)-graded logarithmic intertwining operators
    \(\mathcal{Y}_1\), \(\mathcal{Y}_2\) of respective types
    \(\binom{\Wmod_0}{\Wmod_1,\ \Wmod_4}\), \(\binom{\Wmod_4}{\Wmod_2,\ \Wmod_3}\) satisfy the
  \emph{convergence and 
      extension property for products} if for any \(a_1,a_2,\in B\) and any
    doubly homogeneous elements \(w_0^\prime\in
    \Wmod_0^\prime\), \(w_3\in \Wmod_3\),\ \(w_i\in \Wmod_i^{(a_i)},\ i=1,2\), there exist
    \(M\in \ZZ_{\ge0}\), \(r_1,\dots, r_M,s_1,\dots s_M\in\RR\), \(u_1,\dots u_M,
    v_1,\dots v_M\in\ZZ_{\ge0}\) and analytic functions \(f_1(z),\dots,
    f_M(z)\) on the disc \(|z|<1\) satisfying
    \begin{equation}
      \wt w_1+\wt w_2 + s_k >N,\quad \text{for each}\ k=1,\dots, M,
    \end{equation}
    where \(N\in\ZZ\) depends only on the intertwining operators
    \(\mathcal{Y}_1,\ \mathcal{Y}_2\) and \(a_1+a_2\), such that as a formal power series the matrix element
    \begin{equation}
      \dpair{w_0^\prime}{\mathcal{Y}_1(w_1,z_1)\mathcal{Y}_2(w_2,z_2)w_3}
      \label{eq:prodmatel}
    \end{equation}
    converges absolutely in the region \(|z_1|>|z_2|>0\) and may be
    analytically continued to the multivalued analytic function
    \begin{equation}
      \sum_{k=1}^M z_2^{r_k}(z_1-z_2)^{s_k}(\log
      z_2)^{u_k}(\log(z_1-z_2))^{v_k}f_k\brac*{\frac{z_1-z_2}{z_2}}
      \label{eq:prodconvext}
    \end{equation}
    in the region \(|z_2|>|z_1-z_2|>0\).
    \label{def:convextprod}
  \item We say that two \(B\)-graded logarithmic intertwining operators
    \(\mathcal{Y}_1\), \(\mathcal{Y}_2\) of respective types
    \(\binom{\Wmod_0}{\Wmod_4,\ \Wmod_3}\), \(\binom{\Wmod_4}{\Wmod_1,\ \Wmod_2}\) satisfy the
  \emph{convergence and 
      extension property for iterates} if for any \(a_2,a_3,\in B\) and any
    doubly homogeneous elements \(w_0^\prime\in
    \Wmod_0^\prime\), \(w_1\in \Wmod_3\),\ \(w_i\in \Wmod_i^{(a_i)},\ i=2,3\), there exist
    \(M\in \ZZ_{\ge0}\), \(r_1,\dots, r_M,s_1,\dots s_M\in\RR\), \(u_1,\dots u_M,
    v_1,\dots v_M\in\ZZ_{\ge0}\) and analytic functions \(f_1(z),\dots
    f_M(z)\) on the disc \(|z|<1\) satisfying
    \begin{equation}
      \wt w_2+\wt w_3 + s_k >N,\quad \text{for each}\ k=1,\dots, M ,
    \end{equation}
    where \(N\in\ZZ\) depends only on the intertwining operators
    \(\mathcal{Y}_1,\ \mathcal{Y}_2\) and \(a_2+a_3\), such that as a formal
    power series the matrix element 
    \begin{equation}
      \dpair{w_0^\prime}{\mathcal{Y}_1(\mathcal{Y}_2(w_1,z_1-z_2)w_2,z_2)w_3}
      \label{eq:itmatel}
    \end{equation}
    converges absolutely in the region \(|z_2|>|z_1-z_2|>0\) and may be
    analytically continued to the multivalued analytic function
    \begin{equation}
      \sum_{k=1}^M z_1^{r_k}z_2^{s_k}(\log z_1)^{u_k}(\log
      z_2)^{v_k}f_k\brac*{\frac{z_2}{z_1}}
      \label{eq:itconvext}
    \end{equation}
    in the region \(|z_1|>|z_2|>0\).
  \end{enumerate}
\end{defn}

Consider the Noetherian ring
\(R=\CC[z_1^{\pm1},z_2^{\pm1},(z_1-z_2)^{-1}]\). Then for any quadruple of \(B\)-graded
\(\VOA{V}\)-modules \(\Wmod_0, \Wmod_1, \Wmod_2,\Wmod_3,\) and any triple \((a_1,a_2,a_3)\in B^3\), we define the
\(R\)-module
\begin{equation}
  T^{(a_1,a_2,a_3)}=R\otimes \brac*{\Wmod_0^{\prime}}^{(a_1+a_2+a_3)}\otimes
  \Wmod_1^{(a_1)}\otimes \Wmod_2^{(a_2)}\otimes \Wmod_3^{(a_3)},
\end{equation}
where all the tensor product symbols denote complex tensor products. We will
generally omit the tensor product symbol separating \(R\) from the \(\VOA{V}\)-modules.
The motivation for considering this module is that for any \(B\)-graded module
\(\Wmod_4\) and any pair of grading compatible
logarithmic intertwining operators \(\mathcal{Y}_1, \mathcal{Y}_2\) of respective types \(\binom{\Wmod_0}{\Wmod_1,\ \Wmod_4}\) and
\(\binom{\Wmod_4}{\Wmod_2,\ \Wmod_3}\) it produces matrix elements via the map
\(\phi_{\mathcal{Y}_1,\mathcal{Y}_2}:T^{(a_1,a_2,a_3)}\to
z_1^h\CC\brac*{\set{z_2/z_1}}\sqbrac{z_1^{\pm1},z_2^{\pm1}}\),
where \(h\) is the combined conformal weight of \(w_0^\prime,\ w_1,\ w_2,\
w_3\) and \(\CC\brac*{\set{x}}\) is the space of all power series in \(x\) with
bounded below real exponents (the modules \(\Wmod_i\), \(i=0,1,2,3\) will always
have real conformal weights below),
defined by
\begin{align}
  \phi_{\mathcal{Y}_1,\mathcal{Y}_2}(f(z_1,z_2)w_0^\prime\otimes w_1\otimes
  w_2\otimes w_3)=\iota(f(z_1,z_2))\langle w_0^\prime,
  \mathcal{Y}_1(w_1,z_1)\mathcal{Y}_2(w_2,z_2)w_3\rangle,
  \label{eq:matelmap}
\end{align}
where \(\iota: R\to \CC\powser{z_2/z_1}\sqbrac{z_1^{\pm1},z_2^{\pm1}}\) is the
map expanding elements of \(R\) such that the powers of \(z_2\) are bounded below.
This in turn justifies considering the submodule
\begin{multline}
  J^{(a_1,a_2,a_3)}=\cspn_{R}\left\{\mathcal{A}\brac*{v,w_0^\prime,w_1,w_2,w_3},\mathcal{B}\brac*{v,w_0^\prime,w_1,w_2,w_3},\mathcal{C}\brac*{v,w_0^\prime,w_1,w_2,w_3},\mathcal{D}\brac*{v,w_0^\prime,w_1,w_2,w_3}\in
    T^{(a_1,a_2,a_3)}\right.
    \st\\ \left.v\in\overline{\VOA{V}},\ w_0^\prime\in
    \brac{\Wmod_0^\prime}^{(a_1+a_2+a_3)},\ w_i\in W^{(a_i)},\ i=1,2,3\right\},
\end{multline}
where the generators
\begin{align}
  \mathcal{A}\brac*{v,w_0^\prime,w_1,w_2,w_3}&=-w_0^\prime\otimes v_{-1} w_1\otimes w_2\otimes w_3
    + \sum_{k\ge 0}\binom{-1}{k}(-z_1)^kv_{-1-k}^{\ast}w_0^\prime \otimes
    w_1\otimes w_2\otimes w_3\nonumber\\
    &\quad-\sum_{k\ge0}\binom{-1}{k} (-(z_1-z_2))^{-1-k}
    w_0^\prime\otimes w_1\otimes v_k w_2\otimes w_3 \nonumber\\
                                             &\quad
                                               -\sum_{k\ge0}\binom{-1}{k} (-z_1)^{-1-k}
                                               w_0^\prime\otimes w_1\otimes w_2\otimes v_k w_3 , \nonumber\\
  \mathcal{B}\brac*{v,w_0^\prime,w_1,w_2,w_3}&=-w_0^\prime\otimes w_1\otimes
                                               v_{-1}w_2\otimes w_3+
                                                \sum_{k\ge 0}\binom{-1}{k}(-z_2)^kv_{-1-k}^{\ast}w_0^\prime \otimes
    w_1\otimes w_2\otimes w_3
                                               \nonumber\\
  &\quad -\sum_{k\ge0}\binom{-1}{k} (-(z_1-z_2))^{-1-k}
    w_0^\prime\otimes v_k w_1\otimes w_2\otimes w_3 \nonumber\\
                                             &\quad
                                               -\sum_{k\ge0}\binom{-1}{k} (-z_2)^{-1-k}
                                               w_0^\prime\otimes w_1\otimes
                                               w_2\otimes v_k w_3 , \nonumber\\ 
  \mathcal{C}\brac*{v,w_0^\prime,w_1,w_2,w_3}&= v^{\ast}_{-1}w_0^\prime\otimes w_1\otimes
                                               v_{-1}w_2\otimes w_3-
                                               \sum_{k\ge0}\binom{-1}{k} z_1^{-1-k}
                                               w_0^\prime\otimes v_k w_1\otimes
                                               w_2\otimes w_3
                                               \nonumber\\
  &\quad -\sum_{k\ge0}\binom{-1}{k} z_2^{-1-k}
                                               w_0^\prime\otimes w_1\otimes
    v_k w_2\otimes w_3
    - w_0^\prime\otimes w_1\otimes
    w_2\otimes v_{-1} w_3 , \nonumber\\
  \mathcal{D}\brac*{v,w_0^\prime,w_1,w_2,w_3}&=v_{-1}w_0^\prime\otimes w_1\otimes
                                               v_{-1}w_2\otimes w_3-
                                               \sum_{k\ge0}\binom{-1}{k} z_1^{k+1}
                                               w_0^\prime\otimes \ee^{z_1^{-1}L_{1}}\brac*{-z_1^2}^{L_0}v_k\brac*{-z_1^{-2}}^{L_0}\ee^{-z_1^{-1}L_{1}} w_1\otimes
                                               w_2\otimes w_3
                                               \nonumber\\
  &\quad -\sum_{k\ge0}\binom{-1}{k} z_2^{-1-k}
                                               w_0^\prime\otimes w_1\otimes
    \ee^{z_2^{-1}L_{1}}\brac*{-z_2^2}^{L_0}v_k\brac*{-z_2^{-2}}^{L_0}\ee^{-z_2^{-1}L_{1}} w_2\otimes w_3\nonumber\\
    &\quad - w_0^\prime\otimes w_1\otimes
    w_2\otimes v^{\ast}_{-1} w_3 ,
  \label{eq:jgens}
\end{align}
are preimages of the relations coming from residues of the Jacobi identity for
intertwining operators and where \(v_k^\ast:\Wmod_i^\prime\to \Wmod_i^\prime\) denotes the adjoint of
\(v_k:\Wmod_i\to \Wmod_i\). Hence \(J^{(a_1,a_2,a_3)}\) lies in the kernel of
\(\phi_{\mathcal{Y}_1,\mathcal{Y}_2}\) for any choice of intertwining
operators \(\mathcal{Y}_1,\mathcal{Y}_2\) of the correct types.

Next consider the doubly homogeneous space
\begin{equation}
  T^{(a_1,a_2,a_3)}_{[r]}=\prod_{\substack{r_0,r_1,r_2,r_3\in\RR\\r_0+r_1+r_2+r_3=r}}
  R\otimes \brac*{\Wmod_0^\prime}_{[r_0]}^{(a_1+a_2+a_3)}\otimes \brac*{\Wmod_1}_{[r_1]}^{(a_1)}
  \otimes \brac*{\Wmod_2}_{[r_2]}^{(a_2)}\otimes \brac*{\Wmod_3}_{[r_3]}^{(a_3)}
\end{equation}
to construct the subspaces
\begin{align}
  F_r(T^{(a_1,a_2,a_3)})&=\prod_{s\le r}T_{[s]}^{(a_1,a_2,a_3)} ,\nonumber\\
  F_r(J^{(a_1,a_2,a_3)})&=J^{(a_1,a_2,a_3)}\cap F_r(T^{(a_1,a_2,a_3)}).
\end{align}
These define filtrations on \(T^{(a_1,a_2,a_3)}\) and \(J^{(a_1,a_2,a_3)}\),
respectively, since \(F_s(T^{(a_1,a_2,a_3)})\subset F_r(T^{(a_1,a_2,a_3)})\)
and \(F_s(J^{(a_1,a_2,a_3)})\subset F_r(J^{(a_1,a_2,a_3)})\),
if \(s\le r\), and \(\bigcup_{r\in \RR}
F_r(T^{(a_1,a_2,a_3)})=T^{(a_1,a_2,a_3)}\) and \(\bigcup_{r\in \RR}
F_r(J^{(a_1,a_2,a_3)})=J^{(a_1,a_2,a_3)}\).  Note that if
the \(\Wmod_i, i=0,1,2,3\) are discretely strongly graded, then \(T^{(a_1,a_2,a_3)}_{[r]}\)
is a finite sum of finite dimensional doubly homogeneous spaces tensored with \(R\). Hence
\(T^{(a_1,a_2,a_3)}_{[r]}\) is a finitely generated free
\(R\)-module. Further, \(F_r(T^{(a_1,a_2,a_3)})\) is also a finite sum and
hence also a finitely generated free \(R\)-module. Finally, the ring \(R\) is Noetherian and
so the submodule \(F_r(J^{(a_1,a_2,a_3)})\) is also finitely generated.
\begin{prop}\label{thm:filtbounds}
  Let the \(\VOA{V}\)-modules \(\Wmod_i,\ i=0,1,2,3\) be discretely strongly \(B\)-graded and
  \(B\)-graded \(C_1\)-cofinite as \(\overline{\VOA{V}}\)-modules, then for
  any \(a_1,a_2,a_3\in B\) there exists \(M\in \ZZ\) such that for any \(r\in \RR\)
  \begin{equation}
    F_r(T^{(a_1,a_2,a_3)})\subset F_r(J^{(a_1,a_2,a_3)}) +
    F_M(T^{(a_1,a_2,a_3)})\quad \text{and}\quad
    T^{(a_1,a_2,a_3)}\subset J^{(a_1,a_2,a_3)} +
    F_M(T^{(a_1,a_2,a_3)}).
  \end{equation}
\end{prop}
\begin{proof}
  By assumption the modules \(\Wmod_i,\ i=0,1,2,3\) are \(B\)-graded
  \(C_1\)-cofinite as \(\overline{V}\)-modules, that is, the spaces
  \begin{equation}
    C_1(\Mmod)^{(a)}=\spn{v_{-h}w\in M^{(a)}\st v\in \overline{\VOA{V}}_{[h]}
      h>0, w\in \Mmod}
  \end{equation}
  have finite codimension in \(\Mmod^{(a)}\) for \(\Mmod = \Wmod_i,\
  i=0,1,2,3\). Thus \(\Mmod^{(a)}_{[h]}\subset C_1(\Mmod)^{(a)}\) for
  sufficiently large conformal weight \(h\in \RR\) and hence there exists
  \(M\in \ZZ\) such that
  \begin{multline}
    \bigoplus_{n>M}T^{(a_1,a_2,a_3)}_{[n]}\subset
    C_1\brac*{\Wmod_0^\prime}^{(a_1+a_2+a_3)}\otimes \Wmod_1^{(a_1)}\otimes
    \Wmod_2^{(a_2)}\otimes \Wmod_3^{(a_3)}
    +\brac*{\Wmod_0^\prime}^{(a_1+a_2+a_3)}\otimes C_1\brac*{\Wmod_1}^{(a_1)}\otimes
    \Wmod_2^{(a_2)}\otimes \Wmod_3^{(a_3)}\\
    +\brac*{\Wmod_0^\prime}^{(a_1+a_2+a_3)}\otimes \Wmod_1^{(a_1)}\otimes
    C_1\brac*{\Wmod_2}^{(a_2)}\otimes \Wmod_3^{(a_3)}
    +\brac*{\Wmod_0^\prime}^{(a_1+a_2+a_3)}\otimes \Wmod_1^{(a_1)}\otimes
    \Wmod_2^{(a_2)}\otimes C_1\brac*{\Wmod_3}^{(a_3)}.
    \label{eq:c1prods}
  \end{multline}
  We prove the first inclusion of the proposition by induction on \(r\in
  \RR\). If \(r\le M\), then the inclusion is true by \(F_r(T^{(a_1,a_2,a_3)})\)
  defining a filtration. Next assume that \(F_r(T^{(a_1,a_2,a_3)})\subset F_r(J^{(a_1,a_2,a_3)}) +
    F_M(T^{(a_1,a_2,a_3)})\) is true for all \(r<s\in \RR\) for some
    \(s>M\). We will show that any element of the homogeneous space
    \(T^{(a_1,a_2,a_3)}_{[s]}\) can be written as a sum of elements in
    \(F_s(J^{(a_1,a_2,a_3)})\) and \(F_M(T^{(a_1,a_2,a_3)})\). Since
    \(s>M\), this homogeneous element is an element of the \rhs{} of
    \eqref{eq:c1prods}. We shall only consider the case of this element lying
    in the second summand of the \rhs{}, as the other cases follow
    analogously. Without loss of generality we can assume the element has the
    form
    \(w_0^\prime\otimes v_{-1}w_1\otimes w_2\otimes w_3\in T^{(a_1,a_2,a_3)}_{[s]}\), where
    \(w_0^\prime\in \brac*{\Wmod_0^\prime}^{(a_1+a_2+a_3)}\), \(w_i\in \Wmod_i^{(a_i)}, i=1,2,3\), \(v\in
    \overline{V}_{[h]}, h>0\). By computing the degrees of the summands making up
    \(\mathcal{A}(v,w_0^\prime,w_1,w_2,w_3)\) in \eqref{eq:jgens} we see that
    the three sums over \(k\) all lie in \(F_{s-1}(T^{(a_1,a_2,a_3)})\subset
    F_{s-1}(J^{(a_1,a_2,a_3)}) +
    F_M(T^{(a_1,a_2,a_3)})\) and that 
    \(\mathcal{A}(v,w_0^\prime,w_1,w_2,w_3)\in
    F_s(J^{(a_1,a_2,a_3)})\).
     Further,
    \begin{multline}
      w_0^\prime\otimes v_{-1}w_1\otimes w_2\otimes w_3=
      -\mathcal{A}(v,w_0^\prime,w_1,w_2,w_3)+\sum_{k\ge 0}\binom{-1}{k}(-z_1)^kv_{k}^{\ast}w_0^\prime \otimes
    w_1\otimes w_2\otimes w_3\\
    -\sum_{k\ge0}\binom{-1}{k} (-(z_1-z_2))^{-1-k}
    w_0^\prime\otimes w_1\otimes v_k w_2\otimes w_3 
                                               -\sum_{k\ge0}\binom{-1}{k} (-z_1)^{-1-k}
                                               w_0^\prime\otimes w_1\otimes w_2\otimes v_k w_3.
    \end{multline}
    Thus \(w_0^\prime\otimes v_{-1}w_1\otimes
    w_2\otimes w_3\) lies in the sum \(F_{s}(J^{(a_1,a_2,a_3)}) +
    F_M(T^{(a_1,a_2,a_3)})\) and the first inclusion of the proposition
    follows. The second inclusion follows from \(F_r(T^{(a_1,a_2,a_3)})\) and
    \(F_r(J^{(a_1,a_2,a_3)})\) defining filtrations.
    \begin{align}
      T^{(a_1,a_2,a_3)}&=\bigcup_{r\in \RR}F_r(T^{(a_1,a_2,a_3)})
      \subset\bigcup_{r\in \RR}\brac*{F_{r}(J^{(a_1,a_2,a_3)}) +
        F_M(T^{(a_1,a_2,a_3)})} \nonumber \\
      &=\brac*{\bigcup_{r\in
          \RR}F_{r}(J^{(a_1,a_2,a_3)})}+F_M(T^{(a_1,a_2,a_3)}) 
      =  J^{(a_1,a_2,a_3)}+F_M(T^{(a_1,a_2,a_3)}).
    \end{align}
  \end{proof}

  \begin{cor}
    Let the \(\VOA{V}\)-modules \(\Wmod_i,\ i=0,1,2,3\) be discretely strongly \(B\)-graded and
    \(B\)-graded \(C_1\)-cofinite as \(\overline{\VOA{V}}\)-modules.
    \begin{enumerate}
    \item The quotient \(R\)-module
      \(T^{(a_1,a_2,a_3)}/J^{(a_1,a_2,a_3)}\) is
      finitely generated.
      \label{itm:fingen}
    \item   For any representative \(w\in T^{(a_1,a_2,a_3)}\), we denote its coset in
      \(T^{(a_1,a_2,a_3)}/J^{(a_1,a_2,a_3)}\) by \(\sqbrac*{w}\).
      Let
      \(w_0^\prime \in \brac*{\Wmod_0^\prime}^{(a_1+a_2+a_3)}\) and \(w_i\in
      \Wmod_i^{(a_i)}\) \(i=1,2,3\), and consider the submodules of
      \(T^{(a_1,a_2,a_3)}/ J^{(a_1,a_2,a_3)}\) given by
      \begin{align}
        M_1&=\rgen{\sqbrac*{w_0\otimes L_{-1}^j w_1\otimes w_2\otimes w_3}\st j\in
             \ZZ_{\ge0}},& 
        M_2&=\rgen{\sqbrac*{w_0\otimes  w_1\otimes L_{-1}^jw_2\otimes w_3}\st j\in
             \ZZ_{\ge0}}.
      \end{align}
      Then \(M_1\) and \(M_2\) are finitely generated, in particular, there exist
      \(m,n\in \ZZ_{\ge0}\) and \(a_k(z_1,z_2), b_\ell(z_1,z_2)\in R\), \(1\le k
      \le m\), \(1\le \ell\le n\) such that
      \begin{align}
        \sqbrac*{w_0\otimes L_{-1}^m w_1\otimes w_2\otimes
        w_3}+a_1(z_1,z_2)\sqbrac*{w_0\otimes L_{-1}^{m-1} w_1\otimes w_2\otimes
        w_3}
        +\cdots+a_m(z_1,z_2)\sqbrac*{w_0\otimes w_1\otimes w_2\otimes
        w_3}&=0,\nonumber\\
        \sqbrac*{w_0\otimes  w_1\otimes L_{-1}^n w_2\otimes
        w_3}+b_1(z_1,z_2)\sqbrac*{w_0\otimes w_1\otimes L_{-1}^{n-1} w_2\otimes
        w_3}
        +\cdots+b_n(z_1,z_2)\sqbrac*{w_0\otimes w_1\otimes w_2\otimes
        w_3}&=0.
        \label{eq:genrels}
      \end{align}
      \label{itm:genrels}
    \end{enumerate}
    \label{thm:gens}
  \end{cor}
  \begin{proof}
    Since \(R\) is a Noetherian ring, Part \ref{itm:fingen} holds if
    \(T^{(a_1,a_2,a_3)}/J^{(a_1,a_2,a_3)}\) is
    isomorphic to a subquotient of a finitely generated module over \(R\).
    By \cref{thm:filtbounds} we have the inclusion and identification
    \begin{align}
 	T^{(a_1,a_2,a_3)}/J^{(a_1,a_2,a_3)} & \subset
      \brac*{J^{(a_1,a_2,a_3)}+F_M(T^{(a_1,a_2,a_3)})}/J^{(a_1,a_2,a_3)}
      \cong
      F_M(T^{(a_1,a_2,a_3)})/\brac*{F_M(T^{(a_1,a_2,a_3)})\cap J^{(a_1,a_2,a_3)}}.
    \end{align}
    Thus \(T^{(a_1,a_2,a_3)}/J^{(a_1,a_2,a_3)}\)
    is isomorphic to a subquotient of the finitely generated module
    \(F_M(T^{(a_1,a_2,a_3)})\) and Part \ref{itm:fingen} follows.
    Part \ref{itm:genrels} is an immediate consequence of Part
    \ref{itm:fingen} and the fact that a submodule of a finitely generated module over
    a Noetherian ring is again finitely generated.
  \end{proof}

  \begin{thm}\label{thm:mateldiffs}
    Let the \(\VOA{V}\)-modules \(\Wmod_i,\ i=0,1,2,3\) be discretely strongly \(B\)-graded and
    \(B\)-graded \(C_1\)-cofinite as \(\overline{\VOA{V}}\)-modules, let
    \(\Wmod_4\) be a \(B\)-graded \(\VOA{V}\)-module and let \(\mathcal{Y}_1\),
    \(\mathcal{Y}_2\) be logarithmic grading compatible intertwining operators
    of types \(\binom{\Wmod_0}{\Wmod_1,\ \Wmod_4}\), \(\binom{\Wmod_4}{\Wmod_2,\ \Wmod_3}\),
    respectively. Then for any homogeneous elements \(w_0^\prime \in
    \Wmod_0^\prime\), \(w_i\in \Wmod_i\), \(i=1,2,3\), there exist
    \(m,n\in\ZZ_{\ge0}\) and \(a_k(z_1,z_2),b_\ell(z_1,z_2)\in R\), \(1\le k\le m,\
    1\le \ell\le n\) such that the power series expansion of the matrix element
    \begin{equation}
      \dpair{w_0^\prime}{\mathcal{Y}_1(w_1,z_1)\mathcal{Y}_2(w_2,z_2)w_3}
      \label{eq:matels}
    \end{equation}
    is a solution to the power series expansion of the system of differential equations
    \begin{align}
      \frac{\partial^m \phi}{\partial
      z_1^m}+a_1(z_1,z_2)\frac{\partial^{m-1}\phi}{\partial z_1^{m-1}}+\cdots +
      a_m(z_1,z_2)\phi&=0,& 
      \frac{\partial^n \phi}{\partial
      z_2^n}+b_1(z_1,z_2)\frac{\partial^{n-1}\phi}{\partial z_1^{n-1}}+\cdots
      + b_n(z_1,z_2)\phi=0,
     \label{eq:diffeq}
    \end{align}
    in the region \(|z_1|>|z_2|>0\).
  \end{thm}
  \begin{proof}
    Let \(a_1,a_2,a_3\) be the respective \(B\)-grades of \(w_1,w_2,w_3\),
    then we can assume that the \(B\)-grade of \(w_0^\prime\) is
    \(a_1+a_2+a_3\), because otherwise the matrix element vanishes and the
    theorem follows trivially.
    Recall the map \(\phi_{\mathcal{Y}_1,\mathcal{Y}_2}:T^{(a_1,a_2,a_3)}\to z_1^h
    \CC\brac*{\set{z_2/z_1}}\sqbrac*{z_1^{\pm1},z_2^{\pm2}}\), 
    defined by the 
    formula \eqref{eq:matelmap}. Since \(J^{(a_1,a_2,a_3)}\) lies in the
    kernel of \(\phi_{\mathcal{Y}_1,\mathcal{Y}_2}\), we have an induced map
    \begin{equation}
      \overline{\phi}_{{\mathcal{Y}_1,\mathcal{Y}_2}}:T^{(a_1,a_2,a_3)}/J^{(a_1,a_2,a_3)}\to z_1^h
    \CC\brac*{\set{z_2/z_1}}\sqbrac*{z_1^{\pm1},z_2^{\pm2}}.
    \end{equation}
    The theorem then follows by applying \(\overline{\phi}_{{\mathcal{Y}_1,\mathcal{Y}_2}}\) to the
    relations \eqref{eq:genrels} of \cref{thm:gens}.\ref{itm:genrels}, using
    the \(L_{-1}\) derivative property of intertwining operators and
    expanding in the region \(|z_1|>|z_2|>0\).
  \end{proof}
  
  Systems of differential equations of the form \eqref{eq:diffeq} have
  solutions very close to the expansion required if their singular points are
  regular, see for example \cite[Appendix B]{KnaRep86}. A sufficient
  condition, whose validity we shall verify shortly,
  for regularity at a given singular point is that the coefficients \(a_i,
  b_j\) in the system \eqref{eq:diffeq} have poles of degree at most \(m-i\)
  and \(n-j\) respectively. Such singular points are called simple (see
  \cite[Appendix B]{KnaRep86} for the general definition). The singular points relevant for the convergence
  and extension property for products are \(z_1=z_2\) and \((z_1-z_2)/z_2=0\).
  
  We need to consider new filtrations in addition to those considered
  previously. Let \(\overline{R}=\CC\sqbrac{z_1^{\pm1},z_2^{\pm2}}\), then
  \(R_n=(z_1-z_2)^{-n}\overline{R},\ n\in\ZZ\) equips \(R\) with the structure of a
  filtered ring in the sense that \(R_n\subset R_m\), if \(n\le m\),
  \(R=\bigcup_{n\in \ZZ} R_n\) and \(R_n\cdot R_m\subset R_{m+n}\). The
  \(R\)-module \(T^{(a_1,a_2,a_3)}\) can then also be equipped with a compatible filtration
  \begin{equation}
    R_r(T^{(a_1,a_2,a_3)})=\prod_{\substack{n+h_0+h_1+h_2+h_3\le
        r\\h_i\in\RR}}R_n\otimes 
    \brac*{\Wmod_0^\prime}_{[h_0]}^{(a_1+a_2+a_3)}\otimes
    \brac*{\Wmod_1}_{[h_1]}^{(a_1)}\otimes \brac*{\Wmod_2}_{[h_2]}^{(a_2)}\otimes
    \brac*{\Wmod_3}_{[h_3]}^{(a_3)},\quad r\in\RR ,
  \end{equation}
  in the sense that \(R_r(T^{(a_1,a_2,a_3)})\subset R_s(T^{(a_1,a_2,a_3)})\),
  if \(r\le s\), \(T^{(a_1,a_2,a_3)}=\bigcup_{r\in\RR} R_r(T^{(a_1,a_2,a_3)})\)
  and \(R_n\cdot R_r(T^{(a_1,a_2,a_3)})\subset
  R_{n+r}(T^{(a_1,a_2,a_3)})\). Further, let
  \(R_r(J^{(a_1,a_2,a_3)})=R_r(T^{(a_1,a_2,a_3)})\cap J^{(a_1,a_2,a_3)}\).

  \begin{prop}\label{thm:singfiltbounds}
    Let the \(\VOA{V}\)-modules \(\Wmod_i,\ i=0,1,2,3\) be discretely strongly \(B\)-graded and
  \(B\)-graded \(C_1\)-cofinite as \(\overline{\VOA{V}}\)-modules. Then for
  any \(a_1,a_2,a_3\in B\) there exists \(M\in \ZZ\) such that for any \(r\in \RR\)
  \begin{equation}
    R_r(T^{(a_1,a_2,a_3)})\subset R_r(J^{(a_1,a_2,a_3)}) +
    F_M(T^{(a_1,a_2,a_3)})\quad
    \text{and}\quad
    T^{(a_1,a_2,a_3)}=J^{(a_1,a_2,a_3)}+F_M(T^{(a_1,a_2,a_3)}).
  \end{equation}
  Further, \(T^{(a_1,a_2,a_3)}/J^{(a_1,a_2,a_3)}\) is finitely generated.
\end{prop}
\begin{proof}
  The proof of this proposition mimics the proof of \cref{thm:filtbounds} once
  one has verified that the elements
  \(\mathcal{A}(u,w_0^\prime,w_1,w_2,w_3)\),
  \(\mathcal{B}(u,w_0^\prime,w_1,w_2,w_3)\),
  \(\mathcal{C}(u,w_0^\prime,w_1,w_2,w_3)\) and
  \(\mathcal{D}(u,w_0^\prime,w_1,w_2,w_3)\) lie in \(R_{h}(J)\), where \(h\)
  is the sum of the conformal weights of \(u,w_0^{\prime},w_1,w_2,w_3\).
\end{proof}

We also need to consider the \(\overline{R}\)-module \(U^{(a_1,a_2,a_3)}=\overline{R}\otimes
\brac*{\Wmod_0^\prime}^{(a_1+a_2+a_3)}\otimes \Wmod_1^{(a_1)}\otimes
\Wmod_2^{(a_2)}\otimes \Wmod_3^{(a_3)}\) and denote by \(U^{(a_1,a_2,a_3)}_{[r]}\) the subspace of
conformal weight \(r\in \RR\). Thus \(U^{(a_1,a_2,a_3)}=\prod_{r\in \RR} U_{[r]}^{(a_1,a_2,a_3)}\).

\begin{lem}\label{thm:polebounds}
  Let the \(\VOA{V}\)-modules \(\Wmod_i,\ i=0,1,2,3\) be discretely strongly \(B\)-graded and
  \(B\)-graded \(C_1\)-cofinite as \(\overline{\VOA{V}}\)-modules. For
  any \(a_1,a_2,a_3\in B\) and any doubly homogeneous vectors
  \(w_0^\prime \in \brac*{\Wmod_0^\prime}_{[h_0]}^{(a_1+a_2+a_3)}\), \(w_i\in
  \brac*{\Wmod_i}_{[h_i]}^{(a_i)}\), let \(h=\sum_i h_i\), let \(\overline{h}\)
  be the smallest non-negative representative of the coset \(h+\ZZ\) and let
  \(m_J\in R_h(J^{(a_1,a_2,a_3)})\), \(m_T\in F_M(T^{(a_1,a_2,a_3)})\) be
  vectors satisfying
  \begin{equation}
    w_0^\prime\otimes w_1\otimes w_2\otimes w_3=m_J+ m_T.
  \end{equation}
  Then there exists
  \(S\in \RR\) such that \(\overline{h}+S\in\ZZ_{\ge0}\) and
  \((z_1-z_2)^{h+S}m_T\in U^{(a_1,a_2,a_3)}\).
\end{lem}
\begin{proof}
  Note that the existence of the vectors \(m_J,\ m_T\) is guaranteed by \cref{thm:singfiltbounds}.
  Choose \(S\in \RR\) such that \(\overline{h}+S\in\ZZ_{\ge0}\) and such that
  for any \(r\le -S\), \(T_{[r]}^{(a_1,a_2,a_3)}=0\). Such an \(S\) must exist,
  since the conformal weights of \(T^{(a_1,a_2,a_3)}\) are bounded below by
  assumption.
  By definition, \(R_r(T^{(a_1,a_2,a_3)})\) is spanned by elements of the
  form
  \((z_1-z_2)^{-n}f(z_1,z_2)\widetilde{w}_0\otimes\widetilde{w}_1\otimes\widetilde{w}_2\otimes\widetilde{w}_3\),
  where \(f\in\overline{R}\) and \(n+\sum_i \wt \widetilde{w}_i \le r\). The
  number \(S\) was therefore chosen such that
  \((z_1-z_2)^{r+S}R_r(T^{(a_1,a_2,a_3)})\subset U^{(a_1,a_2,a_3)}\) whenever
  \(r+S\in\ZZ\). Now, by assumption,
  \begin{equation}
    m_T=w_0^\prime\otimes w_1\otimes w_2\otimes w_3-m_J.
  \end{equation}
  The \rhs{} of this equality lies in \(R_{h}(T^{(a_1,a_2,a_3)})\) by construction
  and therefore so does the \lhs{}. Hence \((z_1-z_2)^{h+S}m_T\in U^{(a_1,a_2,a_3)}\).
\end{proof}

\begin{thm}
  Let the \(\VOA{V}\)-modules \(\Wmod_i,\ i=0,1,2,3\) be discretely strongly \(B\)-graded and
  \(B\)-graded \(C_1\)-cofinite as \(\overline{\VOA{V}}\)-modules, let
  \(\Wmod_4\) be a \(B\)-graded \(\VOA{V}\)-module and let \(\mathcal{Y}_1\),
  \(\mathcal{Y}_2\) be logarithmic grading compatible intertwining operators
  of types \(\binom{\Wmod_0}{\Wmod_1,\ \Wmod_4}\), \(\binom{\Wmod_4}{\Wmod_2,\ \Wmod_3}\),
  respectively and consider the system of differential equations of
  \cref{thm:mateldiffs}. For the singular points \(z_1=z_2\) and
  \((z_1-z_2)/z_2=0\) 
  there exist coefficients
  \(a_k(z_1,z_2),\ b_l(z_1,z_2)\in R\) such that these singular points of the
  system of differential equations \eqref{eq:diffeq} satisfied by the matrix
  elements \eqref{eq:matels} are regular.
\end{thm}
\begin{proof}
  We consider first the singular point \(z_1=z_2\). By
  \cref{thm:singfiltbounds,thm:polebounds}, for any \(k\in\ZZ_{\ge0}\)
  together with a vector \(w_0^\prime\otimes L_{-1}^k w_1\otimes w_2\otimes w_3\in
  T^{(a_1,a_2,a_3)}\), where the \(w_i\) are doubly
  homogeneous vectors of total conformal weight \(h\in\RR\), there
  exist \(m_J^{(k)}\in R_{h+k}(J^{(a_1,a_2,a_3)})\) and \(m_T^{(k)}\in
  F_M(T^{(a_1,a_2,a_3)})\) such that
  \begin{equation}
    w_0^\prime\otimes L_{-1}^k w_1\otimes w_2\otimes w_3=m_J^{(k)}+m_T^{(k)} .
  \end{equation}
  Let \(\overline{h}\) be the smallest non-negative representative of the coset \(h+\ZZ\). Then, by \cref{thm:polebounds}, there exists \(S\in \RR\) such that \(\overline{h}+S\in \ZZ_{\ge0}\) and
  \((z_1-z_2)^{h+k+S}m_T^{(k)}\in U^{(a_1,a_2,a_3)}\) and thus
  \((z_1-z_2)^{h+k+S}m_T^{(k)}\in\bigcup_{r\le
    M}U^{(a_1,a_2,a_3)}_{[r]}\). Since the \(V\)-modules \(\Wmod_i\) are
  discretely strongly
  \(B\)-graded and \(B\)-graded \(C_1\)-cofinite, \(\prod_{r\le
    M}U^{(a_1,a_2,a_3)}_{[r]}\) is a finite sum of finitely generated
  \(\overline{R}\)-modules and hence also finitely generated. Thus, since
  \(\overline{R}\) is Noetherian, the submodule generated by the
  \((z_1-z_2)^{h+k+S}m_T^{(k)}\), \(k\in\ZZ_{\ge0}\) is also finitely
  generated. Hence there exists an \(m\in \ZZ_{\ge0}\) such that
  \(\set{(z_1-z_2)^{h+k+S}m_T^{(k)}\st 0\le k\le
  m-1}\) is a finite generating set for this submodule and subsequently there exist \(c_k(z_1,z_2)\in \overline{R}\) such that
  \begin{equation}
    (z_1-z_2)^{h+m+S}m_T^{(m)}+\sum_{k=0}^{m-1} c_k(z_1,z_2)(z_1-z_2)^{h+k+S}m_T^{(k)} = 0.
  \end{equation}
  Therefore,
  \begin{equation}
    w_0^\prime\otimes L_{-1}^m w_1\otimes w_2\otimes
    w_3+\sum_{k=0}^{m-1}c_k(z_1,z_2)(z_1-z_2)^{k-m}w_0^\prime\otimes L_{-1}^k w_1\otimes
    w_2\otimes w_3=m_J^{(m)}+\sum_{k=0}^{m-1}c_k(z_1,z_2) m_J^{(k)} .
  \end{equation}
  Thus in the quotient module \(T^{(a_1,a_2,a_3)}/J^{(a_1,a_2,a_3)}\), we
  obtain (where we again use square brackets to denote cosets)
  \begin{equation}
    \sqbrac*{w_0^\prime\otimes L_{-1}^m w_1\otimes w_2\otimes
      w_3}+\sum_{k=0}^{m-1}c_k(z_1,z_2)(z_1-z_2)^{k-m}\sqbrac*{w_0^\prime\otimes L_{-1}^k w_1\otimes
      w_2\otimes w_3}=0,
    \label{eq:regdiff1}
  \end{equation}
  since \(m_J^{(k)}\in J^{(a_1,a_2,a_3)}\). By a similar line of reasoning
  there exists an \(n\in\ZZ_{\ge0}\) and \(d_\ell(z_1,z_2)\in \overline{R}\)
    such that
      \begin{equation}
    \sqbrac*{w_0^\prime\otimes  w_1\otimes L_{-1}^n w_2\otimes
      w_3}+\sum_{\ell=0}^{m-1}d_\ell(z_1,z_2)(z_1-z_2)^{\ell-n}\sqbrac*{w_0^\prime\otimes  w_1\otimes L_{-1}^k
      w_2\otimes w_3}=0 .
        \label{eq:regdiff2}
  \end{equation}
  Applying the map \(\phi_{\mathcal{Y}_1,\mathcal{Y}_2}\) defined by \eqref{eq:matelmap} and using the
  \(L_{-1}\) property for intertwining operators will then result in a system
  of differential equations for which \(z_1=z_2\) is a simple, and hence regular, singular point.

  To show the regularity of the singular point \((z_1-z_2)/z_2=0\), we
  introduce new gradings on \(R\) and \(T^{(a_1,a_2,a_3)}\). We assign degree
  \(-1\) to the variables \(z_1,z_2\), thus giving \(R\) a \(\ZZ\) grading and
  then grade \(T^{(a_1,a_2,a_3)}\) by adding \(R\)-degrees and conformal
  weights. This implies that the elements
  \(\mathcal{A}(v,w_0^\prime,w_1,w_2,w_3), \ 
  \mathcal{B}(v,w_0^\prime,w_1,w_2,w_3), \ 
  \mathcal{C}(v,w_0^\prime,w_1,w_2,w_3)\) and
  \(\mathcal{D}(v,w_0^\prime,w_1,w_2,w_3)\) are homogeneous with respect to
  this new grading if their arguments
  are doubly homogeneous. The new grading therefore descends to
  \(T^{(a_1,a_2,a_3)}/J^{(a_1,a_2,a_3)}\). Further, 
  for doubly homogeneous elements \(w_0^\prime,w_1,w_2,w_3\), the elements
  \begin{equation}
    \sqbrac*{w_0^\prime\otimes L_{-1}^k w_1\otimes w_2\otimes w_3},\quad
    \sqbrac*{w_0^\prime\otimes w_1\otimes L_{-1}^\ell w_2\otimes w_3}\in
    T^{(a_1,a_2,a_3)}/J^{(a_1,a_2,a_3)} ,
  \end{equation}
  are also homogeneous. Thus the coefficients \(c_k(z_1,z_2),
  d_\ell(z_1,z_2)\) of equations \eqref{eq:regdiff1} and \eqref{eq:regdiff2}
  are elements of degree 0 in \(R\) and can therefore be written as
  Laurent polynomials in \((z_1-z_2)/z_2\). It then follows that the singular
  point \((z_1-z_2)/z_2=0\) is regular.
\end{proof}

The fact that the matrix element \eqref{eq:prodmatel} satisfies an expansion
of the form \eqref{eq:prodconvext} now follows by the reasoning of \cite[Theorem
3.5]{HuaDif04}. A little care is needed when following the reasoning
  of \cite{HuaDif04}, since there only modules with a diagonalisable action of
\(L_0\) are considered. However, as noted in \cite[Part VII, Proof of Theorem
11.8 and Remark 11.9]{HuaLog} the argument extends easily to modules where
\(L_0\) has Jordan blocks. The basic idea is that one can use the \(L_0\)
conjugation property of intertwining operators (recall that \(L_0\) is the
generator of dilations) to rescale the variables in the matrix element 
\eqref{eq:prodmatel} by \(z_2\) so that it becomes a function in
\(z_3=(z_1-z_2)/z_2\) only and the system of differential equations
\eqref{eq:regdiff1} and \eqref{eq:regdiff2} then becomes an ordinary
differential equation for \(z_3\) with a regular singularity at
\(z_3=0\). Similar reasoning for the matrix element \eqref{eq:itmatel} leads one to conclude
that it satisfies the expansion \eqref{eq:itconvext}. Hence \cref{thm:gradconext} follows.

\flushleft

\begin{thebibliography}{10}
 	
 	\bibitem{HuaVer08}
 	Y-Z Huang.
 	\newblock Vertex operator algebras and the {Verlinde} conjecture.
 	\newblock {\em Commun. Contemp. Math.}, 10:103--1054, 2008.
 	\newblock \textsf{arXiv:\mbox{math}/0406291}.
 	
 	\bibitem{FuclHo13}
 	J~Fuchs, C~Schweigert, and C~Stigner.
 	\newblock From non-semisimple {Hopf} algebras to correlation functions for
 	logarithmic {CFT}.
 	\newblock {\em J. Phys.}, A46:494008, 13.
 	\newblock \textsf{arXiv:1302.4683 [\mbox{hep-th}]}.
 	
 	\bibitem{FuclCFT17}
 	J~Fuchs and C~Schweigert.
 	\newblock Consistent systems of correlators in non-semisimple conformal field
 	theory.
 	\newblock {\em Adv. Math.}, 307:598--639, 2017.
 	\newblock \textsf{arXiv:1604.01143 [\mbox{math.QA}]}.
 	
 	\bibitem{CrelMod17}
 	T~Gannon T~Creutzig.
 	\newblock Logarithmic conformal field theory, log-modular tensor categories and
 	modular forms.
 	\newblock {\em J. Phys.}, A50:404004, 2017.
 	\newblock \textsf{arXiv:1605.04630 [\mbox{math.QA}]}.
 	
 	\bibitem{FriCon86}
 	D~Friedan, E~Martinec, and S~Shenker.
 	\newblock Conformal invariance, supersymmetry and string theory.
 	\newblock {\em Nucl. Phys.}, B271:93--165, 1986.
 	
 	\bibitem{WakFoc86}
 	M~Wakimoto.
 	\newblock {Fock} representation of the algebra {$A_1^{(1)}$}.
 	\newblock {\em Comm. Math. Phys.}, 104:605--609, 1986.
 	
 	\bibitem{FeiDSred90}
 	B~Feigin and E~Frenkel.
 	\newblock Quantization of the {Drinfeld-Sokolov} reduction.
 	\newblock {\em Phys. Lett.}, 246:75--81, 1990.
 	
 	\bibitem{MalChi99}
 	F~Malikov, V~Schechtman, and A~Vaintrob.
 	\newblock Chiral de {Rham} complex.
 	\newblock {\em Comm. Math. Phys.}, 204:439--473, 1999.
 	\newblock \textsf{arXiv:\mbox{math}/9803041}.
 	
 	\bibitem{KauLog95}
 	H~Kausch.
 	\newblock Curiosities at $c = -2$.
 	\newblock {\em DAMTP}, 95--52:26, 1995.
 	\newblock \textsf{arXiv:\mbox{hep-th}/9510149}.
 	
 	\bibitem{GabInd96}
 	M~Gaberdiel and H~Kausch.
 	\newblock A rational logarithmic conformal field theory.
 	\newblock {\em Phys. Lett.}, B386:131--137, 1996.
 	\newblock \textsf{arXiv:\mbox{hep-th}/9606050}.
 	
 	\bibitem{GabBdyBlk07}
 	M~Gaberdiel and I~Runkel.
 	\newblock From boundary to bulk in logarithmic {CFT}.
 	\newblock {\em J. Phys.}, A41:075402, 2008.
 	\newblock \textsf{arXiv:0707.0388 [\mbox{hep-th}]}.
 	
 	\bibitem{RunSBos14}
 	I~Runkel.
 	\newblock A braided monoidal category for free super-bosons.
 	\newblock {\em J. Math. Phys.}, 55:59, 2014.
 	\newblock \textsf{arXiv:1209.5554 [\mbox{math.QA}]}.
 	
 	\bibitem{AdaTrip08}
 	D~Adamović and A~Milas.
 	\newblock On the triplet vertex algebra $\mathcal{W}(p)$.
 	\newblock {\em Adv. Math.}, 217:2664--–2699, 2008.
 	\newblock \textsf{arXiv:0707.1857 [\mbox{math.QA}]}.
 	
 	\bibitem{TsuTrip15}
 	A~Tsuchiya and S~Wood.
 	\newblock On the extended {$W$}-algebra of type $\mathfrak{sl}_2$ at positive
 	rational level.
 	\newblock {\em Int. Math. Res. Not.}, 2015:5357--–5435, 2015.
 	\newblock \textsf{arXiv:1302.6435 [\mbox{math.QA}]}.
 	
 	\bibitem{TsuRig13}
 	A~Tsuchiya and S~Wood.
 	\newblock The tensor structure on the representation category of the
 	$\mathcal{W}(p)$ triplet algebra.
 	\newblock {\em J. Phys.}, A46:445203, 2013.
 	\newblock \textsf{arXiv:1201.0419 [\mbox{hep-th}]}.
 	
 	\bibitem{RidBos14}
 	D~Ridout and S~Wood.
 	\newblock Bosonic ghosts at $c=2$ as a logarithmic {CFT}.
 	\newblock {\em Lett. Math. Phys.}, 105:279--307, 2015.
 	\newblock \textsf{arXiv:1408.4185 [\mbox{hep-th}]}.
 	
 	\bibitem{RidSL208}
 	D~Ridout.
 	\newblock $\widehat{\mathfrak{sl}}(2)_{-1/2}$: {A} case study.
 	\newblock {\em Nucl. Phys.}, B814:485--521, 2009.
 	\newblock \textsf{arXiv:0810.3532 [\mbox{hep-th}]}.
 	
 	\bibitem{CreMod12}
 	T~Creutzig and D~Ridout.
 	\newblock Modular data and {Verlinde} formulae for fractional level {WZW}
 	models {I}.
 	\newblock {\em Nucl. Phys.}, B865:83--114, 2012.
 	\newblock \textsf{arXiv:1205.6513 [\mbox{hep-th}]}.
 	
 	\bibitem{CreWZW13}
 	T~Creutzig and D~Ridout.
 	\newblock Modular data and {Verlinde} formulae for fractional level {WZW}
 	models {II}.
 	\newblock {\em Nucl. Phys.}, B875:423--458, 2013.
 	\newblock \textsf{arXiv:1306.4388 [\mbox{hep-th}]}.
 	
 	\bibitem{AdaBG19}
 	D~Adamović and V~Pedić.
 	\newblock On fusion rules and intertwining operators for the {Weyl} vertex
 	algebra.
 	\newblock {\em J. Math. Phys.}, 60:081701, 2019.
 	
 	\bibitem{HuaLog}
 	Y-Z Huang, J~Lepowsky James, and L~Zhang.
 	\newblock Logarithmic tensor product theory {I--VIII}.
 	\newblock \textsf{arXiv:1012.4193 [\mbox{math.QA}], arXiv:1012.4196
 		[\mbox{math.QA}], arXiv:1012.4197 [\mbox{math.QA}], arXiv:1012.4198
 		[\mbox{math.QA}], arXiv:1012.4199 [\mbox{math.QA}] , arXiv:1012.4202
 		[\mbox{math.QA}], arXiv:1110.1929 [\mbox{math.QA}], arXiv:1110.1931
 		[\mbox{math.QA}]}.
 	
 	\bibitem{Beem4D2D15}
 	C~Beem, M~Lemos, P~Liendo, W~Peelaers, L~Rastelli, and B~van Rees.
 	\newblock Infinite chiral symmetry in four dimensions.
 	\newblock {\em Comm. Math. Phys.}, 336:1369--1433, 2015.
 	\newblock \textsf{arXiv:1312.5344 [\mbox{hep-th}]}.
 	
 	\bibitem{CreuBp14}
 	T~Creutzig, D~Ridout, and S~Wood.
 	\newblock Coset constructions of logarithmic (1,p)-models.
 	\newblock {\em Lett. Math. Phys.}, 104:553--583, 2014.
 	\newblock \textsf{arXiv:1305.2665 [\mbox{math.QA}]}.
 	
 	\bibitem{AugBp19}
 	J~Auger, T~Creutzig, S~Kanade, and M~Rupert.
 	\newblock Braided tensor categories related to $\mathcal{B}_p$ vertex algebras.
 	\newblock {\em Comm. Math. Phys.}, 2020.
 	
 	\bibitem{Yang18}
 	J~Yang.
 	\newblock A sufficient condition for convergence and extension property for
 	strongly graded vertex algebras.
 	\newblock {\em Contemporary Mathematics}, 711:119--141, 2018.
 	
 	\bibitem{BloIrr79}
 	R~Block.
 	\newblock The irreducible representations of the {Weyl} algebra {$A_{1}$}.
 	\newblock {\em Lecture Notes in Mathematics}, 740:69--79, 1979.
 	
 	\bibitem{LiSup97}
 	H~Li.
 	\newblock The physics superselection principle in vertex operator algebra
 	theory.
 	\newblock {\em J. Algebra}, 196(2):436--457, 1997.
 	
 	\bibitem{FreVer01}
 	E~Frenkel and D~{Ben-Zvi}.
 	\newblock {\em Vertex Algebras and Algebraic Curves}, volume~88 of {\em
 		Mathematical Surveys and Monographs}.
 	\newblock American Mathematical Society, 2001.
 	
 	\bibitem{DoLGVA93}
 	C~Dong and J~Lepowsky.
 	\newblock {\em Generalized {{Vertex Algebras}} and {{Relative Vertex
 				Operators}}}.
 	\newblock Progress in {{Mathematics}}. {Birkh{\"a}user}, {Boston}, 1993.
 	
 	\bibitem{WoOsp19}
 	S~Wood.
 	\newblock Admissible level $\mathfrak{osp}(1|2)$ minimal models and their
 	relaxed highest weight modules.
 	\newblock {\em Transf. Groups}, 2020:57, 2020.
 	\newblock \textsf{arXiv:1804.01200 [\mbox{math.QA}]}.
 	
 	\bibitem{HilHom97}
 	P~Hilton and U~Stammbach.
 	\newblock {\em A {{Course}} in {{Homological Algebra}}}.
 	\newblock Graduate {{Texts}} in {{Mathematics}}. Springer-Verlag, 2 edition,
 	1997.
 	
 	\bibitem{FjeLog02}
 	J~Fjelstad, J~Fuchs, S~Hwang, A~M Semikhatov, and I~Yu Tipunin.
 	\newblock Logarithmic conformal field theories via logarithmic deformations.
 	\newblock {\em Nucl. Phys.}, 633(3):379--413, 2002.
 	\newblock \textsf{arXiv:\mbox{hep-th}/0201091}.
 	
 	\bibitem{BelTem18}
 	J~Bellet{\^e}te, D~Ridout, and Y~{Saint-Aubin}.
 	\newblock Restriction and induction of indecomposable modules over the
 	{{Temperley}}\textendash{{Lieb}} algebras.
 	\newblock {\em J. Phys.}, 51(4):045201, 2017.
 	\newblock \textsf{arXiv:1605.05159 [\mbox{math-ph}]}.
 	
 	\bibitem{EtiTen15}
 	P~Etingof, G~Shlomo, D~Nikshych, and V~Ostrik.
 	\newblock {\em Tensor Categories}.
 	\newblock Number volume 205 in Mathematical Surveys and Monographs. American
 	Mathematical Society, 2015.
 	
 	\bibitem{HuaApp17}
 	Y-Z Huang.
 	\newblock On the applicability of logarithmic tensor category theory, 2017.
 	\newblock \textsf{arXiv:1702.00133 [\mbox{math.QA}]}.
 	
 	\bibitem{KanRid18}
 	S~Kanade and D~Ridout.
 	\newblock {NGK} and {HLZ}: Fusion for physicists and mathematicians.
 	\newblock In D~Adamovic and P~Papi, editors, {\em Affine, Vertex and
 		{\(W\)}-algebras}, volume~37 of {\em Springer INdAM}, pages 135--181, Cham,
 	2019. Springer.
 	\newblock \textsf{arXiv:1812.10713 [\mbox{math-ph}]}.
 	
 	\bibitem{ZhaHLZgen08}
 	L~Zhang.
 	\newblock Vertex tensor category structure on a category of {Kazhdan-Lusztig}.
 	\newblock {\em New York J. Math.}, 14:261--284, 2008.
 	\newblock \textsf{arXiv:\mbox{math}/0701260}.
 	
 	\bibitem{Wan98}
 	W~Wang.
 	\newblock {$\mathcal{W}_{1+\infty}$} algebra, {$\mathcal{W}_3$} algebra, and
 	{Friedan–Martinec– Shenker} bosonization.
 	\newblock {\em Comm. Math. Phys.}, 195:95 -- 111, 1998.
 	
 	\bibitem{Lin09}
 	A~Linshaw.
 	\newblock Invariant chiral differential operators and the $\mathcal{W}_3$
 	algebra.
 	\newblock {\em J. Pure Appl. Algebra}, 213:632 -- 648, 2009.
 	
 	\bibitem{AdaW01}
 	D~Adamović.
 	\newblock Representations of the vertex algebra {$\mathcal{W}_{1+\infty}$} with
 	a negative integer central charge.
 	\newblock {\em Comm. Algebra}, 29(7):3153--3166, 2001.
 	\newblock \textsf{arXiv:\mbox{math}/9904057}.
 	
 	\bibitem{Kac95}
 	V~Kac and A~Radul.
 	\newblock Representation theory of the vertex algebra {$W_{1 + \infty}$}, 1996.
 	\newblock \textsf{arXiv:\mbox{hep-th}/9512150}.
 	
 	\bibitem{Mat94}
 	Y~Matsuo.
 	\newblock Free fields and quasi-finite representation of {$W_{1+\infty}$}
 	algebra.
 	\newblock {\em Phys. Lett.}, 326(1--2):95–--100, 1994.
 	\newblock \textsf{arXiv:\mbox{hep-th}/9312192}.
 	
 	\bibitem{CreFal13}
 	T~Creutzig and A~Milas.
 	\newblock False theta functions and the verlinde formula.
 	\newblock {\em Adv. Math.}, 262:520 -- 545, 2014.
 	\newblock \textsf{arXiv:1309.6037 [\mbox{math.QA}]}.
 	
 	\bibitem{Ada07}
 	D~Adamović and A~Milas.
 	\newblock Logarithmic intertwining operators and {$W(2,2p-1)$} algebras.
 	\newblock {\em J. Math. Phys.}, 48(7):073503, 2007.
 	\newblock \textsf{arXiv:\mbox{math}/0702081 [\mbox{math.QA}]}.
 	
 	\bibitem{CreLog16}
 	T~Creutzig, A~Milas, and M~Rupert.
 	\newblock Logarithmic link invariants of
 	$\overline{U}_q^h\brac{\mathfrak{sl}_2)}$ and asymptotic dimensions of
 	singlet vertex algebras.
 	\newblock {\em J. Pure Appl. Algebra}, pages 3224 -- 3247, 2017.
 	\newblock \textsf{arXiv:1605.05634 [\mbox{math.QA}]}.
 	
 	\bibitem{HuaDif04}
 	Y-Z Huang.
 	\newblock Differential equations and intertwining operators.
 	\newblock {\em Commun. Contemp. Math.}, 7:375--400, 2005.
 	\newblock \textsf{arXiv:\mbox{math}/0206206}.
 	
 	\bibitem{KnaRep86}
 	A~Knapp.
 	\newblock {\em Representation theory of semisimple groups: an overview based on
 		examples}.
 	\newblock Princeton Landmarks in Mathematics. Princeton University Press, 1986.
 	
 \end{thebibliography}

\end{document}